\documentclass[12pt,preprint,number]{elsarticle}
\usepackage{amsmath,amssymb} \usepackage{euler,eucal}
\usepackage{ifxetex}
\ifxetex
 \usepackage{fontspec} \usepackage{polyglossia}
 \defaultfontfeatures{Mapping=tex-text}
 \setromanfont{URW Palladio L}
\else
 \usepackage[utf8]{inputenc}
 \usepackage[T1]{fontenc}
 \usepackage[english]{babel}
 \usepackage{tgpagella}
\fi
\usepackage{dsfont}
\usepackage{algorithmic} 
\usepackage[ruled]{algorithm} 
\usepackage[unicode=true,plainpages=false]{hyperref} 
\usepackage{graphicx}
\hypersetup{colorlinks=false,pdfborder={0 0 0.1},linkcolor=magenta,anchorcolor=magenta,urlcolor=blue,citecolor=blue}
\newcommand*{\doi}[1]{doi: \href{http://dx.doi.org/#1}{\texttt{#1}}}
\usepackage[a4paper]{geometry}
\usepackage{tikz} \usetikzlibrary{shapes,snakes}
\usepackage{gnuplot-lua-tikz}
\def\trans{T}
\def\t{\trans}

\def\eps{\varepsilon}
\def\phi{\varphi}

\def\R{\mathbb{R}}

\def\T{\mathcal{T}}
\def\O{\mathcal{O}}
\def\A{\mathbf{A}}
\def\C{\mathbb{C}}
\def\N{\mathbb{N}}
\def\h{\gamma}
\def\L{\mathop{\mathcal{L}}\nolimits}

\def\x{\mathbin{\times}}
\def\o{\mathbin{\raise1pt\hbox{$\scriptscriptstyle\mathord\otimes$}}}
\def\had{\mathbin{\raise1pt\hbox{$\scriptstyle\mathord\odot$}}}

\def\conv{\mathbin{\star}}

\def\diag{\mathop{\mathrm{diag}}\nolimits}

\def\eqdef{\mathrel{\stackrel{\mathrm{def}}{=}}}

\newtheorem{remark}{Remark}
\newtheorem{definition}{Definition}

\newtheorem{theorem}{Theorem}[section]
\newtheorem{lemma}[theorem]{Lemma}
\newproof{proof}{Proof}

\begin{document}
\begin{frontmatter}
\title{Superfast solution of linear convolutional Volterra equations using QTT approximation}
\author[uoc]{Jason A. Roberts}                     \ead{j.roberts@chester.ac.uk}
\author[inm,uoc]{Dmitry V. Savostyanov\fnref{lhm}} \ead{dmitry.savostyanov@gmail.com}
\author[inm]{Eugene E. Tyrtyshnikov\fnref{lhm}}    \ead{eugene.tyrtyshnikov@gmail.com}

\address[uoc]{University of Chester, Parkgate Road, Chester, CH1 4BJ, UK}
\address[inm]{Institute of Numerical Mathematics, Russian Academy of Sciences, Gubkina 8, Moscow, 119333, Russia}

\fntext[lhm]{During this work D. V. Savostyanov and E. E. Tyrtyshnikov were supported by the Leverhulme Trust to visit, stay and work at the University of Chester, as the Visiting Research Fellow and the Visiting Professor, respectively.
Their work was also supported in part by RFBR grant 11-01-00549 and Russian Federation Government Contracts $\Pi$1112, 14.740.11.0345, 14.740.11.1067}
\date{November 22, 2012}

\begin{abstract}
We address a linear fractional differential equation and develop  effective solution methods using algorithms for inversion of triangular Toeplitz matrices and the recently proposed QTT format.
The inverses of such matrices can be computed by the divide and conquer and modified Bini's algorithms, for which we present the versions with the QTT approximation.
We also present an efficient formula for the shift of vectors given in QTT format, which is used in the divide and conquer algorithm.
As the result, we reduce the complexity of inversion from the fast Fourier level $\O(n\log n)$ to the speed of superfast Fourier transform, i.e., $\O(\log^2 n).$ 
The results of the paper are illustrated by numerical examples. 
\end{abstract}
\end{frontmatter}

\section{Introduction}
Equations involving derivatives of fractional order are of great importance, due to their role in mathematical models applied in mechanics, biochemistry, electrical engineering, medicine, etc., see~\cite{diethelm-1997,CapMain1,FreedDiet1}.
In this paper we present a superfast algorithm for the numerial solution of the linear equation 
\begin{equation}\label{eq1}
D^{\alpha}_*y(t) = F(t,y(t)) = my(t)+f(t), \quad 0\leq t \leq T, \quad y(0)=y_0, 
\end{equation}
where $0<\alpha<1$ is the order of the fractional operator, $m\in\R$ is a constant referred to as \emph{mass}, and $f(t)$ is a sufficiently well--behaved \emph{forcing} term.
For $\alpha=1/2$ this equation is a scalar version of the Bagley-Torvik equation~\cite{BT1}, which is used in the modelling of viscoelastic materials. 
The definitions of Caputo derivative $D^\alpha_*$ can be found in many sources, e.g.~\cite{Diet1,Pod1}, and are presented in appendix for the convenience.

The classical result of Diethelm~\cite[Lem. 6.2]{Diet1} allows us to rewrite~\eqref{eq1} in the form
\begin{equation}\label{eq2}
 y(t) = y_0 + \frac{1}{\Gamma(\alpha)} \int_0^t  (t-s)^{\alpha-1} \left( my(s)+f(s) \right) ds,
\end{equation}
where $\Gamma(\alpha) = \int_0^\infty e^{-t}t^{\alpha-1} dt$ is the gamma function.
Eq.~\eqref{eq2} is the weakly singular convolutional Volterra equation of the second kind with the Abel--type kernel. 
Volterra equations of second kind are well-studied and are proven to have a unique continuous solution for $0\leq t \leq T,$ see, e.g.~\cite[Thm. 3.2]{linz-1985}.
The solution is asymptotically stable if $m<0$ (see~\cite{GorRut1}) which we will always assume in this paper.

For certain forcing terms, the solution of~\eqref{eq2} can be found using series methods.
In a general framework,  we can discretize~\eqref{eq2} using a collocation or Galerkin method and numerically solve the resulted linear system.
This matrix approach to fractional calculus was brilliantly presented by I. Podlubny in~\cite{podlubny-matr-2000}.
In this paper we consider the collocation method and assume that $y(t)$ is approximated by a piecewise--linear function on a uniform grid $t_j=jh,$ $j=0,\ldots,n,$ where $h=T/n.$
The stability of collocation methods for fractional equations was studied in~\cite{blanck-1995,blanck-1996} and an error analysis can be found in~\cite{DietFordFreed1}.
The discretized equation is the following
$$
y_j = y_0 + \frac{h^\alpha}{\Gamma(\alpha)} \sum_{k=0}^j w_{j,k} (m y_k + f_k), \qquad j=1,\ldots,n,
$$
where $y_j=y(t_j),$ $f_k=f(t_k)$ and $w_{j,k}$ are quadrature weights, defined by integration of piecewise--linear basis functions with Abel-type kernel, i.e.,
\begin{equation}\nonumber
w_{j,k} = \frac{1}{\alpha (\alpha+1)} 
  \left\{ \begin{array}{lc}
    (j-1)^{\alpha+1} - (j-\alpha-1)j^\alpha,                     & k=0, \\
    (j-k-1)^{\alpha+1} - 2 (j-k)^{\alpha+1} + (j-k+1)^{\alpha+1}, & 1 \leq k < j, \\
    1, & k=j.
  \end{array} \right.
\end{equation}
Finally, we obtain the linear system $A y = b$ with triangular Toeplitz matrix and the right-hand side defined as follows,
\begin{equation}\label{eq3}
 \sum_{k=1}^j a_{j-k} y_k = b_j, \qquad j=1,\ldots,n,
\end{equation}
\begin{equation}\nonumber
 \begin{split}
 a_p & = 
  \left\{ \begin{array}{lc}
    1 - \h m, & p=0, \\
      - \h m \left( (p-1)^{\alpha+1} -2p^{\alpha+1} + (p+1)^{\alpha+1}  \right),& p>0,
  \end{array} \right.
  \\
 b_j & = y_0 + \h \left( \sum_{k=1}^j w_{j,k} f_k + w_{j,0} (m y_0 + f_0) \right),
 \end{split}
\end{equation}
where $\h=h^\alpha / \Gamma(\alpha+2).$

The numerical scheme we use is analogous to the fractional Adams method proposed in~\cite{DietFordFreed1}  for a general (e.g. nonlinear) function $F(t,y(t)).$
The method is developed as a generalization of the Adams--Bashforth--Moulton scheme from the classical numerical analysis of ordinary differential equations and a detailed error analysis is provided.
The complexity of the fractional Adams method in the nonlinear case is $\O(n^2).$
To reduce this complexity, we can take into account the decay speed of the Abel kernel~$k(s)=s^{\alpha-1}$ of the integral in~\eqref{eq2}.
The so-called~\emph{fixed memory principle}~\cite{Pod1,Pod2} and more accurate~\emph{nested mesh method}~\cite{FordSimp1,DietFordFreedLuchko1} are based on truncation and approximation of the tail of the integral~\eqref{eq2}, respectively, and have almost linear complexity w.r.t. $n.$
We revise these methods in Sec.~\ref{LOG}.

For linear $F(t,y(t)),$ the problem writes as the linear system~\eqref{eq3}, which can be solved using well--developed algorithms for the inversion of triangular Toeplitz matrices, or triangular strip matrices, as they are referred in~\cite{podlubny-matr-2000}. 
These methods are recalled in Sec.~\ref{TTOEP}, and have the asymptotic complexity of the fast Fourier transform (FFT) algorithm, which is $\O(n \log n).$

Recently, a superfast Fourier transform algorithm was proposed in~\cite{dks-ttfft-2012}, based on the approximation of vectors in the \emph{quantized tensor train} (QTT) format~\cite{osel-tt-2011,khor-qtt-2011}.
The method can be considered as a classical model of quantum superfast Fourier transform algorithm~\cite{EkertJozsa-1998}, and has a square-logarithmical complexity~$\O(\log^2 n)$ for a certain class of vectors, for which such a model is efficient.
This class of vectors is partially established in~\cite{sav-rank1-2012} and include, for example, vectors with sparse Fourier image.
The numerical experiments provided in Sec.~\ref{TTINV} show that the Abel kernel $t(s) = s^{1-\alpha}$ is efficiently approximated by the QTT format for all $0<\alpha <1$ with accuracy up to the machine threshold.
Based on this observation, we propose the superfast inversion algorithm for the triangular Toeplitz matrix~\eqref{eq3}, using the QTT approximation.

The numerical experiments provided in Sec.~\ref{NUM} justify the accuracy and sublinear complexity of the method proposed.

\section{Numerical method with logarithmic memory} \label{LOG}
In \cite{Pod1} and \cite{Pod2} the author describes an approach to the numerical integration involved in solving a fractional problem whereby the first part (or tail) of the integral is ignored (i.e. assuming the value of the integral over this region is negligible) and so the memory of the system is truncated at some point. The error introduced via this process is described in \cite{Pod1} for Riemann-Liouville fractional derivatives. In \cite{FordSimp1} the authors consider the error that is introduced when this approach is applied to problems expressed with respect to the Caputo fractional derivative. The authors show that by introducing a finite memory of fixed length $T$ for the Caputo derivative we introduce an error of the form
\begin{equation} \label{fixerror}
E=\left| \frac{1}{\Gamma\left( 1-\alpha\right)}\int^{t-T}_0\frac{y'(s)}{\left( t-s\right)^\alpha}ds\right|.
\end{equation}
Letting $\sup_{s\in[0,t]}|y'(s)|=M$ then
B
\begin{equation}
\label{fe2}
E\leq\frac{M\left( t^{1-\alpha}-T^{1-\alpha}\right)}{\Gamma\left( 2-\alpha\right)}.
\end{equation}
So for a fixed memory $T<t$ we have a loss of order such that the error does not tend towards zero as the stepsize approaches zero. Indeed, the authors in \cite{FordSimp1} highlight that in order to preserve the order of the method we would need to choose $T$ so that (for a fixed error bound $E$) we have
\begin{equation}
\label{fe3}
T^{1-\alpha}\geq t^{1-\alpha}-\left(\frac{E\Gamma(2-\alpha)}{M}\right),
\end{equation}
which introduces a computational cost --- precisely what the fixed memory principle is trying to avoid.
To overcome this it is proposed in \cite{FordSimp1}, and described further in \cite{DietFordFreedLuchko1}, that the fixed memory principle is amended so that the region of integration $[0,t]$ is decomposed into a sequence of finite-length intervals with differing stepsizes. 
So as we move `backwards' along the interval from $t$ to $0$ the subintervals use coarser and coarser stepsizes, except possibly for some small sub-interval near zero due to the length of this subinterval not being an exact multiple of the current stepsize --- in such circumstances the authors suggest a couple of alternative approaches for this subinterval, one such alternative being the use of the original stepsize.
Such a \emph{nested mesh} approach (see actual mesh on Fig.~\ref{fig:nest}) is possible due to the scaling properties of the fractional integral, which are discussed in \cite{DietFordFreedLuchko1} and \cite{FordSimp1}. Thus the weights (of the Adams--type method described earlier) for calculating $\Omega^\alpha_hf(nh)\approx I^\alpha f(nh)$ with a stepsize $h$ can be used to calculate $\Omega^\alpha_{\omega^ph}f\left(\omega^pnh\right)\approx I^\alpha f\left(n\omega^ph\right)$ using a stepsize of $\omega^ph$.
The authors \cite{FordSimp1} define, for $h\in\R^+$, the mesh $M_h$ by $M_h=\left\{ hn,n\in\N\right\}$. If $\omega ,r,p\in\N$, $\omega>0$, $r>p$, then $M_{\omega^rh}\subset M_{\omega^ph}$. The authors then decompose the interval $[0,t]$, for fixed $T>0$ in the following way:
\begin{equation}
\label{mesh1}
[0,t]=[0,t-\omega ^mT]\cup [t-\omega ^mT, t-\omega^{m-1}T]\cup\cdots\cup [t-\omega T, t-T]\cup [t-T,t]
\end{equation}
where $m\in\N$ is the smallest integer such that $t<\omega^{m+1}T$. A step length of $h$ is used over the most recent time interval $[t-T,t]$ with successively larger step sizes over earlier intervals, as follows.
Let $t,T,h\in\R$, $\omega^{m+1}T>t\geq \omega^mT$, $t>1$, $h>0$ with $t=nh$ for some $n\in\N$. The integral can be rewritten as
\begin{eqnarray}
\label{mesh2}
i^\alpha_{[0,t]}f(t)&=&I^\alpha_{[t-T,t]}f(t)+\sum^{m-1}_{i=0}I^\alpha_{\left[ t-\omega^{i+1}T,t-\omega^iT\right]}f(t)+I^\alpha_{[0,t-\omega^mT]}f(t)\\
&=&I^\alpha_{[t-T,t]}f(t)+\sum^{m-1}_{i=0}\omega^{i\alpha}I^\alpha_{[t-\omega T,t-T]}f\left(\omega^it\right)+\omega^{m\alpha}I^\alpha_{[0,t-\omega^mT]}f\left(\omega^mt\right) ,
\end{eqnarray}
where $I^\alpha_{[t-a,t-b]}f(t)=\frac{1}{\Gamma(\alpha)}\int^{t-b}_{t-a}\frac{f(s)}{(t-s)^{1-\alpha}}ds$.
The authors also show the following:
\begin{theorem}\cite{FordSimp1}
The nested mesh scheme preserves the order of the underlying rule on which it is based.
\end{theorem}
In addition, whilst the computational cost of the full-memory approach is of order $O(N^2)$, the nested-mesh approach has order $O(N\log N)$.
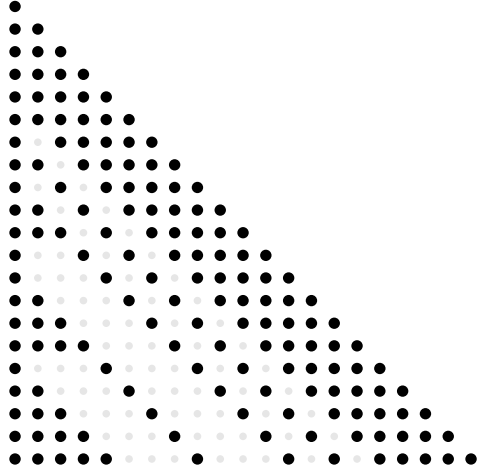
\begin{figure}[t]
 \begin{center}
    \def\Rb{1mm} \tikzstyle{b}=[circle,fill=black   ,inner sep=0mm,minimum size=1.5mm]
  \def\Rw{1mm} \tikzstyle{w}=[circle,fill=black!10,inner sep=0mm,minimum size=1.0mm]
  \begin{tikzpicture}[x=3mm,y=-3mm]
    \node[b] at (0,0) {};

    \node[b] at (0,1) {};
    \node[b] at (1,1) {};

    \node[b] at (0,2) {};
    \node[b] at (1,2) {};
    \node[b] at (2,2) {};
    
    \node[b] at (0,3) {};
    \node[b] at (1,3) {};
    \node[b] at (2,3) {};
    \node[b] at (3,3) {};
    
    \node[b] at (0,4) {};
    \node[b] at (1,4) {};
    \node[b] at (2,4) {};
    \node[b] at (3,4) {};
    \node[b] at (4,4) {};
    
    \node[b] at (0,5) {};
    \node[b] at (1,5) {};
    \node[b] at (2,5) {};
    \node[b] at (3,5) {};
    \node[b] at (4,5) {};
    \node[b] at (5,5) {};
    
    \node[b] at (0,6) {};
    \node[w] at (1,6) {};
    \node[b] at (2,6) {};
    \node[b] at (3,6) {};
    \node[b] at (4,6) {};
    \node[b] at (5,6) {};
    \node[b] at (6,6) {};
    
    \node[b] at (0,7) {};
    \node[b] at (1,7) {};
    \node[w] at (2,7) {};
    \node[b] at (3,7) {};
    \node[b] at (4,7) {};
    \node[b] at (5,7) {};
    \node[b] at (6,7) {};
    \node[b] at (7,7) {};
    
    \node[b] at (0,8) {};
    \node[w] at (1,8) {};
    \node[b] at (2,8) {};
    \node[w] at (3,8) {};
    \node[b] at (4,8) {};
    \node[b] at (5,8) {};
    \node[b] at (6,8) {};
    \node[b] at (7,8) {};
    \node[b] at (8,8) {};
    
    \node[b] at (0,9) {};
    \node[b] at (1,9) {};
    \node[w] at (2,9) {};
    \node[b] at (3,9) {};
    \node[w] at (4,9) {};
    \node[b] at (5,9) {};
    \node[b] at (6,9) {};
    \node[b] at (7,9) {};
    \node[b] at (8,9) {};
    \node[b] at (9,9) {};
    
    \node[b] at ( 0,10) {};
    \node[b] at ( 1,10) {};
    \node[b] at ( 2,10) {};
    \node[w] at ( 3,10) {};
    \node[b] at ( 4,10) {};
    \node[w] at ( 5,10) {};
    \node[b] at ( 6,10) {};
    \node[b] at ( 7,10) {};
    \node[b] at ( 8,10) {};
    \node[b] at ( 9,10) {};
    \node[b] at (10,10) {};
    
    \node[b] at ( 0,11) {};
    \node[w] at ( 1,11) {};
    \node[w] at ( 2,11) {};
    \node[b] at ( 3,11) {};
    \node[w] at ( 4,11) {};
    \node[b] at ( 5,11) {};
    \node[w] at ( 6,11) {};
    \node[b] at ( 7,11) {};
    \node[b] at ( 8,11) {};
    \node[b] at ( 9,11) {};
    \node[b] at (10,11) {};
    \node[b] at (11,11) {};
    
    \node[b] at ( 0,12) {};
    \node[w] at ( 1,12) {};
    \node[w] at ( 2,12) {};
    \node[w] at ( 3,12) {};
    \node[b] at ( 4,12) {};
    \node[w] at ( 5,12) {};
    \node[b] at ( 6,12) {};
    \node[w] at ( 7,12) {};
    \node[b] at ( 8,12) {};
    \node[b] at ( 9,12) {};
    \node[b] at (10,12) {};
    \node[b] at (11,12) {};
    \node[b] at (12,12) {};
    
    \node[b] at ( 0,13) {};
    \node[b] at ( 1,13) {};
    \node[w] at ( 2,13) {};
    \node[w] at ( 3,13) {};
    \node[w] at ( 4,13) {};
    \node[b] at ( 5,13) {};
    \node[w] at ( 6,13) {};
    \node[b] at ( 7,13) {};
    \node[w] at ( 8,13) {};
    \node[b] at ( 9,13) {};
    \node[b] at (10,13) {};
    \node[b] at (11,13) {};
    \node[b] at (12,13) {};
    \node[b] at (13,13) {};
    
    \node[b] at ( 0,14) {};
    \node[b] at ( 1,14) {};
    \node[b] at ( 2,14) {};
    \node[w] at ( 3,14) {};
    \node[w] at ( 4,14) {};
    \node[w] at ( 5,14) {};
    \node[b] at ( 6,14) {};
    \node[w] at ( 7,14) {};
    \node[b] at ( 8,14) {};
    \node[w] at ( 9,14) {};
    \node[b] at (10,14) {};
    \node[b] at (11,14) {};
    \node[b] at (12,14) {};
    \node[b] at (13,14) {};
    \node[b] at (14,14) {};
    
    \node[b] at ( 0,15) {};
    \node[b] at ( 1,15) {};
    \node[b] at ( 2,15) {};
    \node[b] at ( 3,15) {};
    \node[w] at ( 4,15) {};
    \node[w] at ( 5,15) {};
    \node[w] at ( 6,15) {};
    \node[b] at ( 7,15) {};
    \node[w] at ( 8,15) {};
    \node[b] at ( 9,15) {};
    \node[w] at (10,15) {};
    \node[b] at (11,15) {};
    \node[b] at (12,15) {};
    \node[b] at (13,15) {};
    \node[b] at (14,15) {};
    \node[b] at (15,15) {};
    
    \node[b] at ( 0,16) {};
    \node[w] at ( 1,16) {};
    \node[w] at ( 2,16) {};
    \node[w] at ( 3,16) {};
    \node[b] at ( 4,16) {};
    \node[w] at ( 5,16) {};
    \node[w] at ( 6,16) {};
    \node[w] at ( 7,16) {};
    \node[b] at ( 8,16) {};
    \node[w] at ( 9,16) {};
    \node[b] at (10,16) {};
    \node[w] at (11,16) {};
    \node[b] at (12,16) {};
    \node[b] at (13,16) {};
    \node[b] at (14,16) {};
    \node[b] at (15,16) {};
    \node[b] at (16,16) {};
    
    \node[b] at ( 0,17) {};
    \node[b] at ( 1,17) {};
    \node[w] at ( 2,17) {};
    \node[w] at ( 3,17) {};
    \node[w] at ( 4,17) {};
    \node[b] at ( 5,17) {};
    \node[w] at ( 6,17) {};
    \node[w] at ( 7,17) {};
    \node[w] at ( 8,17) {};
    \node[b] at ( 9,17) {};
    \node[w] at (10,17) {};
    \node[b] at (11,17) {};
    \node[w] at (12,17) {};
    \node[b] at (13,17) {};
    \node[b] at (14,17) {};
    \node[b] at (15,17) {};
    \node[b] at (16,17) {};
    \node[b] at (17,17) {};
    
    \node[b] at ( 0,18) {};
    \node[b] at ( 1,18) {};
    \node[b] at ( 2,18) {};
    \node[w] at ( 3,18) {};
    \node[w] at ( 4,18) {};
    \node[w] at ( 5,18) {};
    \node[b] at ( 6,18) {};
    \node[w] at ( 7,18) {};
    \node[w] at ( 8,18) {};
    \node[w] at ( 9,18) {};
    \node[b] at (10,18) {};
    \node[w] at (11,18) {};
    \node[b] at (12,18) {};
    \node[w] at (13,18) {};
    \node[b] at (14,18) {};
    \node[b] at (15,18) {};
    \node[b] at (16,18) {};
    \node[b] at (17,18) {};
    \node[b] at (18,18) {};
    
    \node[b] at ( 0,19) {};
    \node[b] at ( 1,19) {};
    \node[b] at ( 2,19) {};
    \node[b] at ( 3,19) {};
    \node[w] at ( 4,19) {};
    \node[w] at ( 5,19) {};
    \node[w] at ( 6,19) {};
    \node[b] at ( 7,19) {};
    \node[w] at ( 8,19) {};
    \node[w] at ( 9,19) {};
    \node[w] at (10,19) {};
    \node[b] at (11,19) {};
    \node[w] at (12,19) {};
    \node[b] at (13,19) {};
    \node[w] at (14,19) {};
    \node[b] at (15,19) {};
    \node[b] at (16,19) {};
    \node[b] at (17,19) {};
    \node[b] at (18,19) {};
    \node[b] at (19,19) {};
    
    \node[b] at ( 0,20) {};
    \node[b] at ( 1,20) {};
    \node[b] at ( 2,20) {};
    \node[b] at ( 3,20) {};
    \node[b] at ( 4,20) {};
    \node[w] at ( 5,20) {};
    \node[w] at ( 6,20) {};
    \node[w] at ( 7,20) {};
    \node[b] at ( 8,20) {};
    \node[w] at ( 9,20) {};
    \node[w] at (10,20) {};
    \node[w] at (11,20) {};
    \node[b] at (12,20) {};
    \node[w] at (13,20) {};
    \node[b] at (14,20) {};
    \node[w] at (15,20) {};
    \node[b] at (16,20) {};
    \node[b] at (17,20) {};
    \node[b] at (18,20) {};
    \node[b] at (19,20) {};
    \node[b] at (20,20) {};
  \end{tikzpicture}
 
 \end{center}
 \caption{Example of the grids used at subsequent steps of the nested mesh method (from top to bottom). On each line the active points of the grid are shown by black and non-active by grey dots.}
 \label{fig:nest}
\end{figure}

\section{Inversion of triangular Toeplitz matrices} \label{TTOEP}
\subsection{Basic properties of triangular Toeplitz matrices}
Let $\T_n$ be a set of lower triangular Toeplitz $n\times n$ matrices%
\footnote{Here and further we write matrix and vector indices in round brackets instead of putting them as subscripts, in order to introduce the convenient notation for QTT representation later.}%
, i.e.,
$$
A \in\T_n \quad \Leftrightarrow\quad A = \left[a(j,k)\right]_{j,k=0}^{n-1},\quad a(j,k) = a(j-k), \quad  a(p)=0, \quad  p < 0.
$$

It is easy to check the following properties of $\T_n.$
\begin{enumerate}
 \item $A \in\T_n, \: B \in\T_n \:\Rightarrow\: AB \in\T_n;$
 \item $A \in\T_n, \: B \in\T_n \:\Rightarrow\: AB = BA;$
 \item $A \in\T_n, \: a_0\neq0 \:\Rightarrow\:  A^{-1} \in\T_n.$
\end{enumerate}
By the last property, the inverse matrix $B=A^{-1},$ as well as all matrices from $\T_n,$ is defined by its first column.
The standard solution method for triangular linear systems has complexity $\O(n^2)$ and yields the following trivial formula 
\begin{equation}\label{eq:inv}%
 b(0)=\frac{1}{a(0)}, \qquad b(j) = -\frac{1}{a(0)} \sum_{k=1}^j b(j-k) a(k), \quad j=1,\ldots,n-1.
\end{equation}

For $A, B \in\T_n,$ the product  $X=AB \in\T_n$ and is also defined by the first column $x=A b.$
Therefore, matrix-by-matrix multiplication in $\T_n$ is equivalent to the multiplication of a vector by the Toeplitz matrix, i.e., discrete \emph{convolution} $x(j) = \sum_{k=0}^j a(j-k) b(k).$ 
A naive computation by this formula requires $\O(n^2)$ operations, but it is well-known that it can be computed in $\O(n \log n)$ operations using the fast Fourier transform (FFT) algorithm~\cite{gauss-fft,cooleytukey-fft}.
To recall this, we note that each $n \times n$ Toeplitz matrix $T$ is the leading submatrix of some $2n \times 2n$ circulant matrix
\begin{equation}\nonumber
C = \begin{bmatrix} T & * \\ * & T \end{bmatrix}, \qquad C = \left[c(j,k)\right], \quad\mbox{where}\quad c(j,k)=c(j-k \mod 2n),
\end{equation}
and all circulant matrices are diagonalized by unitary Fourier matrix as follows (see cf.~\cite{golub-1996})
$$
C =  F^* \Lambda F, \qquad \Lambda = \sqrt{2n} \diag(F c).
$$
Therefore, multiplication by $C$ and hence by $T$ can be performed by 3 FFTs of size $n$ with complexity $\O(n \log n).$

The inversion of triangular Toeplitz matrices has asymptotically the same complexity, i.e., $c M(n),$ where $M(n)$ denotes the complexity of matrix multiplication.
The modern highly-improved inversion algorithms reduce the constant to the level from $c=1.4$ to $c=1.5,$ see, e.g.~\cite{ttoep-inv-murphy}.
We now recall the classical algorithms, which have slightly larger constant $c,$ but are much more simple and easy to follow.
In Sec.~\ref{TTINV} we will adjust the classical inversion algorithms to use the compressed format for the approximate representation of matrix, reducing the complexity to \emph{sublinear} w.r.t. $n.$

\subsection{Divide and conquer method}
To benefit from the Toeplitz structure and reach $\O(n \log n)$ complexity for the inversion algorithm, we can use the divide-and-conquer strategy. 
This was noted in~\cite{morf-dc-1980} and developed in~\cite{commenges-dc-1984}.
It is easy to check that if $2n \times 2n$ lower triangular Toeplitz matrix $A' \in \T_{2n}$  is partitioned to $n \times n$ matrices, the inverse matrix writes as follows
\begin{equation}\label{eq:dc}
  A' = \begin{bmatrix} A &  \\ C & A \end{bmatrix}, \qquad (A')^{-1} = \begin{bmatrix} A^{-1} &  \\ -A^{-1}CA^{-1} & A^{-1} \end{bmatrix} = 
 \begin{bmatrix} A^{-1} &  \\ & A^{-1} \end{bmatrix} \: \begin{bmatrix} I &  \\ -CA^{-1} & I \end{bmatrix},
\end{equation}
where $A\in\T_n,$ $A^{-1}\in\T_n$ and $C$ is a Toeplitz matrix.
If $n=2^d$ and $A_d\in\T_{2^d},$ this formula yields the reccurent method to compute $A_d^{-1}.$ 
We start from some small $d_0$ and use~\eqref{eq:inv} to compute the inverse of $2^{d_0}\times 2^{d_0}$ leading submatrix $A_{d_0}\in\T_{2^{d_0}}.$
Then we subsequently apply~\eqref{eq:dc} and compute $A^{-1}_d$ in $(d-d_0)$ steps. 
Each step requires to compute the first column of $A^{-1}_t C_t A^{-1}_t$ with $2^t \times 2^t$ Toeplitz matrices $A^{-1}_t$ and $C_t,$ where $t=d_0+1,\ldots,d.$
Each multiplication is done in $\O(t 2^t)$ operations, which summarizes to $\O(d 2^d) = \O(n \log n)$ overall complexity.
More accurately, the cost of the divide and conquer algorithm is smaller than $12$ FFTs of size $n.$

\subsection{Bini's and related approximate methods}
In order to reduce the number of FFTs used in computations and obtain algorithm with better parallel performance, the approximate method to compute $A^{-1}$ for $A\in\T_n$ was proposed in~\cite{bini-1984}.
It is noted that $\T_n$ is the algebra generated by the matrix $H\in\T_n$ with unit elements on the subdiagonal and zeros elsewhere, i.e., transposed Jordan block with zero diagonal.
Therefore, $A\in\T_n$ with first column $a=[a(j)]_{j=0}^{n-1}$ is written as $A = \sum_{j=0}^{n-1} a(j) H^j.$
The idea is to add a small element $\eps^n$ at the top right corner of the matrix and substitute $H$ by $H_\eps = H + \eps^n e_0^\t e_{n-1}.$
It is easy to check that $D_\eps H_\eps D_\eps^{-1} = \eps C, $
where $D_\eps=\diag\{\eps^j\}_{j=0}^{n-1},$ and $C = H_1 = H + e_0^\t e_{n-1}$ generates the algebra of circulant $n \times n$ matrices.
Then $A$ and $A^{-1}$ are approximated as follows
\begin{equation}\label{eq:bini}
 \begin{split}
  A \approx \tilde A_\eps & = \sum_{j=0}^{n-1} a(j) H_\eps^j 
                          = D_\eps^{-1} \left( \sum_{j=0}^{n-1} a(j) \eps^j C_j  \right) D_\eps 
                          = D_\eps^{-1} C_\eps D_\eps
                          = D_\eps^{-1} F^* \Lambda_\eps F D_\eps, \\
  A^{-1} \approx \tilde A_\eps^{-1} & = D_\eps^{-1} F^* \Lambda_\eps^{-1} F D_\eps, 
     \qquad \mbox{where}\quad \Lambda_\eps = \sqrt{n} \diag(F a_\eps), \quad a_\eps(j)=a(j)\eps^j.
 \end{split}
\end{equation}
The first column of $A_\eps^{-1}$ is computed using two FFTs of size $n.$

This idea was revised in~\cite{mng-bini-2004}, where it was proposed to apply Bini's algorithm to the first column $a$ of matrix $A$ padded with $n$ zero elements.
The revised version of Bini's algorithm requires two FFTs of size $2n$ and has better accuracy properties.

\subsection{Newton iteration}
The classical Newton iteration 
\begin{equation}\label{eq:nw}
 B_{k+1} = 2 B_k - B_k A B_k
\end{equation}
was proposed in~\cite{schulz-newton-1933} for the computing the inverse $A^{-1}$ of a nonsingular matrix $A.$
It converges quadratically if initial guess $B_0$ is s.t. $\|I - AB_0\|\leq 1$ in any operator norm of a matrix.
In~\cite{benisrael-newton-1966} it is shown that for $B_0 = \mu A^*$ with some small real $\mu$ Newton iteration converges to the inverse $A^{-1}$ of a nonsingular or pseudoinverse $A^\dagger$ of a singlular matrix $A.$
In~\cite{pan-newton-1991} further deep analysis is provided, for instance, it is shown that $\mu^{-1} = \|A\|_1\|A\|_{\infty}$ is a good and reliable choice.
In relation to Toeplitz and related structured matrices, the Newton iteration with approximation of the result on each step was developed using the concept of displacement ranks~\cite{pan-newton-2002} and tensor product approximations \cite{hkt-iter-2008,oot-newton-2008,ost-latensor-2009}.

For $A\in\T_n$ the choice of initial guess $B_0 = \mu A^*$ is not effective, since $A^* \notin\T_n$ and we can not perform iterations with $B_k\in\T_n,$ which grants low storage and fast multiplication.

If $B_0\in\T_n,$ every Newton iteration costs two convolutions, i.e., $6$ FFTs of size $2n.$ 
\begin{remark}
A single Newton iteration for lower triangular Toeplitz matrices is slower than the divide and conquer method.
\end{remark}
It is not easy to provide a good initial guess $B_0\in\T_n$ for which the Newton iteration with a given matrix $A\in\T_n$ converges in one or few steps.
However, Newton iteration can be used to improve the accuracy of matrix $B \approx A^{-1}, B\in\T_n$ computed by other means, if $\|I - AB\|\leq 1.$
For instance, we can note the following relation between the divide and conquer method and Newton iteration.
\begin{remark}
For matrix $A' \in \T_{2n}$ defined in~\eqref{eq:dc}, the Newton iteration~\eqref{eq:nw} with initial guess
\begin{equation}\nonumber
B_0 = \begin{bmatrix} A^{-1} & 0 \\ 0 & 0\end{bmatrix}, \qquad A \in \T_n,
\end{equation}
gives $B_1 = (A')^{-1}$, i.e., converges in one step and is equivalent to the divide and conquer method~\eqref{eq:dc}.
\end{remark}
Therefore, for $A\in\T_n,$ divide and conquer method is always better than the Newton iteration, which reduces to the divide and conquer method in the special case.

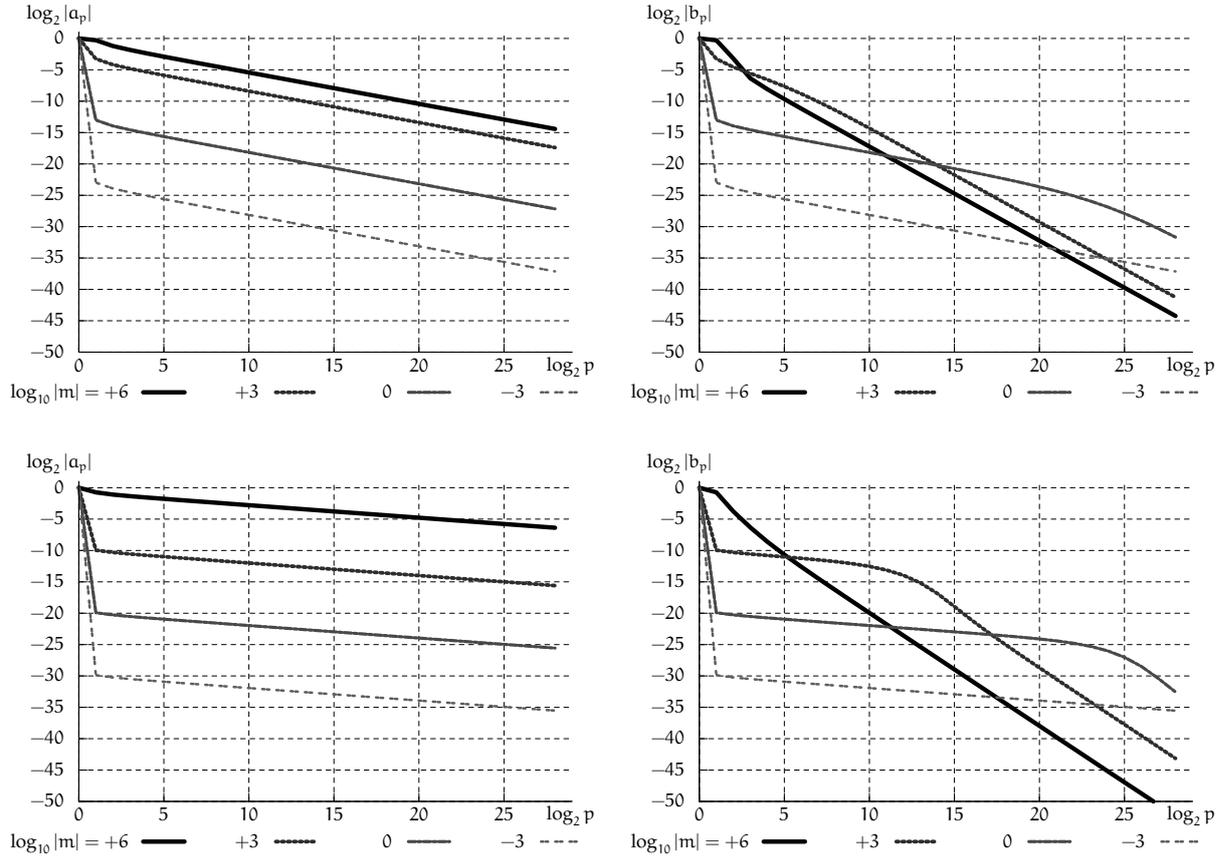
\begin{figure}[t]
 \begin{center} \hfil
  \resizebox{.48\textwidth}{!}{\begin{tikzpicture}[gnuplot]
\gpcolor{gp lt color axes}
\gpsetlinetype{gp lt axes}
\gpsetlinewidth{1.00}
\draw[gp path] (1.012,1.232)--(11.947,1.232);
\gpcolor{gp lt color border}
\gpsetlinetype{gp lt border}
\draw[gp path] (1.012,1.232)--(1.192,1.232);
\node[gp node right] at (0.828,1.232) {$-50$};
\gpcolor{gp lt color axes}
\gpsetlinetype{gp lt axes}
\draw[gp path] (1.012,1.933)--(11.947,1.933);
\gpcolor{gp lt color border}
\gpsetlinetype{gp lt border}
\draw[gp path] (1.012,1.933)--(1.192,1.933);
\node[gp node right] at (0.828,1.933) {$-45$};
\gpcolor{gp lt color axes}
\gpsetlinetype{gp lt axes}
\draw[gp path] (1.012,2.634)--(11.947,2.634);
\gpcolor{gp lt color border}
\gpsetlinetype{gp lt border}
\draw[gp path] (1.012,2.634)--(1.192,2.634);
\node[gp node right] at (0.828,2.634) {$-40$};
\gpcolor{gp lt color axes}
\gpsetlinetype{gp lt axes}
\draw[gp path] (1.012,3.335)--(11.947,3.335);
\gpcolor{gp lt color border}
\gpsetlinetype{gp lt border}
\draw[gp path] (1.012,3.335)--(1.192,3.335);
\node[gp node right] at (0.828,3.335) {$-35$};
\gpcolor{gp lt color axes}
\gpsetlinetype{gp lt axes}
\draw[gp path] (1.012,4.036)--(11.947,4.036);
\gpcolor{gp lt color border}
\gpsetlinetype{gp lt border}
\draw[gp path] (1.012,4.036)--(1.192,4.036);
\node[gp node right] at (0.828,4.036) {$-30$};
\gpcolor{gp lt color axes}
\gpsetlinetype{gp lt axes}
\draw[gp path] (1.012,4.736)--(11.947,4.736);
\gpcolor{gp lt color border}
\gpsetlinetype{gp lt border}
\draw[gp path] (1.012,4.736)--(1.192,4.736);
\node[gp node right] at (0.828,4.736) {$-25$};
\gpcolor{gp lt color axes}
\gpsetlinetype{gp lt axes}
\draw[gp path] (1.012,5.437)--(11.947,5.437);
\gpcolor{gp lt color border}
\gpsetlinetype{gp lt border}
\draw[gp path] (1.012,5.437)--(1.192,5.437);
\node[gp node right] at (0.828,5.437) {$-20$};
\gpcolor{gp lt color axes}
\gpsetlinetype{gp lt axes}
\draw[gp path] (1.012,6.138)--(11.947,6.138);
\gpcolor{gp lt color border}
\gpsetlinetype{gp lt border}
\draw[gp path] (1.012,6.138)--(1.192,6.138);
\node[gp node right] at (0.828,6.138) {$-15$};
\gpcolor{gp lt color axes}
\gpsetlinetype{gp lt axes}
\draw[gp path] (1.012,6.839)--(11.947,6.839);
\gpcolor{gp lt color border}
\gpsetlinetype{gp lt border}
\draw[gp path] (1.012,6.839)--(1.192,6.839);
\node[gp node right] at (0.828,6.839) {$-10$};
\gpcolor{gp lt color axes}
\gpsetlinetype{gp lt axes}
\draw[gp path] (1.012,7.540)--(11.947,7.540);
\gpcolor{gp lt color border}
\gpsetlinetype{gp lt border}
\draw[gp path] (1.012,7.540)--(1.192,7.540);
\node[gp node right] at (0.828,7.540) {$-5$};
\gpcolor{gp lt color axes}
\gpsetlinetype{gp lt axes}
\draw[gp path] (1.012,8.241)--(11.947,8.241);
\gpcolor{gp lt color border}
\gpsetlinetype{gp lt border}
\draw[gp path] (1.012,8.241)--(1.192,8.241);
\node[gp node right] at (0.828,8.241) {$0$};
\gpcolor{gp lt color axes}
\gpsetlinetype{gp lt axes}
\draw[gp path] (1.012,1.232)--(1.012,8.381);
\gpcolor{gp lt color border}
\gpsetlinetype{gp lt border}
\draw[gp path] (1.012,1.232)--(1.012,1.412);
\node[gp node center] at (1.012,0.924) {$0$};
\gpcolor{gp lt color axes}
\gpsetlinetype{gp lt axes}
\draw[gp path] (2.897,1.232)--(2.897,8.381);
\gpcolor{gp lt color border}
\gpsetlinetype{gp lt border}
\draw[gp path] (2.897,1.232)--(2.897,1.412);
\node[gp node center] at (2.897,0.924) {$5$};
\gpcolor{gp lt color axes}
\gpsetlinetype{gp lt axes}
\draw[gp path] (4.783,1.232)--(4.783,8.381);
\gpcolor{gp lt color border}
\gpsetlinetype{gp lt border}
\draw[gp path] (4.783,1.232)--(4.783,1.412);
\node[gp node center] at (4.783,0.924) {$10$};
\gpcolor{gp lt color axes}
\gpsetlinetype{gp lt axes}
\draw[gp path] (6.668,1.232)--(6.668,8.381);
\gpcolor{gp lt color border}
\gpsetlinetype{gp lt border}
\draw[gp path] (6.668,1.232)--(6.668,1.412);
\node[gp node center] at (6.668,0.924) {$15$};
\gpcolor{gp lt color axes}
\gpsetlinetype{gp lt axes}
\draw[gp path] (8.553,1.232)--(8.553,8.381);
\gpcolor{gp lt color border}
\gpsetlinetype{gp lt border}
\draw[gp path] (8.553,1.232)--(8.553,1.412);
\node[gp node center] at (8.553,0.924) {$20$};
\gpcolor{gp lt color axes}
\gpsetlinetype{gp lt axes}
\draw[gp path] (10.439,1.232)--(10.439,8.381);
\gpcolor{gp lt color border}
\gpsetlinetype{gp lt border}
\draw[gp path] (10.439,1.232)--(10.439,1.412);
\node[gp node center] at (10.439,0.924) {$25$};
\draw[gp path] (1.012,8.381)--(1.012,1.232)--(11.947,1.232);
\node[gp node center] at (0.575,8.810) {$\log_2|a_p|$};
\node[gp node right] at (12.603,0.946) {$\log_2p$};
\node[gp node right] at (2.255,0.334) {$\log_{10}|m|=+6$};
\gpcolor{\gprgb{0}{0}{0}}
\gpsetlinetype{gp lt plot 0}
\gpsetlinewidth{7.00}
\draw[gp path] (2.439,0.334)--(3.355,0.334);
\draw[gp path] (1.012,8.241)--(1.389,8.201)--(1.766,8.072)--(2.143,7.985)--(2.520,7.907)%
  --(2.897,7.834)--(3.274,7.762)--(3.651,7.691)--(4.029,7.621)--(4.406,7.551)--(4.783,7.480)%
  --(5.160,7.410)--(5.537,7.340)--(5.914,7.270)--(6.291,7.200)--(6.668,7.130)--(7.045,7.060)%
  --(7.422,6.990)--(7.799,6.920)--(8.176,6.850)--(8.553,6.779)--(8.930,6.709)--(9.308,6.639)%
  --(9.685,6.569)--(10.062,6.499)--(10.439,6.429)--(10.816,6.359)--(11.193,6.289)--(11.570,6.219);
\gpcolor{gp lt color border}
\node[gp node right] at (5.195,0.334) {$+3$};
\gpcolor{\gprgb{178}{178}{178}}
\gpsetlinetype{gp lt plot 2}
\gpsetlinewidth{6.00}
\draw[gp path] (5.379,0.334)--(6.295,0.334);
\draw[gp path] (1.012,8.241)--(1.389,7.785)--(1.766,7.655)--(2.143,7.568)--(2.520,7.491)%
  --(2.897,7.418)--(3.274,7.346)--(3.651,7.275)--(4.029,7.205)--(4.406,7.134)--(4.783,7.064)%
  --(5.160,6.994)--(5.537,6.924)--(5.914,6.854)--(6.291,6.784)--(6.668,6.714)--(7.045,6.644)%
  --(7.422,6.573)--(7.799,6.503)--(8.176,6.433)--(8.553,6.363)--(8.930,6.293)--(9.308,6.223)%
  --(9.685,6.153)--(10.062,6.083)--(10.439,6.013)--(10.816,5.943)--(11.193,5.873)--(11.570,5.803);
\gpcolor{gp lt color border}
\node[gp node right] at (8.135,0.334) {$ 0$};
\gpcolor{\gprgb{282}{282}{282}}
\gpsetlinetype{gp lt plot 4}
\gpsetlinewidth{5.00}
\draw[gp path] (8.319,0.334)--(9.235,0.334);
\draw[gp path] (1.012,8.241)--(1.389,6.416)--(1.766,6.286)--(2.143,6.199)--(2.520,6.122)%
  --(2.897,6.048)--(3.274,5.976)--(3.651,5.906)--(4.029,5.835)--(4.406,5.765)--(4.783,5.695)%
  --(5.160,5.624)--(5.537,5.554)--(5.914,5.484)--(6.291,5.414)--(6.668,5.344)--(7.045,5.274)%
  --(7.422,5.204)--(7.799,5.134)--(8.176,5.064)--(8.553,4.994)--(8.930,4.924)--(9.308,4.853)%
  --(9.685,4.783)--(10.062,4.713)--(10.439,4.643)--(10.816,4.573)--(11.193,4.503)--(11.570,4.433);
\gpcolor{gp lt color border}
\node[gp node right] at (11.075,0.334) {$-3$};
\gpcolor{\gprgb{370}{370}{370}}
\gpsetlinetype{gp lt plot 6}
\gpsetlinewidth{4.00}
\draw[gp path] (11.259,0.334)--(12.175,0.334);
\draw[gp path] (1.012,8.241)--(1.389,5.019)--(1.766,4.889)--(2.143,4.802)--(2.520,4.725)%
  --(2.897,4.651)--(3.274,4.580)--(3.651,4.509)--(4.029,4.438)--(4.406,4.368)--(4.783,4.298)%
  --(5.160,4.228)--(5.537,4.157)--(5.914,4.087)--(6.291,4.017)--(6.668,3.947)--(7.045,3.877)%
  --(7.422,3.807)--(7.799,3.737)--(8.176,3.667)--(8.553,3.597)--(8.930,3.527)--(9.308,3.457)%
  --(9.685,3.386)--(10.062,3.316)--(10.439,3.246)--(10.816,3.176)--(11.193,3.106)--(11.570,3.036);
\gpcolor{gp lt color border}
\gpsetlinetype{gp lt border}
\gpsetlinewidth{1.00}
\draw[gp path] (1.012,8.381)--(1.012,1.232)--(11.947,1.232);
\gpdefrectangularnode{gp plot 1}{\pgfpoint{1.012cm}{1.232cm}}{\pgfpoint{11.947cm}{8.381cm}}
\end{tikzpicture}
} \hfil
  \resizebox{.48\textwidth}{!}{\begin{tikzpicture}[gnuplot]
\gpcolor{gp lt color axes}
\gpsetlinetype{gp lt axes}
\gpsetlinewidth{1.00}
\draw[gp path] (1.012,1.232)--(11.947,1.232);
\gpcolor{gp lt color border}
\gpsetlinetype{gp lt border}
\draw[gp path] (1.012,1.232)--(1.192,1.232);
\node[gp node right] at (0.828,1.232) {$-50$};
\gpcolor{gp lt color axes}
\gpsetlinetype{gp lt axes}
\draw[gp path] (1.012,1.933)--(11.947,1.933);
\gpcolor{gp lt color border}
\gpsetlinetype{gp lt border}
\draw[gp path] (1.012,1.933)--(1.192,1.933);
\node[gp node right] at (0.828,1.933) {$-45$};
\gpcolor{gp lt color axes}
\gpsetlinetype{gp lt axes}
\draw[gp path] (1.012,2.634)--(11.947,2.634);
\gpcolor{gp lt color border}
\gpsetlinetype{gp lt border}
\draw[gp path] (1.012,2.634)--(1.192,2.634);
\node[gp node right] at (0.828,2.634) {$-40$};
\gpcolor{gp lt color axes}
\gpsetlinetype{gp lt axes}
\draw[gp path] (1.012,3.335)--(11.947,3.335);
\gpcolor{gp lt color border}
\gpsetlinetype{gp lt border}
\draw[gp path] (1.012,3.335)--(1.192,3.335);
\node[gp node right] at (0.828,3.335) {$-35$};
\gpcolor{gp lt color axes}
\gpsetlinetype{gp lt axes}
\draw[gp path] (1.012,4.036)--(11.947,4.036);
\gpcolor{gp lt color border}
\gpsetlinetype{gp lt border}
\draw[gp path] (1.012,4.036)--(1.192,4.036);
\node[gp node right] at (0.828,4.036) {$-30$};
\gpcolor{gp lt color axes}
\gpsetlinetype{gp lt axes}
\draw[gp path] (1.012,4.736)--(11.947,4.736);
\gpcolor{gp lt color border}
\gpsetlinetype{gp lt border}
\draw[gp path] (1.012,4.736)--(1.192,4.736);
\node[gp node right] at (0.828,4.736) {$-25$};
\gpcolor{gp lt color axes}
\gpsetlinetype{gp lt axes}
\draw[gp path] (1.012,5.437)--(11.947,5.437);
\gpcolor{gp lt color border}
\gpsetlinetype{gp lt border}
\draw[gp path] (1.012,5.437)--(1.192,5.437);
\node[gp node right] at (0.828,5.437) {$-20$};
\gpcolor{gp lt color axes}
\gpsetlinetype{gp lt axes}
\draw[gp path] (1.012,6.138)--(11.947,6.138);
\gpcolor{gp lt color border}
\gpsetlinetype{gp lt border}
\draw[gp path] (1.012,6.138)--(1.192,6.138);
\node[gp node right] at (0.828,6.138) {$-15$};
\gpcolor{gp lt color axes}
\gpsetlinetype{gp lt axes}
\draw[gp path] (1.012,6.839)--(11.947,6.839);
\gpcolor{gp lt color border}
\gpsetlinetype{gp lt border}
\draw[gp path] (1.012,6.839)--(1.192,6.839);
\node[gp node right] at (0.828,6.839) {$-10$};
\gpcolor{gp lt color axes}
\gpsetlinetype{gp lt axes}
\draw[gp path] (1.012,7.540)--(11.947,7.540);
\gpcolor{gp lt color border}
\gpsetlinetype{gp lt border}
\draw[gp path] (1.012,7.540)--(1.192,7.540);
\node[gp node right] at (0.828,7.540) {$-5$};
\gpcolor{gp lt color axes}
\gpsetlinetype{gp lt axes}
\draw[gp path] (1.012,8.241)--(11.947,8.241);
\gpcolor{gp lt color border}
\gpsetlinetype{gp lt border}
\draw[gp path] (1.012,8.241)--(1.192,8.241);
\node[gp node right] at (0.828,8.241) {$0$};
\gpcolor{gp lt color axes}
\gpsetlinetype{gp lt axes}
\draw[gp path] (1.012,1.232)--(1.012,8.381);
\gpcolor{gp lt color border}
\gpsetlinetype{gp lt border}
\draw[gp path] (1.012,1.232)--(1.012,1.412);
\node[gp node center] at (1.012,0.924) {$0$};
\gpcolor{gp lt color axes}
\gpsetlinetype{gp lt axes}
\draw[gp path] (2.897,1.232)--(2.897,8.381);
\gpcolor{gp lt color border}
\gpsetlinetype{gp lt border}
\draw[gp path] (2.897,1.232)--(2.897,1.412);
\node[gp node center] at (2.897,0.924) {$5$};
\gpcolor{gp lt color axes}
\gpsetlinetype{gp lt axes}
\draw[gp path] (4.783,1.232)--(4.783,8.381);
\gpcolor{gp lt color border}
\gpsetlinetype{gp lt border}
\draw[gp path] (4.783,1.232)--(4.783,1.412);
\node[gp node center] at (4.783,0.924) {$10$};
\gpcolor{gp lt color axes}
\gpsetlinetype{gp lt axes}
\draw[gp path] (6.668,1.232)--(6.668,8.381);
\gpcolor{gp lt color border}
\gpsetlinetype{gp lt border}
\draw[gp path] (6.668,1.232)--(6.668,1.412);
\node[gp node center] at (6.668,0.924) {$15$};
\gpcolor{gp lt color axes}
\gpsetlinetype{gp lt axes}
\draw[gp path] (8.553,1.232)--(8.553,8.381);
\gpcolor{gp lt color border}
\gpsetlinetype{gp lt border}
\draw[gp path] (8.553,1.232)--(8.553,1.412);
\node[gp node center] at (8.553,0.924) {$20$};
\gpcolor{gp lt color axes}
\gpsetlinetype{gp lt axes}
\draw[gp path] (10.439,1.232)--(10.439,8.381);
\gpcolor{gp lt color border}
\gpsetlinetype{gp lt border}
\draw[gp path] (10.439,1.232)--(10.439,1.412);
\node[gp node center] at (10.439,0.924) {$25$};
\draw[gp path] (1.012,8.381)--(1.012,1.232)--(11.947,1.232);
\node[gp node center] at (0.575,8.810) {$\log_2|b_p|$};
\node[gp node right] at (12.603,0.946) {$\log_2p$};
\node[gp node right] at (2.255,0.334) {$\log_{10}|m|=+6$};
\gpcolor{\gprgb{0}{0}{0}}
\gpsetlinetype{gp lt plot 0}
\gpsetlinewidth{7.00}
\draw[gp path] (2.439,0.334)--(3.355,0.334);
\draw[gp path] (1.012,8.241)--(1.389,8.201)--(1.766,7.792)--(2.143,7.350)--(2.520,7.107)%
  --(2.897,6.887)--(3.274,6.672)--(3.651,6.459)--(4.029,6.248)--(4.406,6.037)--(4.783,5.826)%
  --(5.160,5.616)--(5.537,5.406)--(5.914,5.195)--(6.291,4.985)--(6.668,4.775)--(7.045,4.564)%
  --(7.422,4.354)--(7.799,4.144)--(8.176,3.934)--(8.553,3.723)--(8.930,3.513)--(9.308,3.303)%
  --(9.685,3.093)--(10.062,2.882)--(10.439,2.672)--(10.816,2.462)--(11.193,2.252)--(11.570,2.041);
\gpcolor{gp lt color border}
\node[gp node right] at (5.195,0.334) {$+3$};
\gpcolor{\gprgb{178}{178}{178}}
\gpsetlinetype{gp lt plot 2}
\gpsetlinewidth{6.00}
\draw[gp path] (5.379,0.334)--(6.295,0.334);
\draw[gp path] (1.012,8.241)--(1.389,7.785)--(1.766,7.600)--(2.143,7.461)--(2.520,7.320)%
  --(2.897,7.169)--(3.274,7.004)--(3.651,6.826)--(4.029,6.637)--(4.406,6.439)--(4.783,6.235)%
  --(5.160,6.028)--(5.537,5.820)--(5.914,5.611)--(6.291,5.401)--(6.668,5.191)--(7.045,4.981)%
  --(7.422,4.770)--(7.799,4.560)--(8.176,4.350)--(8.553,4.140)--(8.930,3.929)--(9.308,3.719)%
  --(9.685,3.509)--(10.062,3.299)--(10.439,3.088)--(10.816,2.878)--(11.193,2.668)--(11.570,2.458);
\gpcolor{gp lt color border}
\node[gp node right] at (8.135,0.334) {$ 0$};
\gpcolor{\gprgb{282}{282}{282}}
\gpsetlinetype{gp lt plot 4}
\gpsetlinewidth{5.00}
\draw[gp path] (8.319,0.334)--(9.235,0.334);
\draw[gp path] (1.012,8.241)--(1.389,6.416)--(1.766,6.286)--(2.143,6.199)--(2.520,6.121)%
  --(2.897,6.048)--(3.274,5.976)--(3.651,5.905)--(4.029,5.834)--(4.406,5.763)--(4.783,5.692)%
  --(5.160,5.621)--(5.537,5.550)--(5.914,5.478)--(6.291,5.405)--(6.668,5.332)--(7.045,5.257)%
  --(7.422,5.179)--(7.799,5.099)--(8.176,5.015)--(8.553,4.926)--(8.930,4.830)--(9.308,4.724)%
  --(9.685,4.607)--(10.062,4.477)--(10.439,4.330)--(10.816,4.168)--(11.193,3.990)--(11.570,3.801);
\gpcolor{gp lt color border}
\node[gp node right] at (11.075,0.334) {$-3$};
\gpcolor{\gprgb{370}{370}{370}}
\gpsetlinetype{gp lt plot 6}
\gpsetlinewidth{4.00}
\draw[gp path] (11.259,0.334)--(12.175,0.334);
\draw[gp path] (1.012,8.241)--(1.389,5.019)--(1.766,4.889)--(2.143,4.802)--(2.520,4.725)%
  --(2.897,4.651)--(3.274,4.580)--(3.651,4.509)--(4.029,4.438)--(4.406,4.368)--(4.783,4.298)%
  --(5.160,4.228)--(5.537,4.157)--(5.914,4.087)--(6.291,4.017)--(6.668,3.947)--(7.045,3.877)%
  --(7.422,3.807)--(7.799,3.737)--(8.176,3.667)--(8.553,3.597)--(8.930,3.527)--(9.308,3.456)%
  --(9.685,3.386)--(10.062,3.316)--(10.439,3.246)--(10.816,3.176)--(11.193,3.105)--(11.570,3.035);
\gpcolor{gp lt color border}
\gpsetlinetype{gp lt border}
\gpsetlinewidth{1.00}
\draw[gp path] (1.012,8.381)--(1.012,1.232)--(11.947,1.232);
\gpdefrectangularnode{gp plot 1}{\pgfpoint{1.012cm}{1.232cm}}{\pgfpoint{11.947cm}{8.381cm}}
\end{tikzpicture}
} \hfil
 \end{center}
 \begin{center} \hfil
  \resizebox{.48\textwidth}{!}{\begin{tikzpicture}[gnuplot]
\gpcolor{gp lt color axes}
\gpsetlinetype{gp lt axes}
\gpsetlinewidth{1.00}
\draw[gp path] (1.012,1.232)--(11.947,1.232);
\gpcolor{gp lt color border}
\gpsetlinetype{gp lt border}
\draw[gp path] (1.012,1.232)--(1.192,1.232);
\node[gp node right] at (0.828,1.232) {$-50$};
\gpcolor{gp lt color axes}
\gpsetlinetype{gp lt axes}
\draw[gp path] (1.012,1.933)--(11.947,1.933);
\gpcolor{gp lt color border}
\gpsetlinetype{gp lt border}
\draw[gp path] (1.012,1.933)--(1.192,1.933);
\node[gp node right] at (0.828,1.933) {$-45$};
\gpcolor{gp lt color axes}
\gpsetlinetype{gp lt axes}
\draw[gp path] (1.012,2.634)--(11.947,2.634);
\gpcolor{gp lt color border}
\gpsetlinetype{gp lt border}
\draw[gp path] (1.012,2.634)--(1.192,2.634);
\node[gp node right] at (0.828,2.634) {$-40$};
\gpcolor{gp lt color axes}
\gpsetlinetype{gp lt axes}
\draw[gp path] (1.012,3.335)--(11.947,3.335);
\gpcolor{gp lt color border}
\gpsetlinetype{gp lt border}
\draw[gp path] (1.012,3.335)--(1.192,3.335);
\node[gp node right] at (0.828,3.335) {$-35$};
\gpcolor{gp lt color axes}
\gpsetlinetype{gp lt axes}
\draw[gp path] (1.012,4.036)--(11.947,4.036);
\gpcolor{gp lt color border}
\gpsetlinetype{gp lt border}
\draw[gp path] (1.012,4.036)--(1.192,4.036);
\node[gp node right] at (0.828,4.036) {$-30$};
\gpcolor{gp lt color axes}
\gpsetlinetype{gp lt axes}
\draw[gp path] (1.012,4.736)--(11.947,4.736);
\gpcolor{gp lt color border}
\gpsetlinetype{gp lt border}
\draw[gp path] (1.012,4.736)--(1.192,4.736);
\node[gp node right] at (0.828,4.736) {$-25$};
\gpcolor{gp lt color axes}
\gpsetlinetype{gp lt axes}
\draw[gp path] (1.012,5.437)--(11.947,5.437);
\gpcolor{gp lt color border}
\gpsetlinetype{gp lt border}
\draw[gp path] (1.012,5.437)--(1.192,5.437);
\node[gp node right] at (0.828,5.437) {$-20$};
\gpcolor{gp lt color axes}
\gpsetlinetype{gp lt axes}
\draw[gp path] (1.012,6.138)--(11.947,6.138);
\gpcolor{gp lt color border}
\gpsetlinetype{gp lt border}
\draw[gp path] (1.012,6.138)--(1.192,6.138);
\node[gp node right] at (0.828,6.138) {$-15$};
\gpcolor{gp lt color axes}
\gpsetlinetype{gp lt axes}
\draw[gp path] (1.012,6.839)--(11.947,6.839);
\gpcolor{gp lt color border}
\gpsetlinetype{gp lt border}
\draw[gp path] (1.012,6.839)--(1.192,6.839);
\node[gp node right] at (0.828,6.839) {$-10$};
\gpcolor{gp lt color axes}
\gpsetlinetype{gp lt axes}
\draw[gp path] (1.012,7.540)--(11.947,7.540);
\gpcolor{gp lt color border}
\gpsetlinetype{gp lt border}
\draw[gp path] (1.012,7.540)--(1.192,7.540);
\node[gp node right] at (0.828,7.540) {$-5$};
\gpcolor{gp lt color axes}
\gpsetlinetype{gp lt axes}
\draw[gp path] (1.012,8.241)--(11.947,8.241);
\gpcolor{gp lt color border}
\gpsetlinetype{gp lt border}
\draw[gp path] (1.012,8.241)--(1.192,8.241);
\node[gp node right] at (0.828,8.241) {$0$};
\gpcolor{gp lt color axes}
\gpsetlinetype{gp lt axes}
\draw[gp path] (1.012,1.232)--(1.012,8.381);
\gpcolor{gp lt color border}
\gpsetlinetype{gp lt border}
\draw[gp path] (1.012,1.232)--(1.012,1.412);
\node[gp node center] at (1.012,0.924) {$0$};
\gpcolor{gp lt color axes}
\gpsetlinetype{gp lt axes}
\draw[gp path] (2.897,1.232)--(2.897,8.381);
\gpcolor{gp lt color border}
\gpsetlinetype{gp lt border}
\draw[gp path] (2.897,1.232)--(2.897,1.412);
\node[gp node center] at (2.897,0.924) {$5$};
\gpcolor{gp lt color axes}
\gpsetlinetype{gp lt axes}
\draw[gp path] (4.783,1.232)--(4.783,8.381);
\gpcolor{gp lt color border}
\gpsetlinetype{gp lt border}
\draw[gp path] (4.783,1.232)--(4.783,1.412);
\node[gp node center] at (4.783,0.924) {$10$};
\gpcolor{gp lt color axes}
\gpsetlinetype{gp lt axes}
\draw[gp path] (6.668,1.232)--(6.668,8.381);
\gpcolor{gp lt color border}
\gpsetlinetype{gp lt border}
\draw[gp path] (6.668,1.232)--(6.668,1.412);
\node[gp node center] at (6.668,0.924) {$15$};
\gpcolor{gp lt color axes}
\gpsetlinetype{gp lt axes}
\draw[gp path] (8.553,1.232)--(8.553,8.381);
\gpcolor{gp lt color border}
\gpsetlinetype{gp lt border}
\draw[gp path] (8.553,1.232)--(8.553,1.412);
\node[gp node center] at (8.553,0.924) {$20$};
\gpcolor{gp lt color axes}
\gpsetlinetype{gp lt axes}
\draw[gp path] (10.439,1.232)--(10.439,8.381);
\gpcolor{gp lt color border}
\gpsetlinetype{gp lt border}
\draw[gp path] (10.439,1.232)--(10.439,1.412);
\node[gp node center] at (10.439,0.924) {$25$};
\draw[gp path] (1.012,8.381)--(1.012,1.232)--(11.947,1.232);
\node[gp node center] at (0.575,8.810) {$\log_2|a_p|$};
\node[gp node right] at (12.603,0.946) {$\log_2p$};
\node[gp node right] at (2.255,0.334) {$\log_{10}|m|=+6$};
\gpcolor{\gprgb{0}{0}{0}}
\gpsetlinetype{gp lt plot 0}
\gpsetlinewidth{7.00}
\draw[gp path] (2.439,0.334)--(3.355,0.334);
\draw[gp path] (1.012,8.241)--(1.389,8.138)--(1.766,8.088)--(2.143,8.053)--(2.520,8.022)%
  --(2.897,7.993)--(3.274,7.964)--(3.651,7.936)--(4.029,7.908)--(4.406,7.879)--(4.783,7.851)%
  --(5.160,7.823)--(5.537,7.795)--(5.914,7.767)--(6.291,7.739)--(6.668,7.711)--(7.045,7.683)%
  --(7.422,7.655)--(7.799,7.627)--(8.176,7.599)--(8.553,7.571)--(8.930,7.543)--(9.308,7.515)%
  --(9.685,7.487)--(10.062,7.459)--(10.439,7.431)--(10.816,7.403)--(11.193,7.375)--(11.570,7.347);
\gpcolor{gp lt color border}
\node[gp node right] at (5.195,0.334) {$+3$};
\gpcolor{\gprgb{178}{178}{178}}
\gpsetlinetype{gp lt plot 2}
\gpsetlinewidth{6.00}
\draw[gp path] (5.379,0.334)--(6.295,0.334);
\draw[gp path] (1.012,8.241)--(1.389,6.845)--(1.766,6.796)--(2.143,6.761)--(2.520,6.730)%
  --(2.897,6.701)--(3.274,6.672)--(3.651,6.644)--(4.029,6.615)--(4.406,6.587)--(4.783,6.559)%
  --(5.160,6.531)--(5.537,6.503)--(5.914,6.475)--(6.291,6.447)--(6.668,6.419)--(7.045,6.391)%
  --(7.422,6.363)--(7.799,6.335)--(8.176,6.307)--(8.553,6.279)--(8.930,6.251)--(9.308,6.223)%
  --(9.685,6.195)--(10.062,6.167)--(10.439,6.139)--(10.816,6.111)--(11.193,6.083)--(11.570,6.054);
\gpcolor{gp lt color border}
\node[gp node right] at (8.135,0.334) {$ 0$};
\gpcolor{\gprgb{282}{282}{282}}
\gpsetlinetype{gp lt plot 4}
\gpsetlinewidth{5.00}
\draw[gp path] (8.319,0.334)--(9.235,0.334);
\draw[gp path] (1.012,8.241)--(1.389,5.448)--(1.766,5.399)--(2.143,5.364)--(2.520,5.333)%
  --(2.897,5.304)--(3.274,5.275)--(3.651,5.247)--(4.029,5.219)--(4.406,5.190)--(4.783,5.162)%
  --(5.160,5.134)--(5.537,5.106)--(5.914,5.078)--(6.291,5.050)--(6.668,5.022)--(7.045,4.994)%
  --(7.422,4.966)--(7.799,4.938)--(8.176,4.910)--(8.553,4.882)--(8.930,4.854)--(9.308,4.826)%
  --(9.685,4.798)--(10.062,4.770)--(10.439,4.742)--(10.816,4.714)--(11.193,4.686)--(11.570,4.658);
\gpcolor{gp lt color border}
\node[gp node right] at (11.075,0.334) {$-3$};
\gpcolor{\gprgb{370}{370}{370}}
\gpsetlinetype{gp lt plot 6}
\gpsetlinewidth{4.00}
\draw[gp path] (11.259,0.334)--(12.175,0.334);
\draw[gp path] (1.012,8.241)--(1.389,4.052)--(1.766,4.002)--(2.143,3.967)--(2.520,3.936)%
  --(2.897,3.907)--(3.274,3.878)--(3.651,3.850)--(4.029,3.822)--(4.406,3.793)--(4.783,3.765)%
  --(5.160,3.737)--(5.537,3.709)--(5.914,3.681)--(6.291,3.653)--(6.668,3.625)--(7.045,3.597)%
  --(7.422,3.569)--(7.799,3.541)--(8.176,3.513)--(8.553,3.485)--(8.930,3.457)--(9.308,3.429)%
  --(9.685,3.401)--(10.062,3.373)--(10.439,3.345)--(10.816,3.317)--(11.193,3.289)--(11.570,3.261);
\gpcolor{gp lt color border}
\gpsetlinetype{gp lt border}
\gpsetlinewidth{1.00}
\draw[gp path] (1.012,8.381)--(1.012,1.232)--(11.947,1.232);
\gpdefrectangularnode{gp plot 1}{\pgfpoint{1.012cm}{1.232cm}}{\pgfpoint{11.947cm}{8.381cm}}
\end{tikzpicture}
} \hfil
  \resizebox{.48\textwidth}{!}{\begin{tikzpicture}[gnuplot]
\gpcolor{gp lt color axes}
\gpsetlinetype{gp lt axes}
\gpsetlinewidth{1.00}
\draw[gp path] (1.012,1.232)--(11.947,1.232);
\gpcolor{gp lt color border}
\gpsetlinetype{gp lt border}
\draw[gp path] (1.012,1.232)--(1.192,1.232);
\node[gp node right] at (0.828,1.232) {$-50$};
\gpcolor{gp lt color axes}
\gpsetlinetype{gp lt axes}
\draw[gp path] (1.012,1.933)--(11.947,1.933);
\gpcolor{gp lt color border}
\gpsetlinetype{gp lt border}
\draw[gp path] (1.012,1.933)--(1.192,1.933);
\node[gp node right] at (0.828,1.933) {$-45$};
\gpcolor{gp lt color axes}
\gpsetlinetype{gp lt axes}
\draw[gp path] (1.012,2.634)--(11.947,2.634);
\gpcolor{gp lt color border}
\gpsetlinetype{gp lt border}
\draw[gp path] (1.012,2.634)--(1.192,2.634);
\node[gp node right] at (0.828,2.634) {$-40$};
\gpcolor{gp lt color axes}
\gpsetlinetype{gp lt axes}
\draw[gp path] (1.012,3.335)--(11.947,3.335);
\gpcolor{gp lt color border}
\gpsetlinetype{gp lt border}
\draw[gp path] (1.012,3.335)--(1.192,3.335);
\node[gp node right] at (0.828,3.335) {$-35$};
\gpcolor{gp lt color axes}
\gpsetlinetype{gp lt axes}
\draw[gp path] (1.012,4.036)--(11.947,4.036);
\gpcolor{gp lt color border}
\gpsetlinetype{gp lt border}
\draw[gp path] (1.012,4.036)--(1.192,4.036);
\node[gp node right] at (0.828,4.036) {$-30$};
\gpcolor{gp lt color axes}
\gpsetlinetype{gp lt axes}
\draw[gp path] (1.012,4.736)--(11.947,4.736);
\gpcolor{gp lt color border}
\gpsetlinetype{gp lt border}
\draw[gp path] (1.012,4.736)--(1.192,4.736);
\node[gp node right] at (0.828,4.736) {$-25$};
\gpcolor{gp lt color axes}
\gpsetlinetype{gp lt axes}
\draw[gp path] (1.012,5.437)--(11.947,5.437);
\gpcolor{gp lt color border}
\gpsetlinetype{gp lt border}
\draw[gp path] (1.012,5.437)--(1.192,5.437);
\node[gp node right] at (0.828,5.437) {$-20$};
\gpcolor{gp lt color axes}
\gpsetlinetype{gp lt axes}
\draw[gp path] (1.012,6.138)--(11.947,6.138);
\gpcolor{gp lt color border}
\gpsetlinetype{gp lt border}
\draw[gp path] (1.012,6.138)--(1.192,6.138);
\node[gp node right] at (0.828,6.138) {$-15$};
\gpcolor{gp lt color axes}
\gpsetlinetype{gp lt axes}
\draw[gp path] (1.012,6.839)--(11.947,6.839);
\gpcolor{gp lt color border}
\gpsetlinetype{gp lt border}
\draw[gp path] (1.012,6.839)--(1.192,6.839);
\node[gp node right] at (0.828,6.839) {$-10$};
\gpcolor{gp lt color axes}
\gpsetlinetype{gp lt axes}
\draw[gp path] (1.012,7.540)--(11.947,7.540);
\gpcolor{gp lt color border}
\gpsetlinetype{gp lt border}
\draw[gp path] (1.012,7.540)--(1.192,7.540);
\node[gp node right] at (0.828,7.540) {$-5$};
\gpcolor{gp lt color axes}
\gpsetlinetype{gp lt axes}
\draw[gp path] (1.012,8.241)--(11.947,8.241);
\gpcolor{gp lt color border}
\gpsetlinetype{gp lt border}
\draw[gp path] (1.012,8.241)--(1.192,8.241);
\node[gp node right] at (0.828,8.241) {$0$};
\gpcolor{gp lt color axes}
\gpsetlinetype{gp lt axes}
\draw[gp path] (1.012,1.232)--(1.012,8.381);
\gpcolor{gp lt color border}
\gpsetlinetype{gp lt border}
\draw[gp path] (1.012,1.232)--(1.012,1.412);
\node[gp node center] at (1.012,0.924) {$0$};
\gpcolor{gp lt color axes}
\gpsetlinetype{gp lt axes}
\draw[gp path] (2.897,1.232)--(2.897,8.381);
\gpcolor{gp lt color border}
\gpsetlinetype{gp lt border}
\draw[gp path] (2.897,1.232)--(2.897,1.412);
\node[gp node center] at (2.897,0.924) {$5$};
\gpcolor{gp lt color axes}
\gpsetlinetype{gp lt axes}
\draw[gp path] (4.783,1.232)--(4.783,8.381);
\gpcolor{gp lt color border}
\gpsetlinetype{gp lt border}
\draw[gp path] (4.783,1.232)--(4.783,1.412);
\node[gp node center] at (4.783,0.924) {$10$};
\gpcolor{gp lt color axes}
\gpsetlinetype{gp lt axes}
\draw[gp path] (6.668,1.232)--(6.668,8.381);
\gpcolor{gp lt color border}
\gpsetlinetype{gp lt border}
\draw[gp path] (6.668,1.232)--(6.668,1.412);
\node[gp node center] at (6.668,0.924) {$15$};
\gpcolor{gp lt color axes}
\gpsetlinetype{gp lt axes}
\draw[gp path] (8.553,1.232)--(8.553,8.381);
\gpcolor{gp lt color border}
\gpsetlinetype{gp lt border}
\draw[gp path] (8.553,1.232)--(8.553,1.412);
\node[gp node center] at (8.553,0.924) {$20$};
\gpcolor{gp lt color axes}
\gpsetlinetype{gp lt axes}
\draw[gp path] (10.439,1.232)--(10.439,8.381);
\gpcolor{gp lt color border}
\gpsetlinetype{gp lt border}
\draw[gp path] (10.439,1.232)--(10.439,1.412);
\node[gp node center] at (10.439,0.924) {$25$};
\draw[gp path] (1.012,8.381)--(1.012,1.232)--(11.947,1.232);
\node[gp node center] at (0.575,8.810) {$\log_2|b_p|$};
\node[gp node right] at (12.603,0.946) {$\log_2p$};
\node[gp node right] at (2.255,0.334) {$\log_{10}|m|=+6$};
\gpcolor{\gprgb{0}{0}{0}}
\gpsetlinetype{gp lt plot 0}
\gpsetlinewidth{7.00}
\draw[gp path] (2.439,0.334)--(3.355,0.334);
\draw[gp path] (1.012,8.241)--(1.389,8.138)--(1.766,7.711)--(2.143,7.354)--(2.520,7.034)%
  --(2.897,6.745)--(3.274,6.474)--(3.651,6.212)--(4.029,5.954)--(4.406,5.699)--(4.783,5.445)%
  --(5.160,5.192)--(5.537,4.939)--(5.914,4.686)--(6.291,4.434)--(6.668,4.181)--(7.045,3.929)%
  --(7.422,3.677)--(7.799,3.424)--(8.176,3.172)--(8.553,2.920)--(8.930,2.667)--(9.308,2.415)%
  --(9.685,2.163)--(10.062,1.910)--(10.439,1.659)--(10.816,1.409)--(11.081,1.232);
\gpcolor{gp lt color border}
\node[gp node right] at (5.195,0.334) {$+3$};
\gpcolor{\gprgb{178}{178}{178}}
\gpsetlinetype{gp lt plot 2}
\gpsetlinewidth{6.00}
\draw[gp path] (5.379,0.334)--(6.295,0.334);
\draw[gp path] (1.012,8.241)--(1.389,6.845)--(1.766,6.795)--(2.143,6.760)--(2.520,6.728)%
  --(2.897,6.696)--(3.274,6.664)--(3.651,6.629)--(4.029,6.591)--(4.406,6.544)--(4.783,6.484)%
  --(5.160,6.402)--(5.537,6.285)--(5.914,6.114)--(6.291,5.877)--(6.668,5.586)--(7.045,5.288)%
  --(7.422,5.006)--(7.799,4.737)--(8.176,4.476)--(8.553,4.219)--(8.930,3.963)--(9.308,3.709)%
  --(9.685,3.456)--(10.062,3.203)--(10.439,2.951)--(10.816,2.698)--(11.193,2.446)--(11.570,2.193);
\gpcolor{gp lt color border}
\node[gp node right] at (8.135,0.334) {$ 0$};
\gpcolor{\gprgb{282}{282}{282}}
\gpsetlinetype{gp lt plot 4}
\gpsetlinewidth{5.00}
\draw[gp path] (8.319,0.334)--(9.235,0.334);
\draw[gp path] (1.012,8.241)--(1.389,5.448)--(1.766,5.399)--(2.143,5.364)--(2.520,5.333)%
  --(2.897,5.304)--(3.274,5.275)--(3.651,5.247)--(4.029,5.218)--(4.406,5.190)--(4.783,5.162)%
  --(5.160,5.134)--(5.537,5.106)--(5.914,5.078)--(6.291,5.049)--(6.668,5.021)--(7.045,4.992)%
  --(7.422,4.962)--(7.799,4.932)--(8.176,4.899)--(8.553,4.862)--(8.930,4.820)--(9.308,4.767)%
  --(9.685,4.697)--(10.062,4.598)--(10.439,4.453)--(10.816,4.247)--(11.193,3.977)--(11.570,3.677);
\gpcolor{gp lt color border}
\node[gp node right] at (11.075,0.334) {$-3$};
\gpcolor{\gprgb{370}{370}{370}}
\gpsetlinetype{gp lt plot 6}
\gpsetlinewidth{4.00}
\draw[gp path] (11.259,0.334)--(12.175,0.334);
\draw[gp path] (1.012,8.241)--(1.389,4.052)--(1.766,4.002)--(2.143,3.967)--(2.520,3.936)%
  --(2.897,3.907)--(3.274,3.878)--(3.651,3.850)--(4.029,3.822)--(4.406,3.793)--(4.783,3.765)%
  --(5.160,3.737)--(5.537,3.709)--(5.914,3.681)--(6.291,3.653)--(6.668,3.625)--(7.045,3.597)%
  --(7.422,3.569)--(7.799,3.541)--(8.176,3.513)--(8.553,3.485)--(8.930,3.457)--(9.308,3.429)%
  --(9.685,3.401)--(10.062,3.373)--(10.439,3.344)--(10.816,3.316)--(11.193,3.288)--(11.570,3.259);
\gpcolor{gp lt color border}
\gpsetlinetype{gp lt border}
\gpsetlinewidth{1.00}
\draw[gp path] (1.012,8.381)--(1.012,1.232)--(11.947,1.232);
\gpdefrectangularnode{gp plot 1}{\pgfpoint{1.012cm}{1.232cm}}{\pgfpoint{11.947cm}{8.381cm}}
\end{tikzpicture}
} \hfil
 \end{center}
\caption{Decay profiles of triangular Toeplitz matrix~\eqref{eq3} (left) and its inverse (right) for $n=2^{28},$ $h=T/n,$  $T=10$ and $\alpha=0.5$ (top) and $\alpha=0.8$ (bottom) and for different mass $m.$} 
\label{fig:decay} 
\end{figure}
\subsection{Decay of the elements of inverse matrix}
It is instructive to look at the decay profiles of the elements of a triangular Toeplitz matrix~\eqref{eq3} and its inverse, see Fig.~\ref{fig:decay}.
There is a jump in magnitude between diagonal and subdiagonal elements, i.e.,
$$
\frac{a_0}{a_1} = \frac{1-(\gamma m)^{-1}}{2^{\alpha+1} - 2}, \qquad \gamma=\frac{h^\alpha}{\Gamma(\alpha+2)},
$$
where the numerator increases when $n \to \infty,$ $h \to 0$ and tends to one when $m \to -\infty.$
After the jump, elements decay polynomially, i.e., $a_p \sim p^{\alpha-1}$ for $p \geq 1.$ 
For the inverse matrix the behaviour is the same for a certain (possibly very long) set of elements.
However, after certain point the rate of decay changes from $1-\alpha$ to $1+\alpha,$ i.e., $b_p \sim p^{-\alpha-1}$ for $p\geq P.$
The bend point $P$ which is obtained from the experiment, is the monotonically decreasing function $P=P(\h m),$ i.e. the larger is the initial jump, the later the decay of element of the inverse matrix switches to faster rate.
The observed behaviour of elements of inverse matrix allows us to predict the upper bound for the norm of the second half of vector, using the information about the first half.
We will use this property in the next subsection, where the divide and conquer algorithm will be adapted for the vectors approximately given in the low--parametrical tensor--structured format.

\section{Inversion of triangular Toeplitz matrices using QTT approximation} \label{TTINV}
\subsection{Tensor train and quantized tensor train formats}
A \emph{tensor} is an array with $d$ indices (or \emph{modes})
$$
\A = [a(k_1,\ldots,k_d)], \qquad k_p = 0,\ldots,n_p-1, \quad p=1,\ldots,d.
$$
The tensor train (TT) format~\cite{osel-newten-2009eng,osel-tt-2011} for the tensor $\A$ reads\footnote{We will often write the equations in elementwise form, which assumes that all indices run through all possible values.}
\begin{equation}\label{eq:tt}
 a(k_1,k_2,\ldots,k_d) = A^{(1)}_{k_1} A^{(2)}_{k_2} \ldots A^{(d)}_{k_d},
\end{equation}
where each $A^{(p)}_{k_p}$ is an $r_{p-1} \x r_p$ matrix.
Usually the \emph{border conditions} $r_0 = r_d = 1$ are imposed to make every entry $a(k_1,\ldots,k_d)$ a scalar.
However, larger $r_0$ and $r_d$ can be considered and every entry of a tensor $\A = [a(k_1,\ldots,k_d)]$ becomes an $r_0 \x r_d$ matrix.
Values $r_0,\ldots,r_{d-1}$ are referred to as~\emph{TT--ranks} and characterize the~\emph{separation properties} of the tensor~$\A.$
Three-dimensional arrays $A^{(p)} = [A^{(p)}_{k_p}]$ are referred to as~\emph{TT--cores}.

To apply the TT compression to low dimensional data, the idea of \emph{quantization} was proposed~\cite{osel-2d2d-2010,khor-qtt-2011}.
We will explain the idea for a one-dimensional vector $a = [a(k)]_{k=0}^{n-1},$ restricting the discussion to $n=2^d.$
Define the binary notation of index $k$ as follows
\begin{equation}\label{eq:bit}
      k=\overline{k_1\ldots k_d}  \eqdef \sum_{p=1}^d k_p 2^{p-1}, \qquad k_p=0,1.
\end{equation}
The isomorphic mapping $k \leftrightarrow (k_1,\ldots k_d)$ allows us to \emph{reshape} a vector $a=[a(k)]$ into the $d$--tensor $\dot\A=[\dot a(k_1,\ldots,k_d)].$
The TT format~\eqref{eq:tt} for the latter is called the \emph{QTT format} and reads
\begin{equation} \label{eq:qtt}
 a(k) = a(\overline{k_1 \ldots k_d}) = \dot a(k_1,\ldots,k_d) = A^{(1)}_{k_1} \ldots A^{(d)}_{k_d}.
\end{equation}
This idea appears in~\cite{osel-2d2d-2010} in the context of matrix approximation.
In~\cite{khor-qtt-2011} the TT format applied after the quantization of indices was called the \emph{QTT format} and applied to a class of functions discretized on uniform grids, revealing the impressive approximation properties.
In particular, it was proven that the QTT--ranks of $\exp x,$ $\sin x,$ $\cos x,$ $x^p$ are uniformly bounded w.r.t. the grid size.
For the functions $e^{-\alpha x^2},$  $x^\alpha,$  $\frac{\sin x}{x}$, $\frac{1}{x},$ etc., similar properties were found experimentally.


\begin{figure}[t]
 \begin{center} \hfil
  \includegraphics[width=.48\textwidth]{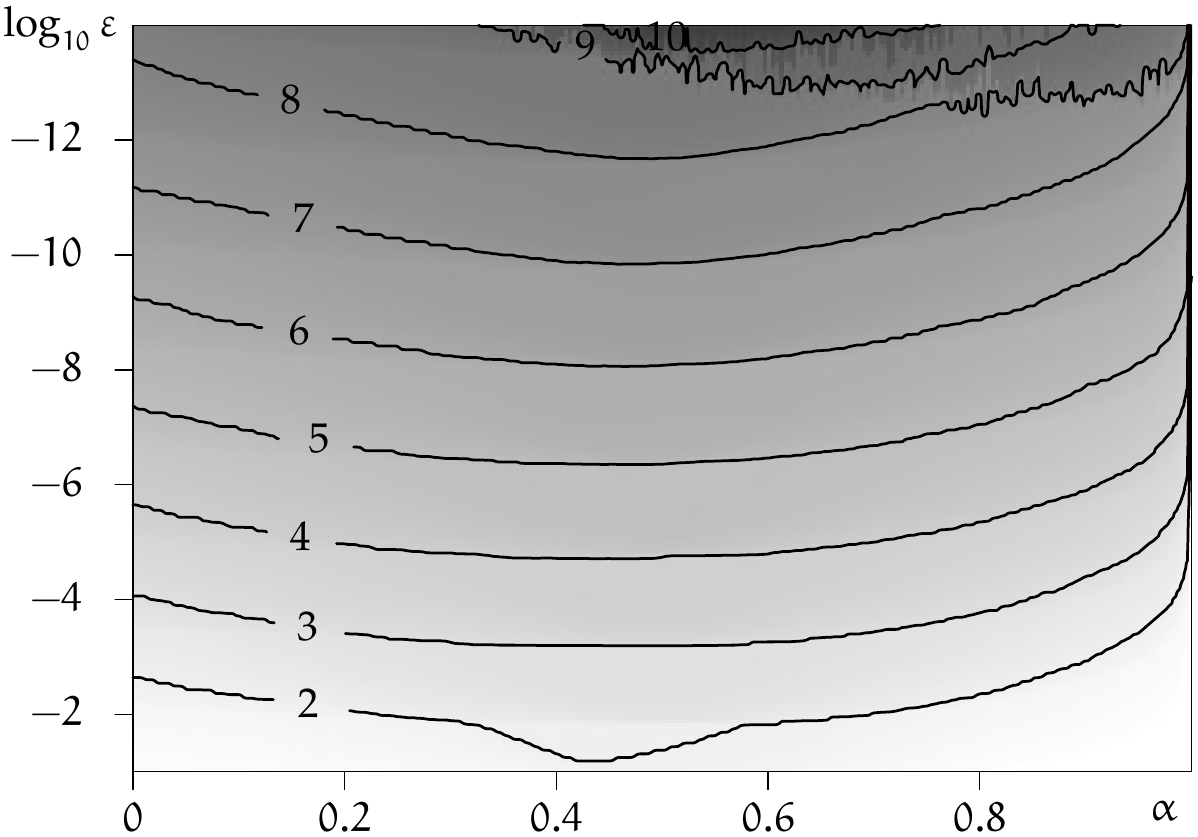} \hfil
  \includegraphics[width=.48\textwidth]{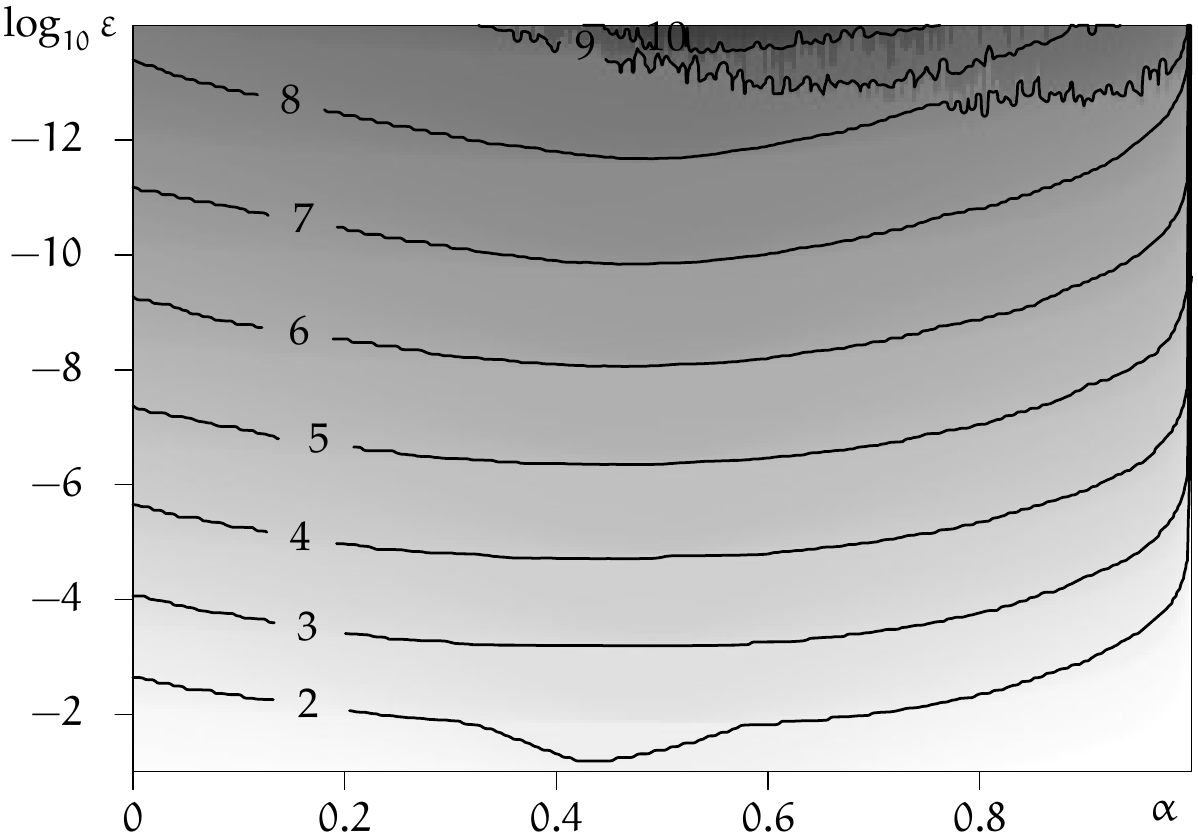} \hfil
 \end{center}
 \begin{center} \hfil
  \includegraphics[width=.48\textwidth]{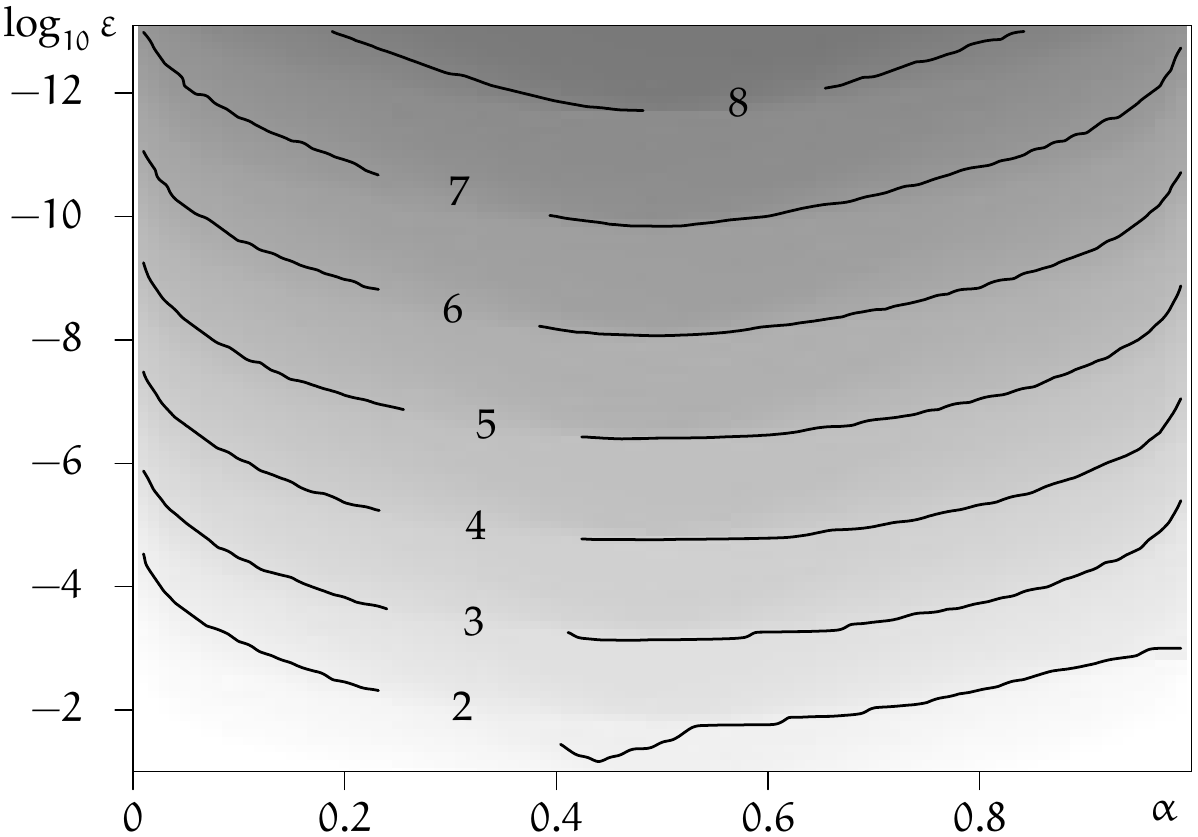} \hfil
  \includegraphics[width=.48\textwidth]{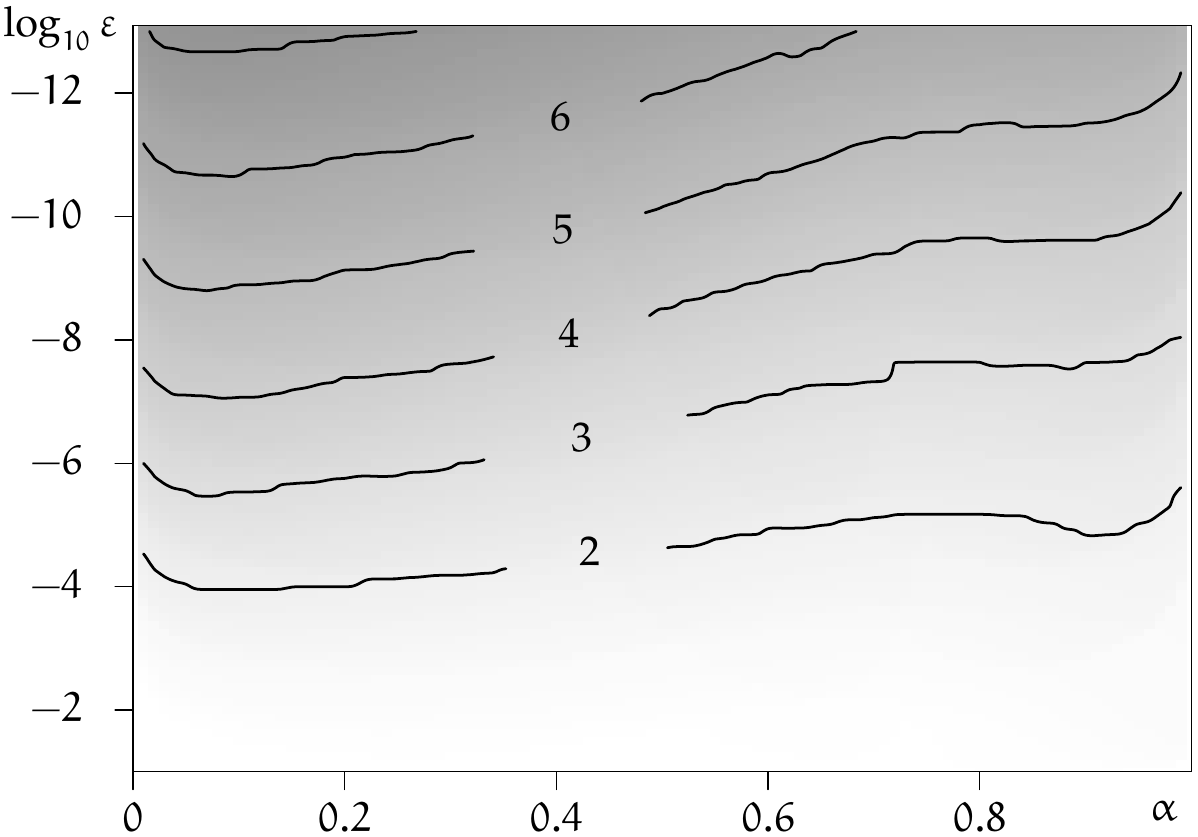} \hfil
 \end{center}
 \caption{Effective QTT rank of vector $[k^{\alpha-1}]$ (top left), vector $[(k-1)^{\alpha+1}-2k^{\alpha+1}+(k+1)^{\alpha+1}]$ (top right), first column of matrix~\eqref{eq3} (bottom left) and its inverse (bottom right) w.r.t. parameter $\alpha$ and relative approximation accuracy  $\eps$ in Frobenius norm. Problem size $n=2^{28},$ maximum time $T=10,$ mass $m=-10^{+6}$} \label{fig:r}
\end{figure}

The QTT separation function of the function $x^{\alpha-1}$ discretized on a uniform grid, is particularly important for us, because it motivates the use of the QTT approximation to develop the superfast  algorithms for the solution of fractional differential equations.
In the numerical experiment we found out that the QTT ranks are very moderate for all $0 < \alpha < 1$ and for accuracy up to the machine threshold. 
The same holds for the first column of matrix~\eqref{eq3} as well as for its inverse.
On Fig.~\ref{fig:r} we show the effective (average) QTT rank w.r.t. $\eps$ and $\alpha.$ 
We note that the effective rank does not overcome $10,$ even for very large grids up to $n=2^{28}.$

To construct a superfast algorithms in the QTT format we first have to compress the data to this format using the algorithm with the sublinear complexity. 
The original TT--SVD algorithm proposed in~\cite{osel-tt-2011} requires all elements of tensor and therefore does not suit for this purpose.
To compress matrix~\eqref{eq3} to QTT format, we apply cross interpolation algorithm TT--ACA proposed in~\cite{so-dmrgi-2011proc}.
This method computes the approximation using only a few elements of the original array, and does not require all elements to be computed.
The comparison of runtimes of TT--SVD and TT--ACA algorthms is given on Fig.~\ref{fig:dmrg}.
The time that is required to choose the good subset of elements for the interpolation depends on the structure of data, which is defined by parameters $\alpha$ and $m.$ 
It is easy to see that the behaviour of data is less regular for the large $\alpha$ and mass, which leads to larger runtimes of TT--ACA.
Nevertheless, we clearly see that TT--ACA outperforms TT--SVD for all examples and has sublinear complexity w.r.t. problem size $n.$

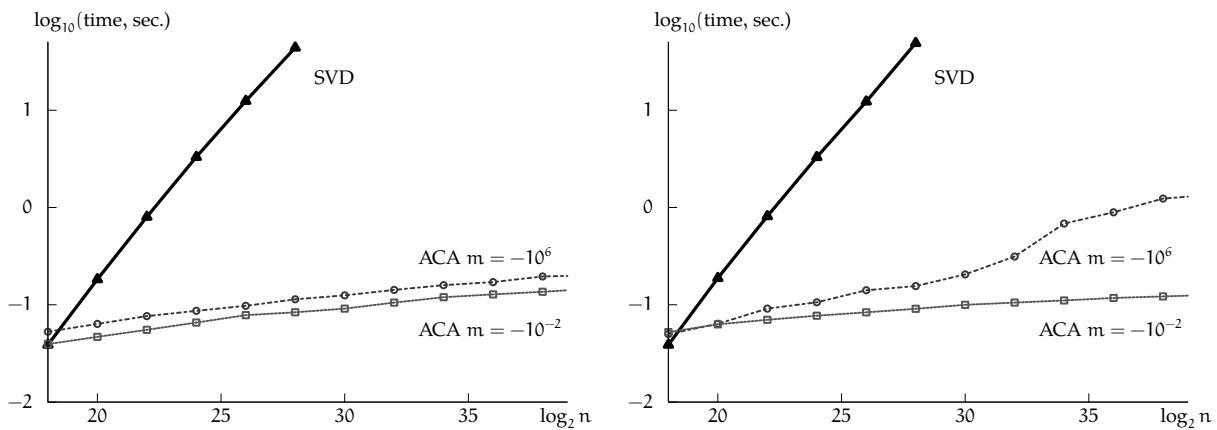
\begin{figure}[t]
 \begin{center} \hfil
  \resizebox{.48\textwidth}{!}{\begin{tikzpicture}[gnuplot]
\gpcolor{gp lt color border}
\gpsetlinetype{gp lt border}
\gpsetlinewidth{1.00}
\draw[gp path] (0.828,0.616)--(1.008,0.616);
\node[gp node right] at (0.644,0.616) {$-2$};
\draw[gp path] (0.828,2.715)--(1.008,2.715);
\node[gp node right] at (0.644,2.715) {$-1$};
\draw[gp path] (0.828,4.813)--(1.008,4.813);
\node[gp node right] at (0.644,4.813) {$0$};
\draw[gp path] (0.828,6.912)--(1.008,6.912);
\node[gp node right] at (0.644,6.912) {$1$};
\draw[gp path] (1.887,0.616)--(1.887,0.796);
\node[gp node center] at (1.887,0.308) {$20$};
\draw[gp path] (4.534,0.616)--(4.534,0.796);
\node[gp node center] at (4.534,0.308) {$25$};
\draw[gp path] (7.182,0.616)--(7.182,0.796);
\node[gp node center] at (7.182,0.308) {$30$};
\draw[gp path] (9.829,0.616)--(9.829,0.796);
\node[gp node center] at (9.829,0.308) {$35$};
\draw[gp path] (0.828,8.381)--(0.828,0.616)--(11.947,0.616);
\node[gp node left] at (0.383,8.847) {$\log_{10}(\mathrm{time},$ sec.)};
\node[gp node right] at (12.614,0.305) {$\log_2n$};
\node[gp node left] at (6.388,7.605) {SVD};
\node[gp node left] at (8.611,3.722) {ACA $m=-10^6$};
\node[gp node left] at (8.611,2.169) {ACA $m=-10^{-2}$};
\gpcolor{\gprgb{0}{0}{0}}
\gpsetlinetype{gp lt plot 0}
\gpsetlinewidth{5.00}
\draw[gp path] (0.828,1.852)--(1.887,3.264)--(2.946,4.611)--(4.005,5.903)--(5.064,7.112)%
  --(6.123,8.261);
\gpsetpointsize{7.20}
\gppoint{gp mark 9}{(0.828,1.852)}
\gppoint{gp mark 9}{(1.887,3.264)}
\gppoint{gp mark 9}{(2.946,4.611)}
\gppoint{gp mark 9}{(4.005,5.903)}
\gppoint{gp mark 9}{(5.064,7.112)}
\gppoint{gp mark 9}{(6.123,8.261)}
\gpcolor{\gprgb{215}{215}{215}}
\gpsetlinetype{gp lt plot 5}
\gpsetlinewidth{3.00}
\draw[gp path] (0.828,2.134)--(1.887,2.302)--(2.946,2.470)--(4.005,2.584)--(5.064,2.692)%
  --(6.123,2.833)--(7.182,2.918)--(8.241,3.036)--(9.300,3.138)--(10.359,3.204)--(11.418,3.326)%
  --(11.947,3.337);
\gpsetpointsize{4.80}
\gppoint{gp mark 6}{(0.828,2.134)}
\gppoint{gp mark 6}{(1.887,2.302)}
\gppoint{gp mark 6}{(2.946,2.470)}
\gppoint{gp mark 6}{(4.005,2.584)}
\gppoint{gp mark 6}{(5.064,2.692)}
\gppoint{gp mark 6}{(6.123,2.833)}
\gppoint{gp mark 6}{(7.182,2.918)}
\gppoint{gp mark 6}{(8.241,3.036)}
\gppoint{gp mark 6}{(9.300,3.138)}
\gppoint{gp mark 6}{(10.359,3.204)}
\gppoint{gp mark 6}{(11.418,3.326)}
\gpcolor{\gprgb{342}{342}{342}}
\gpsetlinetype{gp lt plot 3}
\draw[gp path] (0.828,1.866)--(1.887,2.023)--(2.946,2.175)--(4.005,2.331)--(5.064,2.491)%
  --(6.123,2.553)--(7.182,2.632)--(8.241,2.762)--(9.300,2.878)--(10.359,2.939)--(11.418,2.995)%
  --(11.947,3.028);
\gppoint{gp mark 4}{(0.828,1.866)}
\gppoint{gp mark 4}{(1.887,2.023)}
\gppoint{gp mark 4}{(2.946,2.175)}
\gppoint{gp mark 4}{(4.005,2.331)}
\gppoint{gp mark 4}{(5.064,2.491)}
\gppoint{gp mark 4}{(6.123,2.553)}
\gppoint{gp mark 4}{(7.182,2.632)}
\gppoint{gp mark 4}{(8.241,2.762)}
\gppoint{gp mark 4}{(9.300,2.878)}
\gppoint{gp mark 4}{(10.359,2.939)}
\gppoint{gp mark 4}{(11.418,2.995)}
\gpcolor{gp lt color border}
\gpsetlinetype{gp lt border}
\gpsetlinewidth{1.00}
\draw[gp path] (0.828,8.381)--(0.828,0.616)--(11.947,0.616);
\gpdefrectangularnode{gp plot 1}{\pgfpoint{0.828cm}{0.616cm}}{\pgfpoint{11.947cm}{8.381cm}}
\end{tikzpicture}
} \hfil
  \resizebox{.48\textwidth}{!}{\begin{tikzpicture}[gnuplot]
\gpcolor{gp lt color border}
\gpsetlinetype{gp lt border}
\gpsetlinewidth{1.00}
\draw[gp path] (0.828,0.616)--(1.008,0.616);
\node[gp node right] at (0.644,0.616) {$-2$};
\draw[gp path] (0.828,2.715)--(1.008,2.715);
\node[gp node right] at (0.644,2.715) {$-1$};
\draw[gp path] (0.828,4.813)--(1.008,4.813);
\node[gp node right] at (0.644,4.813) {$0$};
\draw[gp path] (0.828,6.912)--(1.008,6.912);
\node[gp node right] at (0.644,6.912) {$1$};
\draw[gp path] (1.887,0.616)--(1.887,0.796);
\node[gp node center] at (1.887,0.308) {$20$};
\draw[gp path] (4.534,0.616)--(4.534,0.796);
\node[gp node center] at (4.534,0.308) {$25$};
\draw[gp path] (7.182,0.616)--(7.182,0.796);
\node[gp node center] at (7.182,0.308) {$30$};
\draw[gp path] (9.829,0.616)--(9.829,0.796);
\node[gp node center] at (9.829,0.308) {$35$};
\draw[gp path] (0.828,8.381)--(0.828,0.616)--(11.947,0.616);
\node[gp node left] at (0.383,8.847) {$\log_{10}(\mathrm{time},$ sec.)};
\node[gp node right] at (12.614,0.305) {$\log_2n$};
\node[gp node left] at (6.388,7.605) {SVD};
\node[gp node left] at (8.611,3.722) {ACA $m=-10^6$};
\node[gp node left] at (8.611,2.169) {ACA $m=-10^{-2}$};
\gpcolor{\gprgb{0}{0}{0}}
\gpsetlinetype{gp lt plot 0}
\gpsetlinewidth{5.00}
\draw[gp path] (0.828,1.851)--(1.887,3.290)--(2.946,4.624)--(4.005,5.900)--(5.064,7.098)%
  --(6.123,8.360);
\gpsetpointsize{7.20}
\gppoint{gp mark 9}{(0.828,1.851)}
\gppoint{gp mark 9}{(1.887,3.290)}
\gppoint{gp mark 9}{(2.946,4.624)}
\gppoint{gp mark 9}{(4.005,5.900)}
\gppoint{gp mark 9}{(5.064,7.098)}
\gppoint{gp mark 9}{(6.123,8.360)}
\gpcolor{\gprgb{215}{215}{215}}
\gpsetlinetype{gp lt plot 5}
\gpsetlinewidth{3.00}
\draw[gp path] (0.828,2.082)--(1.887,2.305)--(2.946,2.632)--(4.005,2.768)--(5.064,3.028)%
  --(6.123,3.118)--(7.182,3.370)--(8.241,3.756)--(9.300,4.467)--(10.359,4.710)--(11.418,5.006)%
  --(11.947,5.047);
\gpsetpointsize{4.80}
\gppoint{gp mark 6}{(0.828,2.082)}
\gppoint{gp mark 6}{(1.887,2.305)}
\gppoint{gp mark 6}{(2.946,2.632)}
\gppoint{gp mark 6}{(4.005,2.768)}
\gppoint{gp mark 6}{(5.064,3.028)}
\gppoint{gp mark 6}{(6.123,3.118)}
\gppoint{gp mark 6}{(7.182,3.370)}
\gppoint{gp mark 6}{(8.241,3.756)}
\gppoint{gp mark 6}{(9.300,4.467)}
\gppoint{gp mark 6}{(10.359,4.710)}
\gppoint{gp mark 6}{(11.418,5.006)}
\gpcolor{\gprgb{342}{342}{342}}
\gpsetlinetype{gp lt plot 3}
\draw[gp path] (0.828,2.125)--(1.887,2.293)--(2.946,2.391)--(4.005,2.479)--(5.064,2.552)%
  --(6.123,2.628)--(7.182,2.714)--(8.241,2.762)--(9.300,2.808)--(10.359,2.862)--(11.418,2.895)%
  --(11.947,2.912);
\gppoint{gp mark 4}{(0.828,2.125)}
\gppoint{gp mark 4}{(1.887,2.293)}
\gppoint{gp mark 4}{(2.946,2.391)}
\gppoint{gp mark 4}{(4.005,2.479)}
\gppoint{gp mark 4}{(5.064,2.552)}
\gppoint{gp mark 4}{(6.123,2.628)}
\gppoint{gp mark 4}{(7.182,2.714)}
\gppoint{gp mark 4}{(8.241,2.762)}
\gppoint{gp mark 4}{(9.300,2.808)}
\gppoint{gp mark 4}{(10.359,2.862)}
\gppoint{gp mark 4}{(11.418,2.895)}
\gpcolor{gp lt color border}
\gpsetlinetype{gp lt border}
\gpsetlinewidth{1.00}
\draw[gp path] (0.828,8.381)--(0.828,0.616)--(11.947,0.616);
\gpdefrectangularnode{gp plot 1}{\pgfpoint{0.828cm}{0.616cm}}{\pgfpoint{11.947cm}{8.381cm}}
\end{tikzpicture}
} \hfil
 \end{center}
\caption{Runtimes of TT--SVD and TT--ACA algorithms for the approximation of matrix~\eqref{eq3} in the QTT format w.r.t. size $n.$ 
 (left) $\alpha=0.1,$ (right) $\alpha=0.9.$}
\label{fig:dmrg} 
\end{figure}

\subsection{Fourier transform and convolution in QTT format} \label{QTTconv}
Inversion algorithms for triangular Toeplitz matrices $A \in \T_n$ recalled in Sec.~\ref{TTOEP} are based on two main operations: Fourier transform and discrete convolution.
The radix-2 reccurent relation which was known to Gauss~\cite{gauss-fft} and lays behind the famous Cooley-Tuckey FFT algorithm~\cite{cooleytukey-fft} perfectly matches the multilevel structure of QTT format, resulting in the QTT--FFT algorithm~\cite{dks-ttfft-2012}.
For a vector of size $n=2^d$ given approximately in QTT format~\eqref{eq:qtt}, the QTT--FFT computes the Fourier transform  with complexity~$\O(d^2 R^3),$ where $R$ is the maximum QTT rank of the input vector, Fourier image, and all intermediate vectors of the algorithm.

The discrete convolution, i.e., multiplication by Toeplitz matrix, can be performed by three Fourier transforms with complexity $\O(d^2 R^3),$ where $R$ bounds the QTT ranks of both vectors to convolve as well as their Fourier images.
As shown in~\cite{khkaz-conv-2011}, the convolution $c=a \conv b$ of two vectors with QTT ranks bounded by $r_a$ and $r_b,$ can be written in QTT form with QTT ranks bounded by $2r_ar_b.$
This representation has large QTT ranks, which can be reduced to the value bounded by $r_c \leq 2r_ar_b$ using some TT--truncation algorithm.
 We can use SVD--based algorithm proposed in~\cite{osel-tt-2011} or iterative DMRG--type approach proposed in~\cite{Os-mvk2-2011}, resulting in convolution algorithms with $\O(d r_a^3r_b^3)$ and $\O(d (r_a+r_b+r_c) r_a r_b r_c)$ complexity, respectively.
If $r_a \approx r_b \approx r_c \approx R,$ the QTT--FFT and DMRG--based convolution algorithms have complexity $\O(d^2 R^3)$ and $\O(d R^4),$ respectively. 
Therefore, we can not say in general which approach is better, even in the simplest case of almost equal QTT ranks. 
This will be established in numerical experiments.

\subsection{Shifts of vectors in QTT format}
The convolution algorithm proposed in~\cite{khkaz-conv-2011} is based on the remarkable property of shift matrices $L \in \T_{2^d}$ and $U = L^\t,$ where the first column of $L$ is $l=\left(0,1,0,0,\ldots,0\right)^\t.$
It is shown in~\cite{khkaz-conv-2011} that all matrices $L^p,$ $p=1,\ldots,2^d-1$ have all QTT ranks two. 
Hence, if a vector $a$ has QTT ranks $r_1,\ldots,r_{d-1},$ then the right shifted vector $b=L^p a$ for all $p$ has QTT ranks not larger than $2r_1, \ldots, 2r_{d-1}.$
The same holds for left shifts $c=U^p a.$

In the following theorem we improve this result for vectors, shifted by one element.

\begin{theorem}
 Let $a=\left[a(k)\right]_{k=0}^{2^d-1}$ has the QTT representation~\eqref{eq:qtt}, then the vector 
\begin{equation}\nonumber
 b=\begin{bmatrix} x & a(0) & \ldots & a(2^d-2)\end{bmatrix}^\t
\end{equation}
has the QTT representation $b(k) = b(\overline{k_1 \ldots k_d}) = B^{(1)}_{k_1} \ldots B^{(d)}_{k_d}$ with the following cores
\begin{equation}\label{eq:push}
 \begin{split}
  B^{(1)}_0 = \begin{bmatrix} \phantom{A_0^{(1)}} & 1 \end{bmatrix}, \qquad
  B^{(p)}_0 = \begin{bmatrix} A_0^{(p)} &  \\ \phantom{b_p} & 1 \end{bmatrix}, \qquad
  B^{(d)}_0 = \begin{bmatrix} A_0^{(d)} \\  x \end{bmatrix}, \\
  B^{(1)}_1 = \begin{bmatrix} A_0^{(1)} & \phantom{1} \end{bmatrix}, \qquad
  B^{(p)}_1 = \begin{bmatrix} A_1^{(p)} &  \\ b_p & \phantom{1} \end{bmatrix}, \qquad
  B^{(d)}_1 = \begin{bmatrix} A_1^{(d)} \\ b_d \end{bmatrix},
 \end{split}
\end{equation}
where $p=2,\ldots,d-1$ and  $b_q = A^{(1)}_1 \ldots A^{(q-1)}_1 A^{(q)}_0$ for $q=2,\ldots,d.$ 
Similarly, the vector 
\begin{equation}\nonumber
 c=\begin{bmatrix} a(1) & \ldots & a(2^d-1) & y \end{bmatrix}^\t
\end{equation}
has the QTT representation $c(k) = c(\overline{k_1 \ldots k_d}) = C^{(1)}_{k_1} \ldots C^{(d)}_{k_d}$ with the following cores
\begin{equation}\label{eq:pull}
 \begin{split}
  C^{(1)}_0 = \begin{bmatrix} A_1^{(1)} & \phantom{1} \end{bmatrix}, \qquad
  C^{(p)}_0 = \begin{bmatrix} A_0^{(p)} &  \\ c_p & \phantom{1} \end{bmatrix}, \qquad
  C^{(d)}_0 = \begin{bmatrix} A_0^{(d)} \\ c_d  \end{bmatrix}, \\
  C^{(1)}_1 = \begin{bmatrix} \phantom{A_0^{(1)}} & 1 \end{bmatrix}, \qquad
  C^{(p)}_1 = \begin{bmatrix} A_1^{(p)} &  \\ & 1 \end{bmatrix}, \qquad
  C^{(d)}_1 = \begin{bmatrix} A_1^{(d)} \\ y \end{bmatrix},
 \end{split}
\end{equation}
where $p=2,\ldots,d-1$ and  $c_q = A^{(1)}_0 \ldots A^{(q-1)}_0 A^{(q)}_1$ for $q=2,\ldots,d.$ 
\end{theorem}
\begin{proof}
We check~\eqref{eq:push} straightforwardly. For $k=0$ it holds
\begin{equation}\nonumber
b(0) = B^{(1)}_{0} \ldots B^{(d)}_{0} = 
       \begin{bmatrix} \phantom{A_0^{(1)}} & 1 \end{bmatrix} \: \ldots \:
       \begin{bmatrix} A_0^{(p)} &  \\ \phantom{b_p} & 1 \end{bmatrix} \: \ldots \:
       \begin{bmatrix} A_0^{(d)} \\  x \end{bmatrix} = x
\end{equation}
For $k=\overline{k_1k_2\ldots k_d}$ with $k_1=1$ it holds
\begin{equation}\nonumber
 \begin{split}
  b(k) & = b(\overline{1 k_2\ldots k_d}) = B^{(1)}_{1} B^{(2)}_{k_2} \ldots B^{(d)}_{k_d} = 
       \begin{bmatrix} A_0^{(1)} & \phantom{1} \end{bmatrix} \: 
       \begin{bmatrix} A_{k_2}^{(2)} &  \\ * & * \end{bmatrix} \: \ldots \:
       \begin{bmatrix} A_{k_d}^{(d)} \\  * \end{bmatrix} 
       \\ & = A_0^{(1)} A_{k_2}^{(2)} \ldots A_{k_d}^{(d)} = a(\overline{0 k_2 \ldots k_d}) = a(k-1),
 \end{split}       
\end{equation}
where ``$*$'' denotes arbitrarily zero or non-zero element.
For $k=\overline{k_1k_2k_3\ldots k_d}$ with $k_1=0$ and $k_2=1$ it holds
\begin{equation}\nonumber
 \begin{split}
  b(k) & = b(\overline{01 k_3\ldots k_d}) =  B^{(1)}_{0} B^{(2)}_{1} B^{(3)}_{k_3} \ldots B^{(d)}_{k_d} = 
       \begin{bmatrix} \phantom{A_0^{(1)}} & {1} \end{bmatrix} \: 
       \begin{bmatrix} A_1^{(2)} &  \\ {b_2} & \phantom{1} \end{bmatrix} \: 
       \begin{bmatrix} A_{k_3}^{(3)} &  \\ * & * \end{bmatrix} \: \ldots \:
       \begin{bmatrix} A_{k_d}^{(d)} \\  * \end{bmatrix} \\
       & = b_2 A_{k_3}^{(3)} \ldots A_{k_d}^{(d)} = A_1^{(1)} A_0^{(2)} A_{k_3}^{(3)} \ldots A_{k_d}^{(d)} = a(\overline{10 k_3\ldots k_d}) = a(k-1).
 \end{split}
\end{equation}
Finally, for $k=\overline{k_1k_2k_3\ldots k_d}$ with $k_1=\ldots=k_{p-1}=0$ and $k_p=1$ it holds
\begin{equation}\nonumber
 \begin{split}
  b(k) & = b(\overline{\underbrace{0 \ldots 0}_{p-1\:\textrm{zeros}}1 k_{p+1}\ldots k_d}) = B^{(1)}_{0}B^{(2)}_{0}  \ldots B^{(p-1)}_{0} B^{(p)}_{1} B^{(p+1)}_{k_{p+1}} \ldots B^{(d)}_{k_d} 
       \\ & = 
       \begin{bmatrix} \phantom{A_0^{(1)}} & {1} \end{bmatrix} \: 
       \begin{bmatrix} A_0^{(2)} &  \\ \phantom{b_2} & {1} \end{bmatrix} \: \ldots
       \begin{bmatrix} A_0^{(p-1)} &  \\ \phantom{b_2} & {1} \end{bmatrix} \: 
       \begin{bmatrix} A_{1}^{(p)} &  \\ b_p &  \phantom{1} \end{bmatrix} \: 
       \begin{bmatrix} A_{k_{p+1}}^{(p+1)} &  \\ * & * \end{bmatrix} \: \ldots \: \ldots
       \begin{bmatrix} A_{k_d}^{(d)} \\  * \end{bmatrix}
       \\ & = b_p A_{k_{p+1}}^{(p+1)} \ldots A_{k_d}^{(d)} = A_1^{(1)} \ldots A_1^{(p-1)} A_0^{(p)} A_{k_{p+1}}^{(p+1)} \ldots A_{k_d}^{(d)} 
       \\ & = a(\overline{\underbrace{1 \ldots 1}_{p-1\:\textrm{ones}}0 k_{p+1}\ldots k_d})
            = a(k-1).
 \end{split}
\end{equation}
Equation~\eqref{eq:pull} is verified in the same way.
\end{proof}

\subsection{Divide and conquer algorithm in QTT format}

We are now ready to present the version of divide and conquer algorithm which operates with data given approximately in QTT format.
Let $n=2^d$ and consider $A\in\T_n$ defined by the first column $a(k)$ which is represented in the QTT format~\eqref{eq:qtt}.
As previously, let $A_t$ denote $2^t \times 2^t$ leading submatrix of $A.$
For small $d_0$ we can invert $A_{d_0}$ using standard divide and conquer method and approximate the first column of $A_{d_0}^{-1}$ in QTT format using SVD-based algorithm~\cite{osel-tt-2011}.

Now suppose that for some $t$ the first column of $A_t^{-1}$ is computed in QTT format, and we have to compute the QTT approximation of the first column of $A_{t+1}^{-1},$ using the recursion~\eqref{eq:dc}.
It is necessary to describe the Toeplitz matrix $C$ which lies in the lower part of $A_{t+1}.$
The first column of $C$ is $c_+ = [a(2^t), a(2^t+1), \ldots, a(2^{t+1}-1)]^\t$ and has the following QTT representation
\begin{equation}\label{eq:col}
c_+(\overline{k_1 \ldots k_t}) = a(\overline{k_1 \ldots k_t} + 2^t) = A^{(1)}_{k_1} A^{(2)}_{k_2} \ldots A^{(t)}_{k_t} A^{(t+1)}_1 A^{(t+2)}_0 \ldots A^{(d)}_0.
\end{equation}
The first row of $C$ is $c_- = [a(2^t), a(2^t-1), \ldots, a(1)]^\t.$
To construct the QTT representation for $c_-,$ first write the QTT representation for $a=[a(0), \ldots, a(2^t-1)]^\t,$ which is
$$
a(\overline{k_1 \ldots k_t}) = A^{(1)}_{k_1} A^{(2)}_{k_2} \ldots A^{(t)}_{k_t} A^{(t+1)}_0 A^{(t+2)}_0 \ldots A^{(d)}_0.
$$
Then apply the `pull' operation and construct QTT format for $a'=[a(1), \ldots, a(2^t)]^\t$ as follows
$$
a'(\overline{k_1 \ldots k_t}) = C^{(1)}_{k_1} C^{(2)}_{k_2} \ldots C^{(t)}_{k_t} A^{(t+1)}_0 A^{(t+2)}_0 \ldots A^{(d)}_0,
$$
where TT--cores $C_{k_q}^{(q)}$ are defined by~\eqref{eq:pull}. 
Finally, revert the ordering of elements in the vector $a'$ to obtain the QTT format for $c_-$ as follows
\begin{equation}\label{eq:row}
c_-(\overline{k_1 \ldots k_t}) = C^{(1)}_{1-k_1} C^{(2)}_{1-k_2} \ldots C^{(t)}_{1-k_t} A^{(t+1)}_0 A^{(t+2)}_0 \ldots A^{(d)}_0.
\end{equation}

We summarize the above steps in Alg.~\ref{alg:dc}. 
Note that the workhorse of divide and conquer method is the discrete convolution in QTT format, which can be performed by two different methods.
This results in two variants of algorithms with different performance, which will be studied in numerical experiments.

\begin{algorithm}[t]
 \caption{Divide and conquer in QTT format} \label{alg:dc}
 \begin{algorithmic}[1]
  \REQUIRE{$A\in\T_n,$ $n=2^d,$ given by vector $[a(k)]_{k=0}^{n-1}$ in QTT format~\eqref{eq:qtt}}
  \ENSURE{$B=A^{-1}\in\T_n$ given in QTT format}
  \STATE For small $d_0,$ compute the first column of $2^{d_0}\times 2^{d_0}$ leading submatrix $A_{d_0}.$
  Compute  $B_{d_0}=A_{d_0}^{-1}$ by~\eqref{eq:inv} and approximate in in QTT format by TT--SVD algorithm~\cite{osel-tt-2011}.
  \FOR{$t=d_0,\ldots,d-1$}
   \STATE Compute the first row and column of matrix $C_t$ in~\eqref{eq:dc} in QTT format by~\eqref{eq:col} and~\eqref{eq:row}.
   \STATE Compute the first column of $B_t C_t B_t$ by two convolutions in QTT format, see Sec.~\ref{QTTconv}.
   \STATE Combine the first column of $B_t$ and first column of $B_t C_t B_t$ given in QTT format as follows
   \begin{equation}\nonumber
    b(\overline{k_1\ldots k_t}) = B^{(1)}_{k_1} \ldots B^{(t)}_{k_1}, \qquad 
    g(\overline{k_1\ldots k_t}) = G^{(1)}_{k_1} \ldots G^{(t)}_{k_1},
   \end{equation}
   to the single vector $b'$ in QTT format, which is defined as follows
   \begin{equation}\nonumber
    b'(\overline{k_1\ldots k_t k_{t+1}}) = 
     \begin{bmatrix} B^{(1)}_{k_1} & G^{(1)}_{k_1}\end{bmatrix} \:
     \begin{bmatrix} B^{(2)}_{k_2} & \\ & G^{(2)}_{k_2}\end{bmatrix} \: 
     \ldots \:
     \begin{bmatrix} B^{(t)}_{k_t} & \\ & G^{(t)}_{k_t}\end{bmatrix} \: 
     \begin{bmatrix} 1-k_{t+1} \\  k_{t+1}\end{bmatrix} \: 
   \end{equation}
   \STATE Apply TT--truncate algorithm to $b'$ to reduce the ranks of QTT representation.
  \ENDFOR
 \end{algorithmic}
\end{algorithm}

\subsection{Modified Bini's algorithm in QTT format}
The implementation of~\eqref{eq:bini} in the QTT format is very straightforward.
It is enough to mention that the QTT format of vector $[\eps^j]_{j=0}^{n-1},$ $n=2^d,$ has QTT--ranks one~(see~\cite{khor-qtt-2011}), since
$$
\eps^j = \eps^{\overline{j_1j_2\ldots j_d}} = \eps^{j_1} \eps^{2j_2} \ldots \eps^{2^{d-1}j_d}.
$$
Therefore, multiplication of a vector in QTT format by diagonal matrix $D_\eps$ requires only the appropriate scaling of TT--cores.
By Alg.~\ref{alg:bi} we present the QTT version of the modified Bini's algorithm~\cite[Alg. 2]{mng-bini-2004}.
The algorithm includes two Fourier transforms in the QTT format which can not be substituted by discrete convolution.
 Note that this algorithm contains two approximation errors:
 \begin{itemize}
  \item The first comes form original approximation of triangular Toeplitz matrix $A$ by the diagonally scaled circulant matrix $A_\eps.$
  The accuracy of this approximation is governed by parameter $\eps.$ 
  According to the numerical tests made by the authors of~\cite{mng-bini-2004}, the good choice for Bini's and modified Bini's methods are $\eps^n = 0.5 \times 10^{-8}$ and $\eps^n=10^{-5},$ respectively.
  \item The second error comes from TT--truncation algorithm applied in the QTT--FFT algorithm and on each step of Newton iteration.
  The threshold parameter of TT--truncation should be usually smaller than $\eps^n$ in order to maintain the accuracy of the result after diagonal scaling.  
 \end{itemize}

\begin{algorithm}[t]
 \caption{Modified Bini's method in QTT format} \label{alg:bi}
 \begin{algorithmic}[1]
  \REQUIRE{$A\in\T_n,$ $n=2^d,$ given by vector $[a(k)]_{k=0}^{n-1}$ in QTT format~\eqref{eq:qtt}}
  \ENSURE{$B_\eps = A^{-1}_\eps \approx A^{-1}\in\T_n$ given in QTT format}
  \STATE Choose $0< \eps < 1$ and let $\hat a(k) = \eps^k a(k)$ for $k=0,\ldots,n-1$ and $\hat a(k)=0$ for $k=n,\ldots,2n-1.$
  The QTT representation of $\hat a$ is the following
  $$
  \hat a(k) = \hat a(\overline{k_1\ldots k_d k_{d+1}}) = \hat A^{(1)}_{k_1} \ldots \hat A^{(d)}_{k_d} (1-k_{d+1}), \qquad
  \hat A^{(p)}_{k_p} = \eps^{2^{p-1} k_p} A^{(p)}_{k_p}, \quad p=1,\ldots,d.
  $$
  \STATE Apply QTT--FFT~\cite{dks-ttfft-2012} to compute the size--$2n$ Fourier transform $\lambda = \sqrt{2n} F \hat a.$
  \STATE Apply Newton iteration~\eqref{eq:nw} to compute $c=\lambda^{-1}$ in the QTT format.
  Each iteration step includes the pointwise (Hadamard) multiplication of vectors in QTT format and TT--truncation to reduce the QTT--ranks.
  \STATE  Apply QTT--FFT again to compute the size--$2n$ Fourier transform $\hat b =  F^* c  / \sqrt{2n}$ in the QTT format
  $$
  \hat b(k) = \hat b(\overline{k_1 \ldots k_d k_{d+1}}) = \hat B^{(1)}_{k_1} \ldots \hat B^{(d)}_{k_d} \hat B^{(d+1)}_{k_{d+1}}.
  $$
  \STATE The QTT representation of the first column of $B_\eps$ is the following
  $$
  b_\eps(k) = b_\eps(\overline{k_1\ldots k_d}) = B^{(1)}_{k_1} \ldots B^{(d)}_{k_d} \hat B^{(d+1)}_{0}, \qquad  
  B^{(p)}_{k_p} = \eps^{-2^{p-1} k_p}  \hat B^{(p)}_{k_p}, \quad p=1,\ldots,d.
  $$
 \end{algorithmic}
\end{algorithm}

\section{Numerical experiments} \label{NUM}
\subsection{Timings of inversion algorithms}
\begin{figure}[p]
 \begin{center} \hfil
  \resizebox{.48\textwidth}{!}{\begin{tikzpicture}[gnuplot]
\gpcolor{gp lt color border}
\gpsetlinetype{gp lt border}
\gpsetlinewidth{1.00}
\draw[gp path] (0.828,0.616)--(1.008,0.616);
\node[gp node right] at (0.644,0.616) {$-2$};
\draw[gp path] (0.828,2.185)--(1.008,2.185);
\node[gp node right] at (0.644,2.185) {$-1$};
\draw[gp path] (0.828,3.753)--(1.008,3.753);
\node[gp node right] at (0.644,3.753) {$0$};
\draw[gp path] (0.828,5.322)--(1.008,5.322);
\node[gp node right] at (0.644,5.322) {$1$};
\draw[gp path] (0.828,6.891)--(1.008,6.891);
\node[gp node right] at (0.644,6.891) {$2$};
\draw[gp path] (0.828,0.616)--(0.828,0.796);
\node[gp node center] at (0.828,0.308) {$10$};
\draw[gp path] (3.754,0.616)--(3.754,0.796);
\node[gp node center] at (3.754,0.308) {$15$};
\draw[gp path] (6.680,0.616)--(6.680,0.796);
\node[gp node center] at (6.680,0.308) {$20$};
\draw[gp path] (9.606,0.616)--(9.606,0.796);
\node[gp node center] at (9.606,0.308) {$25$};
\draw[gp path] (0.828,8.381)--(0.828,0.616)--(11.947,0.616);
\node[gp node left] at (0.939,8.303) {$\log_{10}(\mathrm{time},$ sec.)};
\node[gp node right] at (12.614,0.305) {$\log_2n$};
\gpcolor{\gprgb{0}{0}{0}}
\gpsetlinetype{gp lt plot 0}
\gpsetlinewidth{3.00}
\draw[gp path] (0.828,1.664)--(1.998,2.047)--(3.169,2.342)--(4.339,2.650)--(5.510,2.838)%
  --(6.680,3.030)--(7.851,3.174)--(9.021,3.320)--(10.191,3.436)--(11.362,3.550);
\gpsetpointsize{6.00}
\gppoint{gp mark 3}{(0.828,1.664)}
\gppoint{gp mark 3}{(1.998,2.047)}
\gppoint{gp mark 3}{(3.169,2.342)}
\gppoint{gp mark 3}{(4.339,2.650)}
\gppoint{gp mark 3}{(5.510,2.838)}
\gppoint{gp mark 3}{(6.680,3.030)}
\gppoint{gp mark 3}{(7.851,3.174)}
\gppoint{gp mark 3}{(9.021,3.320)}
\gppoint{gp mark 3}{(10.191,3.436)}
\gppoint{gp mark 3}{(11.362,3.550)}
\gpsetlinewidth{6.00}
\draw[gp path] (0.828,2.337)--(1.998,3.300)--(3.169,4.252)--(4.339,4.703)--(5.510,5.263)%
  --(6.680,5.393)--(7.851,5.329)--(9.021,5.556)--(10.191,5.474)--(11.362,5.580);
\gpsetpointsize{8.00}
\gppoint{gp mark 3}{(0.828,2.337)}
\gppoint{gp mark 3}{(1.998,3.300)}
\gppoint{gp mark 3}{(3.169,4.252)}
\gppoint{gp mark 3}{(4.339,4.703)}
\gppoint{gp mark 3}{(5.510,5.263)}
\gppoint{gp mark 3}{(6.680,5.393)}
\gppoint{gp mark 3}{(7.851,5.329)}
\gppoint{gp mark 3}{(9.021,5.556)}
\gppoint{gp mark 3}{(10.191,5.474)}
\gppoint{gp mark 3}{(11.362,5.580)}
\gpsetlinetype{gp lt plot 1}
\gpsetlinewidth{3.00}
\draw[gp path] (0.828,1.192)--(1.998,1.382)--(3.169,1.522)--(4.339,1.630)--(5.510,1.726)%
  --(6.680,1.787)--(7.851,1.851)--(9.021,1.899)--(10.191,1.933)--(11.362,1.946);
\gpsetpointsize{6.00}
\gppoint{gp mark 7}{(0.828,1.192)}
\gppoint{gp mark 7}{(1.998,1.382)}
\gppoint{gp mark 7}{(3.169,1.522)}
\gppoint{gp mark 7}{(4.339,1.630)}
\gppoint{gp mark 7}{(5.510,1.726)}
\gppoint{gp mark 7}{(6.680,1.787)}
\gppoint{gp mark 7}{(7.851,1.851)}
\gppoint{gp mark 7}{(9.021,1.899)}
\gppoint{gp mark 7}{(10.191,1.933)}
\gppoint{gp mark 7}{(11.362,1.946)}
\gpcolor{\gprgb{448}{448}{448}}
\gpsetlinetype{gp lt plot 0}
\gpsetlinewidth{7.00}
\draw[gp path] (3.967,0.616)--(4.339,0.961)--(4.924,1.533)--(5.510,2.113)--(6.095,2.773)%
  --(6.680,3.335)--(7.265,3.905)--(7.851,4.430)--(8.436,4.954)--(9.021,5.469)--(9.606,5.989)%
  --(10.191,6.496)--(10.777,7.080)--(11.362,7.531);
\gpsetlinetype{gp lt plot 1}
\gpsetlinewidth{5.00}
\draw[gp path] (3.730,0.616)--(3.754,0.636)--(4.339,0.989)--(4.924,1.696)--(5.510,2.198)%
  --(6.095,2.496)--(6.680,3.026)--(7.265,3.587)--(7.851,4.069)--(8.436,4.618)--(9.021,5.086)%
  --(9.606,5.616)--(10.191,6.208)--(10.777,6.858)--(11.362,7.202);
\gpcolor{gp lt color border}
\gpsetlinetype{gp lt border}
\gpsetlinewidth{1.00}
\draw[gp path] (0.828,8.381)--(0.828,0.616)--(11.947,0.616);
\gpdefrectangularnode{gp plot 1}{\pgfpoint{0.828cm}{0.616cm}}{\pgfpoint{11.947cm}{8.381cm}}
\end{tikzpicture}
} \hfil
  \resizebox{.48\textwidth}{!}{\begin{tikzpicture}[gnuplot]
\gpcolor{gp lt color border}
\gpsetlinetype{gp lt border}
\gpsetlinewidth{1.00}
\draw[gp path] (0.828,0.616)--(1.008,0.616);
\node[gp node right] at (0.644,0.616) {$-2$};
\draw[gp path] (0.828,2.185)--(1.008,2.185);
\node[gp node right] at (0.644,2.185) {$-1$};
\draw[gp path] (0.828,3.753)--(1.008,3.753);
\node[gp node right] at (0.644,3.753) {$0$};
\draw[gp path] (0.828,5.322)--(1.008,5.322);
\node[gp node right] at (0.644,5.322) {$1$};
\draw[gp path] (0.828,6.891)--(1.008,6.891);
\node[gp node right] at (0.644,6.891) {$2$};
\draw[gp path] (0.828,0.616)--(0.828,0.796);
\node[gp node center] at (0.828,0.308) {$10$};
\draw[gp path] (3.754,0.616)--(3.754,0.796);
\node[gp node center] at (3.754,0.308) {$15$};
\draw[gp path] (6.680,0.616)--(6.680,0.796);
\node[gp node center] at (6.680,0.308) {$20$};
\draw[gp path] (9.606,0.616)--(9.606,0.796);
\node[gp node center] at (9.606,0.308) {$25$};
\draw[gp path] (0.828,8.381)--(0.828,0.616)--(11.947,0.616);
\node[gp node left] at (0.939,8.303) {$\log_{10}(\mathrm{time},$ sec.)};
\node[gp node right] at (12.614,0.305) {$\log_2n$};
\gpcolor{\gprgb{0}{0}{0}}
\gpsetlinetype{gp lt plot 0}
\gpsetlinewidth{3.00}
\draw[gp path] (0.828,1.692)--(1.998,2.067)--(3.169,2.358)--(4.339,2.581)--(5.510,2.755)%
  --(6.680,2.908)--(7.851,2.975)--(9.021,3.100)--(10.191,3.229)--(11.362,3.331);
\gpsetpointsize{6.00}
\gppoint{gp mark 3}{(0.828,1.692)}
\gppoint{gp mark 3}{(1.998,2.067)}
\gppoint{gp mark 3}{(3.169,2.358)}
\gppoint{gp mark 3}{(4.339,2.581)}
\gppoint{gp mark 3}{(5.510,2.755)}
\gppoint{gp mark 3}{(6.680,2.908)}
\gppoint{gp mark 3}{(7.851,2.975)}
\gppoint{gp mark 3}{(9.021,3.100)}
\gppoint{gp mark 3}{(10.191,3.229)}
\gppoint{gp mark 3}{(11.362,3.331)}
\gpsetlinewidth{6.00}
\draw[gp path] (0.828,2.303)--(1.998,3.141)--(3.169,3.831)--(4.339,4.561)--(5.510,4.611)%
  --(6.680,4.447)--(7.851,4.618)--(9.021,4.502)--(10.191,4.443)--(11.362,4.380);
\gpsetpointsize{8.00}
\gppoint{gp mark 3}{(0.828,2.303)}
\gppoint{gp mark 3}{(1.998,3.141)}
\gppoint{gp mark 3}{(3.169,3.831)}
\gppoint{gp mark 3}{(4.339,4.561)}
\gppoint{gp mark 3}{(5.510,4.611)}
\gppoint{gp mark 3}{(6.680,4.447)}
\gppoint{gp mark 3}{(7.851,4.618)}
\gppoint{gp mark 3}{(9.021,4.502)}
\gppoint{gp mark 3}{(10.191,4.443)}
\gppoint{gp mark 3}{(11.362,4.380)}
\gpsetlinetype{gp lt plot 1}
\gpsetlinewidth{3.00}
\draw[gp path] (0.828,1.371)--(1.998,1.525)--(3.169,1.645)--(4.339,1.714)--(5.510,1.761)%
  --(6.680,1.717)--(7.851,1.700)--(9.021,1.748)--(10.191,1.756)--(11.362,1.662);
\gpsetpointsize{6.00}
\gppoint{gp mark 7}{(0.828,1.371)}
\gppoint{gp mark 7}{(1.998,1.525)}
\gppoint{gp mark 7}{(3.169,1.645)}
\gppoint{gp mark 7}{(4.339,1.714)}
\gppoint{gp mark 7}{(5.510,1.761)}
\gppoint{gp mark 7}{(6.680,1.717)}
\gppoint{gp mark 7}{(7.851,1.700)}
\gppoint{gp mark 7}{(9.021,1.748)}
\gppoint{gp mark 7}{(10.191,1.756)}
\gppoint{gp mark 7}{(11.362,1.662)}
\gpcolor{\gprgb{448}{448}{448}}
\gpsetlinetype{gp lt plot 0}
\gpsetlinewidth{7.00}
\draw[gp path] (3.967,0.616)--(4.339,0.961)--(4.924,1.533)--(5.510,2.113)--(6.095,2.773)%
  --(6.680,3.335)--(7.265,3.905)--(7.851,4.430)--(8.436,4.954)--(9.021,5.469)--(9.606,5.989)%
  --(10.191,6.496)--(10.777,7.080)--(11.362,7.531);
\gpsetlinetype{gp lt plot 1}
\gpsetlinewidth{5.00}
\draw[gp path] (3.730,0.616)--(3.754,0.636)--(4.339,0.989)--(4.924,1.696)--(5.510,2.198)%
  --(6.095,2.496)--(6.680,3.026)--(7.265,3.587)--(7.851,4.069)--(8.436,4.618)--(9.021,5.086)%
  --(9.606,5.616)--(10.191,6.208)--(10.777,6.858)--(11.362,7.202);
\gpcolor{gp lt color border}
\gpsetlinetype{gp lt border}
\gpsetlinewidth{1.00}
\draw[gp path] (0.828,8.381)--(0.828,0.616)--(11.947,0.616);
\gpdefrectangularnode{gp plot 1}{\pgfpoint{0.828cm}{0.616cm}}{\pgfpoint{11.947cm}{8.381cm}}
\end{tikzpicture}
} \hfil
 \end{center}
 \begin{center} \hfil
  \resizebox{.48\textwidth}{!}{\begin{tikzpicture}[gnuplot]
\gpcolor{gp lt color border}
\gpsetlinetype{gp lt border}
\gpsetlinewidth{1.00}
\draw[gp path] (0.828,0.616)--(1.008,0.616);
\node[gp node right] at (0.644,0.616) {$-2$};
\draw[gp path] (0.828,2.185)--(1.008,2.185);
\node[gp node right] at (0.644,2.185) {$-1$};
\draw[gp path] (0.828,3.753)--(1.008,3.753);
\node[gp node right] at (0.644,3.753) {$0$};
\draw[gp path] (0.828,5.322)--(1.008,5.322);
\node[gp node right] at (0.644,5.322) {$1$};
\draw[gp path] (0.828,6.891)--(1.008,6.891);
\node[gp node right] at (0.644,6.891) {$2$};
\draw[gp path] (0.828,0.616)--(0.828,0.796);
\node[gp node center] at (0.828,0.308) {$10$};
\draw[gp path] (3.754,0.616)--(3.754,0.796);
\node[gp node center] at (3.754,0.308) {$15$};
\draw[gp path] (6.680,0.616)--(6.680,0.796);
\node[gp node center] at (6.680,0.308) {$20$};
\draw[gp path] (9.606,0.616)--(9.606,0.796);
\node[gp node center] at (9.606,0.308) {$25$};
\draw[gp path] (0.828,8.381)--(0.828,0.616)--(11.947,0.616);
\node[gp node left] at (0.939,8.303) {$\log_{10}(\mathrm{time},$ sec.)};
\node[gp node right] at (12.614,0.305) {$\log_2n$};
\gpcolor{\gprgb{0}{0}{0}}
\gpsetlinetype{gp lt plot 0}
\gpsetlinewidth{3.00}
\draw[gp path] (0.828,1.804)--(1.998,2.185)--(3.169,2.542)--(4.339,2.825)--(5.510,3.038)%
  --(6.680,3.239)--(7.851,3.395)--(9.021,3.541)--(10.191,3.689)--(11.362,3.793);
\gpsetpointsize{6.00}
\gppoint{gp mark 3}{(0.828,1.804)}
\gppoint{gp mark 3}{(1.998,2.185)}
\gppoint{gp mark 3}{(3.169,2.542)}
\gppoint{gp mark 3}{(4.339,2.825)}
\gppoint{gp mark 3}{(5.510,3.038)}
\gppoint{gp mark 3}{(6.680,3.239)}
\gppoint{gp mark 3}{(7.851,3.395)}
\gppoint{gp mark 3}{(9.021,3.541)}
\gppoint{gp mark 3}{(10.191,3.689)}
\gppoint{gp mark 3}{(11.362,3.793)}
\gpsetlinewidth{6.00}
\draw[gp path] (0.828,2.501)--(1.998,3.321)--(3.169,3.997)--(4.339,4.515)--(5.510,4.954)%
  --(6.680,5.241)--(7.851,5.546)--(9.021,5.814)--(10.191,6.233)--(11.362,6.467);
\gpsetpointsize{8.00}
\gppoint{gp mark 3}{(0.828,2.501)}
\gppoint{gp mark 3}{(1.998,3.321)}
\gppoint{gp mark 3}{(3.169,3.997)}
\gppoint{gp mark 3}{(4.339,4.515)}
\gppoint{gp mark 3}{(5.510,4.954)}
\gppoint{gp mark 3}{(6.680,5.241)}
\gppoint{gp mark 3}{(7.851,5.546)}
\gppoint{gp mark 3}{(9.021,5.814)}
\gppoint{gp mark 3}{(10.191,6.233)}
\gppoint{gp mark 3}{(11.362,6.467)}
\gpsetlinetype{gp lt plot 1}
\gpsetlinewidth{3.00}
\draw[gp path] (0.828,2.670)--(1.998,3.048)--(3.169,3.262)--(4.339,3.409)--(5.510,3.508)%
  --(6.680,3.606)--(7.851,3.693)--(9.021,3.760)--(10.191,3.818)--(11.362,3.932);
\gpsetpointsize{6.00}
\gppoint{gp mark 7}{(0.828,2.670)}
\gppoint{gp mark 7}{(1.998,3.048)}
\gppoint{gp mark 7}{(3.169,3.262)}
\gppoint{gp mark 7}{(4.339,3.409)}
\gppoint{gp mark 7}{(5.510,3.508)}
\gppoint{gp mark 7}{(6.680,3.606)}
\gppoint{gp mark 7}{(7.851,3.693)}
\gppoint{gp mark 7}{(9.021,3.760)}
\gppoint{gp mark 7}{(10.191,3.818)}
\gppoint{gp mark 7}{(11.362,3.932)}
\gpcolor{\gprgb{448}{448}{448}}
\gpsetlinetype{gp lt plot 0}
\gpsetlinewidth{7.00}
\draw[gp path] (3.967,0.616)--(4.339,0.961)--(4.924,1.533)--(5.510,2.113)--(6.095,2.773)%
  --(6.680,3.335)--(7.265,3.905)--(7.851,4.430)--(8.436,4.954)--(9.021,5.469)--(9.606,5.989)%
  --(10.191,6.496)--(10.777,7.080)--(11.362,7.531);
\gpsetlinetype{gp lt plot 1}
\gpsetlinewidth{5.00}
\draw[gp path] (3.730,0.616)--(3.754,0.636)--(4.339,0.989)--(4.924,1.696)--(5.510,2.198)%
  --(6.095,2.496)--(6.680,3.026)--(7.265,3.587)--(7.851,4.069)--(8.436,4.618)--(9.021,5.086)%
  --(9.606,5.616)--(10.191,6.208)--(10.777,6.858)--(11.362,7.202);
\gpcolor{gp lt color border}
\gpsetlinetype{gp lt border}
\gpsetlinewidth{1.00}
\draw[gp path] (0.828,8.381)--(0.828,0.616)--(11.947,0.616);
\gpdefrectangularnode{gp plot 1}{\pgfpoint{0.828cm}{0.616cm}}{\pgfpoint{11.947cm}{8.381cm}}
\end{tikzpicture}
} \hfil
  \resizebox{.48\textwidth}{!}{\begin{tikzpicture}[gnuplot]
\gpcolor{gp lt color border}
\gpsetlinetype{gp lt border}
\gpsetlinewidth{1.00}
\draw[gp path] (0.828,0.616)--(1.008,0.616);
\node[gp node right] at (0.644,0.616) {$-2$};
\draw[gp path] (0.828,2.185)--(1.008,2.185);
\node[gp node right] at (0.644,2.185) {$-1$};
\draw[gp path] (0.828,3.753)--(1.008,3.753);
\node[gp node right] at (0.644,3.753) {$0$};
\draw[gp path] (0.828,5.322)--(1.008,5.322);
\node[gp node right] at (0.644,5.322) {$1$};
\draw[gp path] (0.828,6.891)--(1.008,6.891);
\node[gp node right] at (0.644,6.891) {$2$};
\draw[gp path] (0.828,0.616)--(0.828,0.796);
\node[gp node center] at (0.828,0.308) {$10$};
\draw[gp path] (3.754,0.616)--(3.754,0.796);
\node[gp node center] at (3.754,0.308) {$15$};
\draw[gp path] (6.680,0.616)--(6.680,0.796);
\node[gp node center] at (6.680,0.308) {$20$};
\draw[gp path] (9.606,0.616)--(9.606,0.796);
\node[gp node center] at (9.606,0.308) {$25$};
\draw[gp path] (0.828,8.381)--(0.828,0.616)--(11.947,0.616);
\node[gp node left] at (0.939,8.303) {$\log_{10}(\mathrm{time},$ sec.)};
\node[gp node right] at (12.614,0.305) {$\log_2n$};
\gpcolor{\gprgb{0}{0}{0}}
\gpsetlinetype{gp lt plot 0}
\gpsetlinewidth{3.00}
\draw[gp path] (0.828,1.771)--(1.998,2.167)--(3.169,2.542)--(4.339,2.787)--(5.510,2.960)%
  --(6.680,3.148)--(7.851,3.309)--(9.021,3.456)--(10.191,3.568)--(11.362,3.657);
\gpsetpointsize{6.00}
\gppoint{gp mark 3}{(0.828,1.771)}
\gppoint{gp mark 3}{(1.998,2.167)}
\gppoint{gp mark 3}{(3.169,2.542)}
\gppoint{gp mark 3}{(4.339,2.787)}
\gppoint{gp mark 3}{(5.510,2.960)}
\gppoint{gp mark 3}{(6.680,3.148)}
\gppoint{gp mark 3}{(7.851,3.309)}
\gppoint{gp mark 3}{(9.021,3.456)}
\gppoint{gp mark 3}{(10.191,3.568)}
\gppoint{gp mark 3}{(11.362,3.657)}
\gpsetlinewidth{6.00}
\draw[gp path] (0.828,2.488)--(1.998,3.267)--(3.169,3.983)--(4.339,4.533)--(5.510,4.986)%
  --(6.680,5.196)--(7.851,5.501)--(9.021,5.865)--(10.191,6.154)--(11.362,6.585);
\gpsetpointsize{8.00}
\gppoint{gp mark 3}{(0.828,2.488)}
\gppoint{gp mark 3}{(1.998,3.267)}
\gppoint{gp mark 3}{(3.169,3.983)}
\gppoint{gp mark 3}{(4.339,4.533)}
\gppoint{gp mark 3}{(5.510,4.986)}
\gppoint{gp mark 3}{(6.680,5.196)}
\gppoint{gp mark 3}{(7.851,5.501)}
\gppoint{gp mark 3}{(9.021,5.865)}
\gppoint{gp mark 3}{(10.191,6.154)}
\gppoint{gp mark 3}{(11.362,6.585)}
\gpsetlinetype{gp lt plot 1}
\gpsetlinewidth{3.00}
\draw[gp path] (0.828,2.801)--(1.998,3.115)--(3.169,3.301)--(4.339,3.419)--(5.510,3.470)%
  --(6.680,3.553)--(7.851,3.621)--(9.021,3.625)--(10.191,3.654)--(11.362,3.629);
\gpsetpointsize{6.00}
\gppoint{gp mark 7}{(0.828,2.801)}
\gppoint{gp mark 7}{(1.998,3.115)}
\gppoint{gp mark 7}{(3.169,3.301)}
\gppoint{gp mark 7}{(4.339,3.419)}
\gppoint{gp mark 7}{(5.510,3.470)}
\gppoint{gp mark 7}{(6.680,3.553)}
\gppoint{gp mark 7}{(7.851,3.621)}
\gppoint{gp mark 7}{(9.021,3.625)}
\gppoint{gp mark 7}{(10.191,3.654)}
\gppoint{gp mark 7}{(11.362,3.629)}
\gpcolor{\gprgb{448}{448}{448}}
\gpsetlinetype{gp lt plot 0}
\gpsetlinewidth{7.00}
\draw[gp path] (3.967,0.616)--(4.339,0.961)--(4.924,1.533)--(5.510,2.113)--(6.095,2.773)%
  --(6.680,3.335)--(7.265,3.905)--(7.851,4.430)--(8.436,4.954)--(9.021,5.469)--(9.606,5.989)%
  --(10.191,6.496)--(10.777,7.080)--(11.362,7.531);
\gpsetlinetype{gp lt plot 1}
\gpsetlinewidth{5.00}
\draw[gp path] (3.730,0.616)--(3.754,0.636)--(4.339,0.989)--(4.924,1.696)--(5.510,2.198)%
  --(6.095,2.496)--(6.680,3.026)--(7.265,3.587)--(7.851,4.069)--(8.436,4.618)--(9.021,5.086)%
  --(9.606,5.616)--(10.191,6.208)--(10.777,6.858)--(11.362,7.202);
\gpcolor{gp lt color border}
\gpsetlinetype{gp lt border}
\gpsetlinewidth{1.00}
\draw[gp path] (0.828,8.381)--(0.828,0.616)--(11.947,0.616);
\gpdefrectangularnode{gp plot 1}{\pgfpoint{0.828cm}{0.616cm}}{\pgfpoint{11.947cm}{8.381cm}}
\end{tikzpicture}
} \hfil
 \end{center}
 \begin{center} \hfil
  \resizebox{.48\textwidth}{!}{\begin{tikzpicture}[gnuplot]
\gpcolor{gp lt color border}
\gpsetlinetype{gp lt border}
\gpsetlinewidth{1.00}
\draw[gp path] (0.828,0.616)--(1.008,0.616);
\node[gp node right] at (0.644,0.616) {$-2$};
\draw[gp path] (0.828,2.185)--(1.008,2.185);
\node[gp node right] at (0.644,2.185) {$-1$};
\draw[gp path] (0.828,3.753)--(1.008,3.753);
\node[gp node right] at (0.644,3.753) {$0$};
\draw[gp path] (0.828,5.322)--(1.008,5.322);
\node[gp node right] at (0.644,5.322) {$1$};
\draw[gp path] (0.828,6.891)--(1.008,6.891);
\node[gp node right] at (0.644,6.891) {$2$};
\draw[gp path] (0.828,0.616)--(0.828,0.796);
\node[gp node center] at (0.828,0.308) {$10$};
\draw[gp path] (3.754,0.616)--(3.754,0.796);
\node[gp node center] at (3.754,0.308) {$15$};
\draw[gp path] (6.680,0.616)--(6.680,0.796);
\node[gp node center] at (6.680,0.308) {$20$};
\draw[gp path] (9.606,0.616)--(9.606,0.796);
\node[gp node center] at (9.606,0.308) {$25$};
\draw[gp path] (0.828,8.381)--(0.828,0.616)--(11.947,0.616);
\node[gp node left] at (0.939,8.303) {$\log_{10}(\mathrm{time},$ sec.)};
\node[gp node right] at (12.614,0.305) {$\log_2n$};
\gpcolor{\gprgb{0}{0}{0}}
\gpsetlinetype{gp lt plot 0}
\gpsetlinewidth{3.00}
\draw[gp path] (0.828,1.803)--(1.998,2.205)--(3.169,2.566)--(4.339,2.836)--(5.510,3.038)%
  --(6.680,3.270)--(7.851,3.414)--(9.021,3.604)--(10.191,3.752)--(11.362,3.872);
\gpsetpointsize{6.00}
\gppoint{gp mark 3}{(0.828,1.803)}
\gppoint{gp mark 3}{(1.998,2.205)}
\gppoint{gp mark 3}{(3.169,2.566)}
\gppoint{gp mark 3}{(4.339,2.836)}
\gppoint{gp mark 3}{(5.510,3.038)}
\gppoint{gp mark 3}{(6.680,3.270)}
\gppoint{gp mark 3}{(7.851,3.414)}
\gppoint{gp mark 3}{(9.021,3.604)}
\gppoint{gp mark 3}{(10.191,3.752)}
\gppoint{gp mark 3}{(11.362,3.872)}
\gpsetlinewidth{6.00}
\draw[gp path] (0.828,2.526)--(1.998,3.405)--(3.169,4.135)--(4.339,4.725)--(5.510,5.184)%
  --(6.680,5.527)--(7.851,5.828)--(9.021,6.104)--(10.191,6.378)--(11.362,6.643);
\gpsetpointsize{8.00}
\gppoint{gp mark 3}{(0.828,2.526)}
\gppoint{gp mark 3}{(1.998,3.405)}
\gppoint{gp mark 3}{(3.169,4.135)}
\gppoint{gp mark 3}{(4.339,4.725)}
\gppoint{gp mark 3}{(5.510,5.184)}
\gppoint{gp mark 3}{(6.680,5.527)}
\gppoint{gp mark 3}{(7.851,5.828)}
\gppoint{gp mark 3}{(9.021,6.104)}
\gppoint{gp mark 3}{(10.191,6.378)}
\gppoint{gp mark 3}{(11.362,6.643)}
\gpsetlinetype{gp lt plot 1}
\gpsetlinewidth{3.00}
\draw[gp path] (0.828,2.994)--(1.998,3.458)--(3.169,3.741)--(4.339,3.963)--(5.510,4.142)%
  --(6.680,4.232)--(7.851,4.321)--(9.021,4.430)--(10.191,4.495)--(11.362,4.524);
\gpsetpointsize{6.00}
\gppoint{gp mark 7}{(0.828,2.994)}
\gppoint{gp mark 7}{(1.998,3.458)}
\gppoint{gp mark 7}{(3.169,3.741)}
\gppoint{gp mark 7}{(4.339,3.963)}
\gppoint{gp mark 7}{(5.510,4.142)}
\gppoint{gp mark 7}{(6.680,4.232)}
\gppoint{gp mark 7}{(7.851,4.321)}
\gppoint{gp mark 7}{(9.021,4.430)}
\gppoint{gp mark 7}{(10.191,4.495)}
\gppoint{gp mark 7}{(11.362,4.524)}
\gpcolor{\gprgb{448}{448}{448}}
\gpsetlinetype{gp lt plot 0}
\gpsetlinewidth{7.00}
\draw[gp path] (3.967,0.616)--(4.339,0.961)--(4.924,1.533)--(5.510,2.113)--(6.095,2.773)%
  --(6.680,3.335)--(7.265,3.905)--(7.851,4.430)--(8.436,4.954)--(9.021,5.469)--(9.606,5.989)%
  --(10.191,6.496)--(10.777,7.080)--(11.362,7.531);
\gpsetlinetype{gp lt plot 1}
\gpsetlinewidth{5.00}
\draw[gp path] (3.730,0.616)--(3.754,0.636)--(4.339,0.989)--(4.924,1.696)--(5.510,2.198)%
  --(6.095,2.496)--(6.680,3.026)--(7.265,3.587)--(7.851,4.069)--(8.436,4.618)--(9.021,5.086)%
  --(9.606,5.616)--(10.191,6.208)--(10.777,6.858)--(11.362,7.202);
\gpcolor{gp lt color border}
\gpsetlinetype{gp lt border}
\gpsetlinewidth{1.00}
\draw[gp path] (0.828,8.381)--(0.828,0.616)--(11.947,0.616);
\gpdefrectangularnode{gp plot 1}{\pgfpoint{0.828cm}{0.616cm}}{\pgfpoint{11.947cm}{8.381cm}}
\end{tikzpicture}
} \hfil
  \resizebox{.48\textwidth}{!}{\begin{tikzpicture}[gnuplot]
\gpcolor{gp lt color border}
\gpsetlinetype{gp lt border}
\gpsetlinewidth{1.00}
\draw[gp path] (0.828,0.616)--(1.008,0.616);
\node[gp node right] at (0.644,0.616) {$-2$};
\draw[gp path] (0.828,2.185)--(1.008,2.185);
\node[gp node right] at (0.644,2.185) {$-1$};
\draw[gp path] (0.828,3.753)--(1.008,3.753);
\node[gp node right] at (0.644,3.753) {$0$};
\draw[gp path] (0.828,5.322)--(1.008,5.322);
\node[gp node right] at (0.644,5.322) {$1$};
\draw[gp path] (0.828,6.891)--(1.008,6.891);
\node[gp node right] at (0.644,6.891) {$2$};
\draw[gp path] (0.828,0.616)--(0.828,0.796);
\node[gp node center] at (0.828,0.308) {$10$};
\draw[gp path] (3.754,0.616)--(3.754,0.796);
\node[gp node center] at (3.754,0.308) {$15$};
\draw[gp path] (6.680,0.616)--(6.680,0.796);
\node[gp node center] at (6.680,0.308) {$20$};
\draw[gp path] (9.606,0.616)--(9.606,0.796);
\node[gp node center] at (9.606,0.308) {$25$};
\draw[gp path] (0.828,8.381)--(0.828,0.616)--(11.947,0.616);
\node[gp node left] at (0.939,8.303) {$\log_{10}(\mathrm{time},$ sec.)};
\node[gp node right] at (12.614,0.305) {$\log_2n$};
\gpcolor{\gprgb{0}{0}{0}}
\gpsetlinetype{gp lt plot 0}
\gpsetlinewidth{3.00}
\draw[gp path] (0.828,1.838)--(1.998,2.292)--(3.169,2.778)--(4.339,3.162)--(5.510,3.438)%
  --(6.680,3.700)--(7.851,3.932)--(9.021,4.111)--(10.191,4.246)--(11.362,4.388);
\gpsetpointsize{6.00}
\gppoint{gp mark 3}{(0.828,1.838)}
\gppoint{gp mark 3}{(1.998,2.292)}
\gppoint{gp mark 3}{(3.169,2.778)}
\gppoint{gp mark 3}{(4.339,3.162)}
\gppoint{gp mark 3}{(5.510,3.438)}
\gppoint{gp mark 3}{(6.680,3.700)}
\gppoint{gp mark 3}{(7.851,3.932)}
\gppoint{gp mark 3}{(9.021,4.111)}
\gppoint{gp mark 3}{(10.191,4.246)}
\gppoint{gp mark 3}{(11.362,4.388)}
\gpsetlinewidth{6.00}
\draw[gp path] (0.828,2.622)--(1.998,3.619)--(3.169,4.554)--(4.339,5.638)--(5.510,6.884)%
  --(6.680,7.797)--(7.514,8.381);
\draw[gp path] (8.374,8.381)--(9.021,8.090)--(10.191,7.464)--(11.362,7.328);
\gpsetpointsize{8.00}
\gppoint{gp mark 3}{(0.828,2.622)}
\gppoint{gp mark 3}{(1.998,3.619)}
\gppoint{gp mark 3}{(3.169,4.554)}
\gppoint{gp mark 3}{(4.339,5.638)}
\gppoint{gp mark 3}{(5.510,6.884)}
\gppoint{gp mark 3}{(6.680,7.797)}
\gppoint{gp mark 3}{(9.021,8.090)}
\gppoint{gp mark 3}{(10.191,7.464)}
\gppoint{gp mark 3}{(11.362,7.328)}
\gpsetlinetype{gp lt plot 1}
\gpsetlinewidth{3.00}
\draw[gp path] (0.828,3.738)--(1.998,4.338)--(3.169,4.843)--(4.339,5.148)--(5.510,5.374)%
  --(6.680,5.511)--(7.851,5.496)--(9.021,5.521)--(10.191,5.566)--(11.362,5.527);
\gpsetpointsize{6.00}
\gppoint{gp mark 7}{(0.828,3.738)}
\gppoint{gp mark 7}{(1.998,4.338)}
\gppoint{gp mark 7}{(3.169,4.843)}
\gppoint{gp mark 7}{(4.339,5.148)}
\gppoint{gp mark 7}{(5.510,5.374)}
\gppoint{gp mark 7}{(6.680,5.511)}
\gppoint{gp mark 7}{(7.851,5.496)}
\gppoint{gp mark 7}{(9.021,5.521)}
\gppoint{gp mark 7}{(10.191,5.566)}
\gppoint{gp mark 7}{(11.362,5.527)}
\gpcolor{\gprgb{448}{448}{448}}
\gpsetlinetype{gp lt plot 0}
\gpsetlinewidth{7.00}
\draw[gp path] (3.967,0.616)--(4.339,0.961)--(4.924,1.533)--(5.510,2.113)--(6.095,2.773)%
  --(6.680,3.335)--(7.265,3.905)--(7.851,4.430)--(8.436,4.954)--(9.021,5.469)--(9.606,5.989)%
  --(10.191,6.496)--(10.777,7.080)--(11.362,7.531);
\gpsetlinetype{gp lt plot 1}
\gpsetlinewidth{5.00}
\draw[gp path] (3.730,0.616)--(3.754,0.636)--(4.339,0.989)--(4.924,1.696)--(5.510,2.198)%
  --(6.095,2.496)--(6.680,3.026)--(7.265,3.587)--(7.851,4.069)--(8.436,4.618)--(9.021,5.086)%
  --(9.606,5.616)--(10.191,6.208)--(10.777,6.858)--(11.362,7.202);
\gpcolor{gp lt color border}
\gpsetlinetype{gp lt border}
\gpsetlinewidth{1.00}
\draw[gp path] (0.828,8.381)--(0.828,0.616)--(11.947,0.616);
\gpdefrectangularnode{gp plot 1}{\pgfpoint{0.828cm}{0.616cm}}{\pgfpoint{11.947cm}{8.381cm}}
\end{tikzpicture}
} \hfil
 \end{center}
\caption{Runtimes of divide and conquer algorithm (solid lines) and modified Bini's algorithm (dashed lines) for the inversion of triangular Toeplitz matrix~\eqref{eq3} in full and in the QTT formats (grey and black lines, respectively) w.r.t. problem size $n$ and step size $h=T/n.$ 
Fixed maximum time $T=10,$ fractional order $\alpha=0.2$ (left) and $\alpha=0.8$ (right), mass $m=-10^{-5}$ (top), $m=-10^0$ (middle), $m=-10^5$ (bottom).} 
\label{fig:qtime} 
\end{figure}
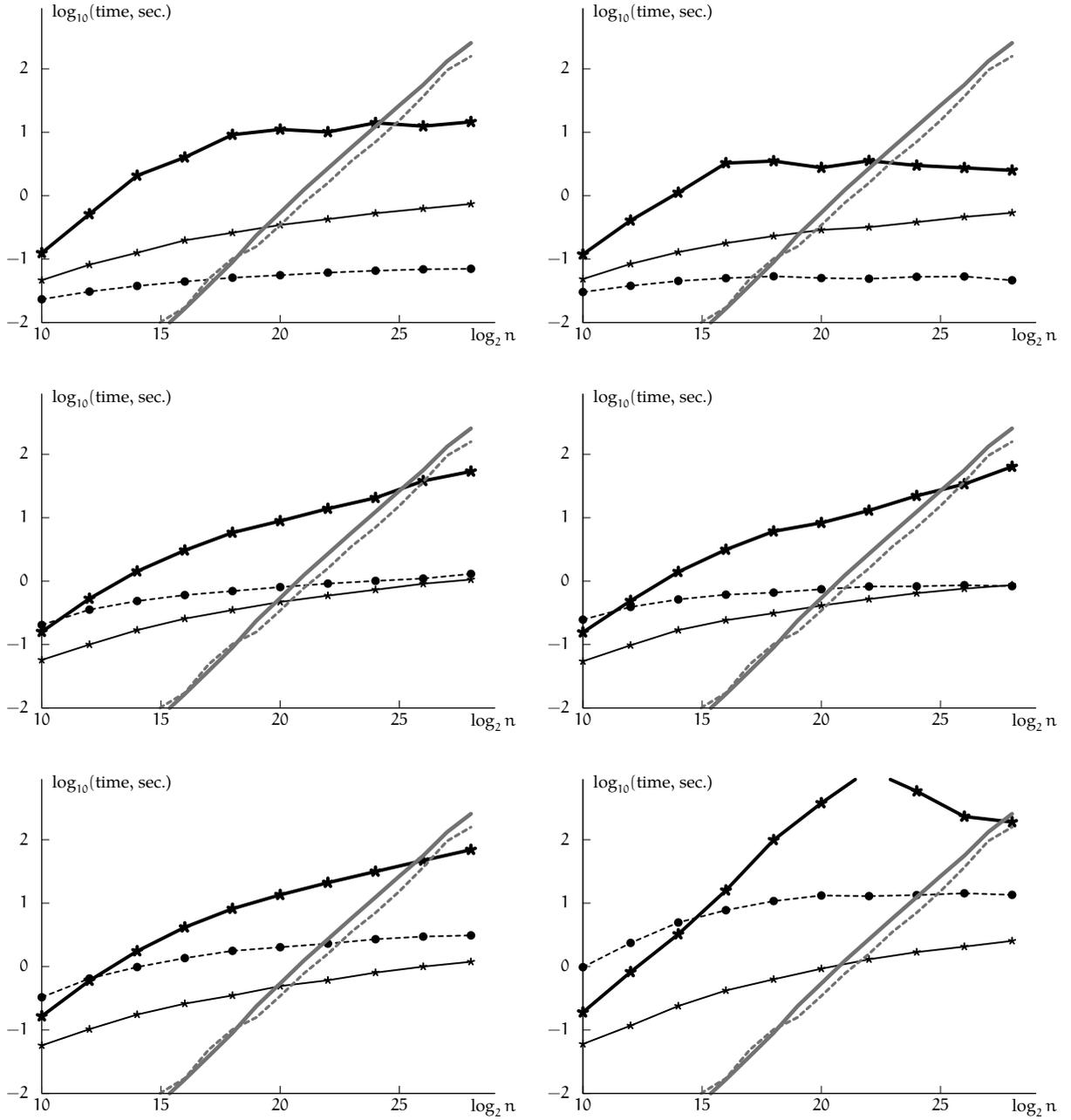
On Fig.~\ref{fig:qtime} we show the runtime of inversion algorithms for triangular Toeplitz matrices in full and in the QTT format w.r.t. problem size and for different parameters $\alpha, m.$
Standard inversion algorithms have the $\O(n\log n)$ complexity which depends only on problem size.
Quite contrarily, the complexity and runtime of QTT algorithms depend on QTT--ranks of input and intermediate vectors, which are sensitive to the fractional order $\alpha,$ mass $m$ and step size $h.$ 
They also depend crucially on the method used to compute the discrete convolution in the QTT format.
We can note that the divide and conquer algorithm~\ref{alg:dc} which uses QTT--conv algorithm~\cite{khkaz-conv-2011} is always significantly faster than the same method which uses QTT--FFT algorithm~\cite{dks-ttfft-2012} to compute the convolution.
However, QTT--FFT works well in modified Bini's algorithm~\ref{alg:bi}, which appears to be the fastest method when mass is small in modulus.  
For large mass the divide and conquer algorithm~\ref{alg:dc} with QTT--conv is preferable to the modified Bini's algorithm~\ref{alg:bi}.
For mass $m\sim -1$ these methods have the same asymptotical complexity. 

From Fig.~\ref{fig:qtime} it can be easily seen that the QTT algorithms are asymptotically faster than the algorithms in full format.
For practical computations it is very important at which size $n$ there is a \emph{crossover point}, i.e., the minimum value of $n$ for which the QTT algorithms are actually faster than the algorithms in full format. 
Numerical experiments show that for a wide range of parameters $\alpha$ and $m$ the crossover point between full and QTT divide and conquer methods is $\log_2 n \simeq {20}.$
This value is about the same as the crossover point between FFT and QTT--FFT algorithm applied to signals with sparse Fourier image~\cite{dks-ttfft-2012}.
The crossover point between full and QTT versions of the modified Bini's algorithm depends on $m$ and $\alpha$ and can be even smaller, e.g. $\log_2 n \simeq 17$ for $m$ small in modulus. 

\subsection{Accuracy test for constant forcing}
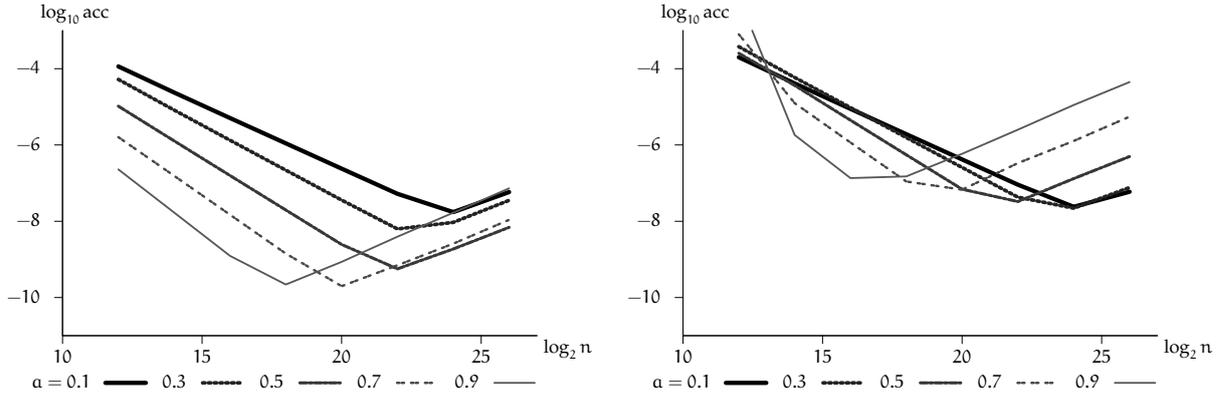
\begin{figure}[t]
 \begin{center} \hfil
  \resizebox{.48\textwidth}{!}{\begin{tikzpicture}[gnuplot]
\gpcolor{gp lt color border}
\gpsetlinetype{gp lt border}
\gpsetlinewidth{1.00}
\draw[gp path] (1.192,2.283)--(1.012,2.283);
\node[gp node right] at (0.828,2.283) {$-10$};
\draw[gp path] (1.192,4.025)--(1.012,4.025);
\node[gp node right] at (0.828,4.025) {$-8$};
\draw[gp path] (1.192,5.768)--(1.012,5.768);
\node[gp node right] at (0.828,5.768) {$-6$};
\draw[gp path] (1.192,7.510)--(1.012,7.510);
\node[gp node right] at (0.828,7.510) {$-4$};
\draw[gp path] (1.192,1.412)--(1.192,1.232);
\node[gp node center] at (1.192,0.924) {$10$};
\draw[gp path] (4.355,1.412)--(4.355,1.232);
\node[gp node center] at (4.355,0.924) {$15$};
\draw[gp path] (7.518,1.412)--(7.518,1.232);
\node[gp node center] at (7.518,0.924) {$20$};
\draw[gp path] (10.682,1.412)--(10.682,1.232);
\node[gp node center] at (10.682,0.924) {$25$};
\draw[gp path] (1.192,8.381)--(1.192,1.412)--(11.947,1.412);
\node[gp node left] at (0.547,8.799) {$\log_{10}\mathrm{acc}$};
\node[gp node left] at (11.947,1.133) {$\log_2n$};
\node[gp node right] at (1.979,0.334) {$a=0.1$};
\gpcolor{\gprgb{0}{0}{0}}
\gpsetlinetype{gp lt plot 0}
\gpsetlinewidth{7.00}
\draw[gp path] (2.163,0.334)--(3.079,0.334);
\draw[gp path] (2.457,7.562)--(3.723,6.965)--(4.988,6.383)--(6.253,5.805)--(7.518,5.227)%
  --(8.784,4.652)--(10.049,4.234)--(11.314,4.692);
\gpcolor{gp lt color border}
\node[gp node right] at (4.183,0.334) {$0.3$};
\gpcolor{\gprgb{136}{136}{136}}
\gpsetlinetype{gp lt plot 2}
\gpsetlinewidth{6.00}
\draw[gp path] (4.367,0.334)--(5.283,0.334);
\draw[gp path] (2.457,7.265)--(3.723,6.564)--(4.988,5.876)--(6.253,5.192)--(7.518,4.509)%
  --(8.784,3.850)--(10.049,4.000)--(11.314,4.503);
\gpcolor{gp lt color border}
\node[gp node right] at (6.387,0.334) {$0.5$};
\gpcolor{\gprgb{215}{215}{215}}
\gpsetlinetype{gp lt plot 4}
\gpsetlinewidth{5.00}
\draw[gp path] (6.571,0.334)--(7.487,0.334);
\draw[gp path] (2.457,6.654)--(3.723,5.860)--(4.988,5.070)--(6.253,4.282)--(7.518,3.495)%
  --(8.784,2.938)--(10.049,3.392)--(11.314,3.886);
\gpcolor{gp lt color border}
\node[gp node right] at (8.591,0.334) {$0.7$};
\gpcolor{\gprgb{282}{282}{282}}
\gpsetlinetype{gp lt plot 6}
\gpsetlinewidth{4.00}
\draw[gp path] (8.775,0.334)--(9.691,0.334);
\draw[gp path] (2.457,5.947)--(3.723,5.062)--(4.988,4.175)--(6.253,3.287)--(7.518,2.543)%
  --(8.784,3.020)--(10.049,3.514)--(11.314,4.058);
\gpcolor{gp lt color border}
\node[gp node right] at (10.795,0.334) {$0.9$};
\gpcolor{\gprgb{342}{342}{342}}
\gpsetlinetype{gp lt plot 0}
\gpsetlinewidth{3.00}
\draw[gp path] (10.979,0.334)--(11.895,0.334);
\draw[gp path] (2.457,5.210)--(3.723,4.223)--(4.988,3.235)--(6.253,2.580)--(7.518,3.098)%
  --(8.784,3.673)--(10.049,4.220)--(11.314,4.777);
\gpcolor{gp lt color border}
\gpsetlinetype{gp lt border}
\gpsetlinewidth{1.00}
\draw[gp path] (1.192,8.381)--(1.192,1.412)--(11.947,1.412);
\gpdefrectangularnode{gp plot 1}{\pgfpoint{1.192cm}{1.412cm}}{\pgfpoint{11.947cm}{8.381cm}}
\end{tikzpicture}
} \hfil
  \resizebox{.48\textwidth}{!}{\begin{tikzpicture}[gnuplot]
\gpcolor{gp lt color border}
\gpsetlinetype{gp lt border}
\gpsetlinewidth{1.00}
\draw[gp path] (1.192,2.283)--(1.012,2.283);
\node[gp node right] at (0.828,2.283) {$-10$};
\draw[gp path] (1.192,4.025)--(1.012,4.025);
\node[gp node right] at (0.828,4.025) {$-8$};
\draw[gp path] (1.192,5.768)--(1.012,5.768);
\node[gp node right] at (0.828,5.768) {$-6$};
\draw[gp path] (1.192,7.510)--(1.012,7.510);
\node[gp node right] at (0.828,7.510) {$-4$};
\draw[gp path] (1.192,1.412)--(1.192,1.232);
\node[gp node center] at (1.192,0.924) {$10$};
\draw[gp path] (4.355,1.412)--(4.355,1.232);
\node[gp node center] at (4.355,0.924) {$15$};
\draw[gp path] (7.518,1.412)--(7.518,1.232);
\node[gp node center] at (7.518,0.924) {$20$};
\draw[gp path] (10.682,1.412)--(10.682,1.232);
\node[gp node center] at (10.682,0.924) {$25$};
\draw[gp path] (1.192,8.381)--(1.192,1.412)--(11.947,1.412);
\node[gp node left] at (0.547,8.799) {$\log_{10}\mathrm{acc}$};
\node[gp node left] at (11.947,1.133) {$\log_2n$};
\node[gp node right] at (1.979,0.334) {$a=0.1$};
\gpcolor{\gprgb{0}{0}{0}}
\gpsetlinetype{gp lt plot 0}
\gpsetlinewidth{7.00}
\draw[gp path] (2.163,0.334)--(3.079,0.334);
\draw[gp path] (2.457,7.771)--(3.723,7.174)--(4.988,6.592)--(6.253,6.014)--(7.518,5.437)%
  --(8.784,4.860)--(10.049,4.355)--(11.314,4.699);
\gpcolor{gp lt color border}
\node[gp node right] at (4.183,0.334) {$0.3$};
\gpcolor{\gprgb{136}{136}{136}}
\gpsetlinetype{gp lt plot 2}
\gpsetlinewidth{6.00}
\draw[gp path] (4.367,0.334)--(5.283,0.334);
\draw[gp path] (2.457,8.013)--(3.723,7.310)--(4.988,6.621)--(6.253,5.937)--(7.518,5.254)%
  --(8.784,4.575)--(10.049,4.324)--(11.314,4.795);
\gpcolor{gp lt color border}
\node[gp node right] at (6.387,0.334) {$0.5$};
\gpcolor{\gprgb{215}{215}{215}}
\gpsetlinetype{gp lt plot 4}
\gpsetlinewidth{5.00}
\draw[gp path] (6.571,0.334)--(7.487,0.334);
\draw[gp path] (2.457,7.862)--(3.723,7.128)--(4.988,6.333)--(6.253,5.543)--(7.518,4.757)%
  --(8.784,4.475)--(10.049,5.000)--(11.314,5.504);
\gpcolor{gp lt color border}
\node[gp node right] at (8.591,0.334) {$0.7$};
\gpcolor{\gprgb{282}{282}{282}}
\gpsetlinetype{gp lt plot 6}
\gpsetlinewidth{4.00}
\draw[gp path] (8.775,0.334)--(9.691,0.334);
\draw[gp path] (2.457,8.295)--(3.723,6.723)--(4.988,5.824)--(6.253,4.934)--(7.518,4.747)%
  --(8.784,5.354)--(10.049,5.860)--(11.314,6.414);
\gpcolor{gp lt color border}
\node[gp node right] at (10.795,0.334) {$0.9$};
\gpcolor{\gprgb{342}{342}{342}}
\gpsetlinetype{gp lt plot 0}
\gpsetlinewidth{3.00}
\draw[gp path] (10.979,0.334)--(11.895,0.334);
\draw[gp path] (2.760,8.381)--(3.723,5.996)--(4.988,5.011)--(6.253,5.050)--(7.518,5.566)%
  --(8.784,6.117)--(10.049,6.681)--(11.314,7.205);
\gpcolor{gp lt color border}
\gpsetlinetype{gp lt border}
\gpsetlinewidth{1.00}
\draw[gp path] (1.192,8.381)--(1.192,1.412)--(11.947,1.412);
\gpdefrectangularnode{gp plot 1}{\pgfpoint{1.192cm}{1.412cm}}{\pgfpoint{11.947cm}{8.381cm}}
\end{tikzpicture}
} \hfil
 \end{center}
\caption{Accuracy of the solution of the test problem~\eqref{eq:fconst} in the relative Frobenius norm w.r.t. problem size $n$ and for different fractional parameters $\alpha.$ 
Fixed maximum time $T=10$ (left) and $T=10^5$ (right). Mass $m=-1.$} 
\label{fig:acc} 
\end{figure}
We consider a simple problem for which the analytical solution is available, namely the one with constant forcing term.
\begin{equation}\label{eq:fconst}
 D^{\alpha}_*y(t)=my(t)+\lambda, \qquad y(0)=y_0.
\end{equation}
The analytical solution is written in the following form
\begin{equation} \label{eq:fsol}
y(t)=y_0E_\alpha\left(m t^\alpha\right) +\frac{\lambda}{m}E_\alpha\left(m t^\alpha\right) -\frac{\lambda}{m},
\end{equation}
where $E_\alpha$ is the Mittag--Lefler function~\cite{ML1, ML2}, which can be expressed and computed by certain (sometimes slow-converging) series.

The accuracy verification results are shown on Fig.~\ref{fig:acc}. 
We see that as the problem size grows, the accuracy improves until certain point and then the error start growing.
This is explained by the machine threshold errors amplified by the condition number of the matrix $A$ from~\eqref{eq3} which is unbounded as $n$ grows to the infinity.

\subsection{Accuracy of the Laplace transform}
\begin{figure}[t]
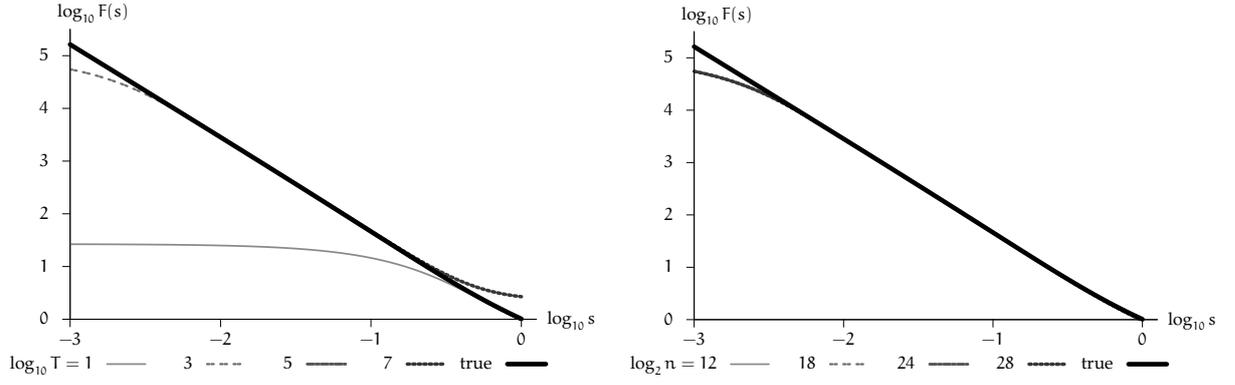

 \begin{center} \hfil
  \resizebox{.48\textwidth}{!}{\input{./Pic/lap/l3aa80n22sm0.tikz}} \hfil
  \resizebox{.48\textwidth}{!}{\input{./Pic/lap/l3aa80t03sm0.tikz}} \hfil
 \end{center}
\caption{The Laplace transform of the solution~\eqref{eq:sol34} and its discrete approximations. Fractional parameter $\alpha=0.8,$ mass $m=-1.$
Left: fixed number of grid points $n=2^{22},$ different maximum time $T.$ 
Right: fixed maximum time $T=10^3,$ different $n.$}  
\label{fig:lapa} 
\end{figure}

Consider the following test equation
\begin{equation}\label{eq:34}
D^{\alpha}_*y(t)=my(t)+t^\frac{3}{4},  \qquad y(0)=1.
\end{equation}
Since the forcing term $f(t)=t^{3/4}$ does not have a short Taylor series representation, this problem could be difficult for methods based on it.
Unlike the previous example, the analytical solution in space domain is not available.
Instead we can solve the problem using the Laplace transform.\footnote{History of the Laplace transform and other essential details can be found in, eg.~\cite{laplace-2003}.}
The  Laplace transform of a function $f(t)$ with the appropriate speed of decay is defined by
\begin{equation}\label{eq:lap}
 F(s) = \L\{f(t)\} = \int_0^{\infty} e^{-st} f(t) dt.
\end{equation}
The Laplace transform of a convolution is the product of Laplace transforms,
\begin{equation}\nonumber
 (f \star g)(t) = \int_0^t f(\tau) g(t-\tau) d\tau \qquad\Leftrightarrow\qquad \L\{f \star g\}(s) = \L\{f\}(s)  \L\{g\}(s).
\end{equation}
This allows to simplify the equation~\eqref{eq:34} and find the Laplace transform of the solution,
\begin{equation} \label{eq:sol34}
 Y(s)=\frac{1}{s^{1-\alpha}(s\alpha-m)}+\frac{\Gamma(1.75)}{s^{1.75} (s^\alpha-m)}.
\end{equation}

The inverse Laplace transform $y(t)=\L^{-1}\{Y(s)\}$ is given by the complex contour integral and is difficult for numerical computation.
However, we can easily compute the Laplace transform of the discrete solution $\L\tilde y = \tilde Y$ in points $\{s_k\}$ using a rectangle quadrature rule,
\begin{equation} \label{eq:lapd}
 \tilde Y(s_k) \approx \tilde Y_k = h \sum_{j=0}^n e^{-t_j s_k} \tilde y(t_j), \qquad t_j=jh.
\end{equation}
Then we compare $Y(s_k)$ and $\tilde Y_k$ to establish the accuracy of the discrete solution $\tilde y(t_j)=y_j.$
It is easy to see that equation~\eqref{eq:lapd} contains three sources of errors, i.e. the ones of the discrete solution, of the quadrature rule and of the truncation of indefinite integral~\eqref{eq:lap} to the finite interval $[0:T],$ $T=nh.$
To compute $\tilde Y(s)$ accurately for small $s$ we should take $T > s^{-1} \log \eps^{-1},$ where $\eps$ is a machine precision error.
To keep the quadrature rule error small, we should also use  grids with small time step $h.$
This is shown on Fig.~\ref{fig:lapa}, where the exact Laplace transform~\eqref{eq:sol34} is compared with its discrete approximations for different $T$ and $n.$
These factors motivate the use of very large grid size $n$ and hence the QTT approach. 
It should be noted that the discrete Laplace transform~\eqref{eq:lapd} is computed perfectly in the QTT format since the QTT--ranks of the exponent are all ones. 

Finally, on Fig.~\ref{fig:lape} we show the accuracy of the Laplace transform of the solution~\eqref{eq:sol34} for $10^{-3} \leq s \leq 10^0$ and for different $T$ and $n.$
It is clear that large problem size is essential for the accurate representation of the solution in the Laplace transform space.

\begin{figure}[t]
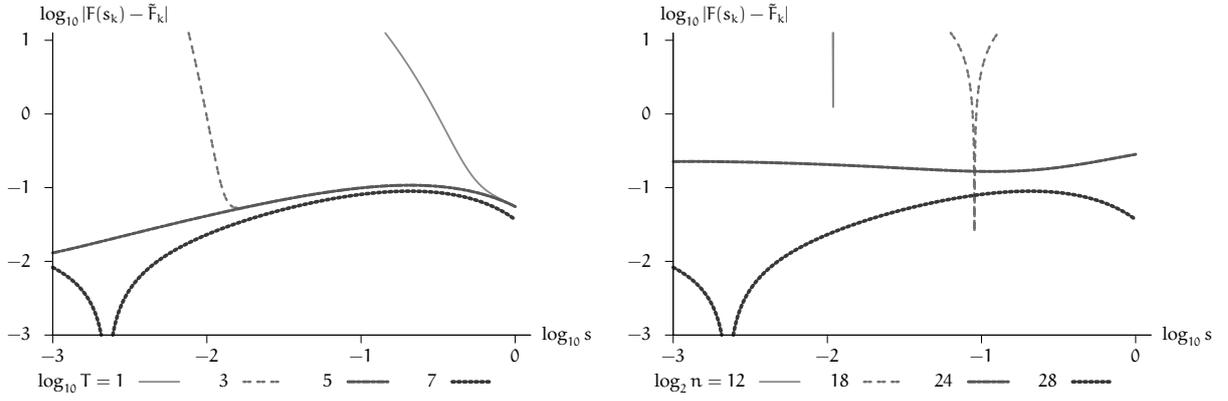

 \begin{center} \hfil
  \resizebox{.48\textwidth}{!}{\input{./Pic/lap/l3ea80n28sm0.tikz}} \hfil
  \resizebox{.48\textwidth}{!}{\input{./Pic/lap/l3ea80t07sm0.tikz}} \hfil
 \end{center}
\caption{Accuracy of the Laplace transform~\eqref{eq:sol34} given by discrete approximation~\eqref{eq:lapd}. Fractional parameter $\alpha=0.8,$ mass $m=-1.$
Left: fixed number of grid points $n=2^{28},$ different maximum time $T.$ 
Right: fixed maximum time $T=10^7,$ different $n.$} 
\label{fig:lape} 
\end{figure}

\section{Conclusions and future work} \label{CONC}
We present the new family of algorithms for the solution of linear fractional ODEs.
Our approach develops the framework of matrix algorithms for fractional calculus~\cite{podlubny-matr-2000} by embedding the QTT tensor decomposition inside matrices, as proposed in~\cite{osel-2d2d-2010}.
The proposed algorithms works on matrix level and can be formally applied for the inversion of any triangular Toeplitz matrix, as well as the one obtained by discretisation of a linear fractional calculus problem. 
The workhorse of the inversion algorithms is the discrete convolution and/or Fourier transform of vectors given/approximated in the compressed QTT form.
The success of the proposed algorithms, however, is based on the representability of the initial matrix and intermediate vectors arising in computations in the QTT format with a modest accuracy.

As the motivating example we consider a simple linear fractional differential equation which reduces to  the weakly singular convolutional Volterra equation with the Abel-type kernel.
The QTT approximation method benefits from both the smoothness and decay of the Abel kernel, which results in efficient QTT--representation of problem matrix with the accuracy up to the machine precision. 
As shown by numerical experiments, the QTT--ranks of the intermediate vectors in the proposed algorithms remain bounded or grow slowly with the problem size.
As the result, our algorithms of the inversion of triangular Toeplitz matrices demonstrate sublinear $o(n)$ complexity, which falls down to the complexity $\O(\log^2 n)$ of the superfast Fourier transform in certain cases. 
For our implementation the crossover point with the standard algorithms based on the FFTW library for the considered experiments is $17 \lesssim \log_2n \lesssim 21,$ i.e., the developed methods give not only the asymptotical benefit, but also a practical speedup for the problems of moderate size.

The proposed approach opens a new class of algorithms for the fractional calculus, i.e., methods of sublinear complexity.
The developed techniques can be applied to the fractional equations with several differential operators of different order.
They also can be generalised to fractional PDEs in two and more dimensions and to the nonlinear fractional problem.
This would be the topic of further work, which will be reported elsewhere.

\section*{References}

\newpage
\def\appendixname{}
\appendix
\section{Fractional differential operators}
We begin by presenting established definitions of the fractional Riemann-Louiville operator, the fractional Riemann-Louiville derivative and a modified form of the fractional derivative --- the Caputo derivative. These definitions can be found in a variety of ealier sources, including \cite{Diet1} and \cite{Pod1}.

\begin{definition}\cite{Diet1}
\label{RLI1}
Let $\alpha\in\R^+$. The operator $J^\alpha_a$, defined on $L_1[a,b]$ by $$J^\alpha_ag(t):=\frac{1}{\Gamma(\alpha)}\int^t_a(t-s)^{\alpha-1}g(s)ds$$
for $a\leq t\leq b$, is called the Riemann-Louville fractional integral operator of order $\alpha$. For $\alpha=0$, we set $J^0_aL=I$, the identity operator.
\end{definition}

\begin{definition}\cite{Diet1}
\label{RLF1}
Let $\alpha\in\R^+$ and $p=\lceil\alpha\rceil$. The operator $D^\alpha_a$, defined by
$$D^\alpha_ag:=D^pJ^{p-\alpha}_ag$$
is called the Riemann-Louiville fractional differential operator of order $\alpha$. For $\alpha=0$, we set $D^0_a:=I$, the identity operator.
\end{definition}

\begin{definition}\cite{Diet1}
\label{CF1}
Assume that $\alpha\geq 0$ and that $g$ is such that $D^\alpha_a\left( g-T_{p-1}[g;a]\right)$ exists, where $p=\lceil\alpha\rceil$ and $T_{p-1}[g;a]$ is the Taylor polynomial of degree $p-1$ for the function $g$ about the point $t=a$; $T_{p-1}[g;a]:=0$ for $p=0$. Then we define the function $D^\alpha_{*a}g$ by
$$D^\alpha_{*a}g:=D^\alpha_a\left( g-T_{p-1}[g;a]\right) .$$
The operator $D^\alpha_{*a}$ is called the Caputo differential operator of order $\alpha$.
\end{definition}

We have chosen a problem defined in terms of the Caputo derivative because it allows us to specify non-homogeneous initial conditions for our test equation and thus it is more advantageous for modelling real-world phenomena \cite{Diet1}. This allows us to draw comparisons with existing work (such as \cite{DietFordFreed1} and \cite{FordSimp1}) which focus on the Caputo form of the derivative. 

We also note that for our range of $\alpha$ we will always have $p=\lceil\alpha\rceil=1$. Also, as is sometimes the case in the literature we will omit our starting value $a=0$ for $t$ from our notation; i.e. we will write $J^\alpha$ for $J^\alpha_a$, $D^\alpha$ for $D^\alpha_a$ and $D^\alpha_*$ for $D^\alpha_{*a}$.

Solutions to fractional problems such as (\ref{eq1}) are formulated as functions of the Mittag-Lefler function (first defined by Mittag-Leffler in \cite{ML1}, \cite{ML2}). A discussion of the relevant properties may be found in \cite{Diet1} and \cite{Pod1}; we re-present the details necessary to this paper below.

\begin{definition}
\label{defn1}
Let $\alpha>0$. The function $E_\alpha$ defined by $$E_\alpha(z):=\sum^{\infty}_{j=0}\frac{z^j}{\Gamma (j\alpha +1)}$$ whenever the series converges is called the Mittag-Leffler function of order $\alpha$.
\end{definition}

\begin{definition}
\label{defn2}
Let $\alpha_1,\alpha_2>0$. The function $E_{\alpha_1,\alpha_2}$ defined by $$E_{\alpha_1,\alpha_2}(z):=\sum^{\infty}_{j=0}\frac{z^j}{\Gamma\left( j\alpha_1+\alpha_2\right)}$$ whenever the series converges is called the two-parameter Mittag-Leffler function with parameters $\alpha_1$ and $\alpha_2$.
\end{definition}
Note that $E_\alpha(z)=E_{\alpha,1}(z)$.

\begin{theorem}(\cite{Diet1})
\label{thmitcon}
Consider the two-parameter Mittag-Leffler function $E_{\alpha_1,\alpha_2}$ for some $\alpha_1,\alpha_2>0$. The power series defining $E_{\alpha_1,\alpha_2}(z)$ is convergent for all $z\in\C$.
\end{theorem}

\begin{theorem}(\cite{Diet1})
Let $\alpha>0$ and $\lambda\in\R$. Moreover define $$y(t):=E_\alpha\left( \lambda t^\alpha\right) , \,\, x\geq 0.$$ Then $$D^\alpha_*y(t)=\lambda y(t).$$
\end{theorem}

It is straightforward to apply the more general theory and results in texts such as \cite{Diet1} (which have their origins in analogous results from classical calculus) in order to prove that a solution exists for (\ref{eq1}) over a finite range of $t$. We present such adapted results.

\begin{theorem}
\label{exist1}
Let $K>0$, $b>0$. Define $G:=\left\{ (t,y):t\in[0,b], \left| y-y_0\right|\leq K\right\}$ and let $z:G\rightarrow\R$, such that $z(t,y(t))=my(t)+f(t)$ under the above conditions, be continuous. Furthermore define $M:=\sup_{(t,y)\in G}|z(t,y(t))|$ and $$B=\left\{\begin{array}{ll}b&M=0\\ \min\left\{ b,\left(\frac{K\Gamma(\alpha+1)}{M}\right)^\frac{1}{\alpha}\right\}&{\rm otherwise}\end{array}\right. .$$ Then there exists a unique function $y\in C[0,B]$ solving (\ref{eq1}).
\end{theorem}

In order to prove Theorem \ref{exist1} we first need to prove the following lemma:
\begin{lemma}
\label{exist2}
Assuming the conditions of Theorem \ref{exist1}, the function $y\in C[0,B]$ is a solution of the initial value problem (\ref{eq1}) if and only if it is a solution of the Volterra integral equation
\begin{equation}
\label{volt1}
y(t)=y_0+\frac{1}{\Gamma(\alpha)}\int^t_0(t-s)^{\alpha-1}z(s)ds.
\end{equation}
\end{lemma}
{\bf Proof of Lemma \ref{exist2}:} Assume $y(t)$ is a solution of (\ref{volt1}). Writing (\ref{volt1}) in operator form we have
\begin{equation}
\label{volt2}
y(t)=y_0+J^\alpha z(t)
\end{equation}
and subsequently applying $D^\alpha_*$ to both sides yields
\begin{eqnarray*}
D^\alpha_*y(t)&=&D^\alpha_*y_0+D^\alpha_*J^\alpha z(t)\\
&=&0+z(t,y(t)).
\end{eqnarray*}
Thus, recalling that $z(t,y(t))=my(t)+f(t)$, $y(t)$ must also solve (\ref{eq1}). Proving the condition in the other direction, we now assume that $y(t)$ is a solution to (\ref{volt1}). Recalling that $z(t)\in C[0,B]$ we may write (\ref{eq1}) as
\begin{eqnarray*}
z(t,y(t))&=&D^\alpha_*y(t)\\
&=&D^\alpha\left( y-y_0\right) (t)\\
&=&DJ^{1-\alpha}y(t)-DJ^{1-\alpha}y_0
\end{eqnarray*}
Applying $J$ to both sides yields
\begin{eqnarray*}
Jz(t,y(t))&=&JDJ^{1-\alpha}y(t)-JDJ^{1-\alpha}y_0\\
&=&J^{1-\alpha}y(t)-J^{1-\alpha}y_0
\end{eqnarray*}
Now applying $D^{1-\alpha}$ to both sides we have
$$D^{1-\alpha}Jz(t,y(t))=y(t)-y_0,$$
which may be rearranged to give the precise Volterra equation we require:
$$y_0+J^{\alpha}z(t,y(t))=y(t).\square$$

{\bf Proof of Theorem \ref{exist1}:}
Suppose that $M=0$; then $z(t,y(t))=0\forall (t,y)\in G$. For this case, $y:[0,B]\rightarrow\R$ such that $y(t)=y_0$ is a solution of the initial value problem (\ref{eq1}).

Otherwise, supposing $M\neq 0$, we must use our Lemma \ref{exist2}, which asserts that (\ref{eq1} is equivalent to the Volterra equation (\ref{volt1}). We introduce the set $U:=\left\{ y\in C[0,B]:\left\| y-y_0\right\|_\infty\leq K\right\}$. $U$ is a closed, convex subset of the Banach space of all continuous functions on $[0,B]$, equipped with the Chebyshev norm. Hence, $U$ is a Banach space also. Since $y=y_0\in U$ also, we conclude that $U$ is non-empty. We define an operator $R$ on $U$ by
\begin{equation}
\label{thp1}
(Ry)(t):=y_0+\frac{1}{\Gamma(\alpha)}\int^t_0(t-s)^{\alpha-1}z(s,y(s))ds.
\end{equation}
We can now rewrite our Volterra equation (\ref{volt1}) as
\begin{equation}
\label{thp2}
y=Ry
\end{equation}
and our task of proving the existence of a solution to (\ref{volt1}) (which is equivalent to proving the existence of a solution to our original problem (\ref{eq1})) now becomes one of proving that the operator $R$ has a fixed point. Consider that, for $0\leq t_1\leq t_2\leq B$, we have
\begin{eqnarray*}
\left| (Ry)\left( t_1\right)-(Ry)\left( t_2\right)\right|&=&\frac{1}{\Gamma(\alpha)}\left|\int^{t_1}_0(t_1-s)^{\alpha-1}z(s,y(s))ds-\int^{t_2}_0(t_2-s)^{\alpha-1}z(s,y(s))ds\right|\\
&\leq&\frac{M}{\Gamma(\alpha)}\left(\int^{t_1}_0\left| (t_1-s)^{\alpha-1}-(t_2-s)^{\alpha-1}\right| ds+\int^{t_2}_{t_1}(t_2-s)^{\alpha-1}ds\right)\\
&&\,\,\,\,({\rm since}\,\,\alpha<1\Rightarrow\alpha-1<0\Rightarrow\left( t_1-s\right)^{\alpha-1}\geq\left( t_2-s\right)^{\alpha-1})\\
&\leq&\frac{2M}{\Gamma(\alpha+1)}\left( t_2-t_1\right)^\alpha\\
&\rightarrow&0 \,\,{\rm as} \,\, t_2\rightarrow t_1.
\end{eqnarray*}
In addition, for $y\in U$, $t\in[0,B]$ we have
\begin{eqnarray*}
\left| (Ry)(t)-y_0\right| &=& \frac{1}{\Gamma(\alpha)}\left|\int^t_0(t-s)^{\alpha-1}z(s,y(s))ds\right|\\
&\leq &\frac{MK\Gamma(\alpha+1)}{\Gamma(\alpha+1)M}\\
&=& K.
\end{eqnarray*}
Thus, $y\in U\Rightarrow Ry\in U$. In order to show that we have a fixed point we must now show that $R(U):=\left\{ R(u):u\in Y\right\}$ is a relatively compact set.
For $w\in R(U)$ we have, for all $\t\in[0,B]$,
\begin{eqnarray*}
|w(t)|&=&\left| (Ry)(t)\right|\\
&\leq &y_0+\frac{1}{\Gamma(\alpha)}\int^t_0(t-s)^{\alpha-1}|z(s,y(s))|ds\\
&\leq &y_0+K.
\end{eqnarray*}
If $|t_2-t_1|<\delta$ then $$|(Ry)(t_1)-(Ry)(t_2)|\leq\frac{2M\delta^\alpha}{\Gamma(\alpha+1)};$$
thus the set $R(U)$ is equicontinuous (due to the right hand side's indpendence of $y$, $t_2$, $t_1$). We can therefore apply the Arzel\`{a}-Ascoli Theorem \cite{Arz1} to conclude that $A(U)$ is a relatively compact set and consequently apply Schauder's Fixed Point Theorem \cite{Sch1} to conclude that $R$ has a fixed point. By construction, our fixed point solves the original problem (\ref{eq1}). All that remains is to prove that our solution (whose existence we have just proved) is unique.

Suppose we have a second solution (i.e. a second fixed point), $\tilde{y}$. If we consider $\left\| y-\tilde{y}\right\|_\infty$ then repeated use of $Ry=y$ and $R\tilde{y}=y$ yields $$\left\| R^jy-R^j\tilde{y}\right|_\infty\leq\frac{\left(|m|B^\alpha\right)^j}{\Gamma(1+\alpha j)}\left\| y-\tilde{y}\right\|_\infty .$$

To confirm uniqueness then, we just need to confirm convergence of the above; and what we have on the right hand side of the inequality is the power series definition of the Mittag-Leffler function $E_\alpha\left( |m|B^\alpha\right)$, which we know converges, by theorem \ref{thmitcon}. Our proof is therefore concluded.$\square$

\end{document}